\documentclass[a4paper,10pt,reqno]{amsart}
\usepackage[foot]{amsaddr}
\usepackage[english]{babel}
\usepackage[utf8]{inputenc}
\usepackage{graphicx}
\graphicspath{{figures/}}
\usepackage[colorlinks=true,linkcolor=black,urlcolor=black,citecolor=black]{hyperref}
\usepackage{amsmath}
\usepackage{amsfonts}
\usepackage{amssymb,amsmath}
\usepackage{amsthm}
\usepackage{fixmath}
\usepackage{tikz}
\usepackage{tikz-cd}
\usetikzlibrary{matrix}
\usetikzlibrary{arrows,arrows.meta}
\tikzcdset{arrows={line width=1pt},arrow style=tikz, diagrams={>=stealth}}
\tikzcdset{every label/.append style = {scale=1,yshift=0.2ex}, nodes={font=\Large}} 
\usepackage{subcaption}
\usepackage{enumerate}
\usepackage[shortlabels]{enumitem}
\usepackage{mathtools}
\usepackage{todonotes}
\numberwithin{equation}{section}

\newtheorem{theorem}{Theorem}[section]

\newtheorem{proposition}[theorem]{Proposition}

\newtheorem{remark}{Remark}[section]
\newtheorem{corollary}{Corollary}[section]

\newcommand{\x}{{{x}}}
\newcommand{\xx}{{y}}
\renewcommand{\d}{\mathop{}\!\mathrm{d}}
\newcommand{\D}{\Omega}
\newcommand{\DC}{\widetilde{\D}}
\newcommand{\DI}{\D_{I}}
\newcommand{\DD}{\D\cup\DI}
\newcommand{\R}{\mathbb{R}}

\newcommand{\NM}{\mathcal{N}}
\newcommand{\kernel}{\gamma}
\newcommand{\kerneld}{\gamma(\x,\xx)}
\newcommand{\V}{{V}}
\newcommand{\Va}{\V_A}
\newcommand{\Vao}{\V_{A_0}}
\newcommand{\Vb}{\V_B}
\newcommand{\Sh}{\widetilde{S}_h}
\newcommand{\Sho}{S_h}
\newcommand{\M}{M}
\renewcommand{\tt}{\tau}
\newcommand*{\norms}[1]{{\left|{#1}\right|}}
\newcommand*{\norm}[1]{{\left\lVert#1\right\rVert}}

\newcommand*{\duala}[1]{{\left\langle{#1}\right\rangle}_{\Va^\prime\times\Va}}
\newcommand*{\dualar}[1]{{\left\langle{#1}\right\rangle}_{\Va\times\Va^\prime}}
\newcommand*{\dual}[1]{{\left\langle{#1}\right\rangle}}
\newcommand{\cker}{c_{\kernel}}
\newcommand{\ckerr}{C_{\kernel}}
\newcommand{\ckerrgrad}{\hat{C}_{\kernel}}
\newcommand{\cpot}{c_{F}}
\newcommand{\param}{\xi}
\newcommand{\spann}{\mathrm{span}}
\newcommand{\G}{\mathcal{G}}
\newcommand{\inp}[1]{\overline{#1}_\tt}
\newcommand{\inpl}[1]{\underline{#1}_\tt}
\newcommand{\inpt}[1]{\hat{#1}_\tt}
\newcommand{\inppm}[1]{\overline{#1}_{\pm,\tt}}
\newcommand{\inpp}[1]{\overline{#1}_{+,\tt}}
\newcommand{\inpm}[1]{\overline{#1}_{-,\tt}}
\newcommand{\intm}[1]{\overline{#1}}
\newcommand{\Q}{Q}
\newcommand{\QC}{\widetilde{Q}}
\newcommand{\J}{J_k}
\DeclareMathOperator*{\esssup}{ess\,sup}

\makeatother

\begin{document}
\title[Nonlocal Cahn-Hilliard model permitting sharp interfaces]{On a nonlocal Cahn-Hilliard model\\ permitting sharp interfaces}
\author[O. Burkovska]{Olena Burkovska${}^{1}$}
\address{${}^1$ Computer Science and Mathematics Division, Oak Ridge National Laboratory, One Bethel Valley Road, TN 37831, USA}
\email{burkovskao@ornl.gov}
\author[M. Gunzburger]{Max Gunzburger${}^2$}
\address{${}^2$
Department of Scientific Computing, Florida State University, 400 
Dirac Science Library\\ Tallahasse, FL 32306-4120, USA\\
and the Oden Institute for Computer Engineering and Sciences, University of Texas at Austin\\ Austin, TX 78712, USA}
\email{mgunzburger@fsu.edu}

\thanks{This material is based upon work supported by the U.S. Department of Energy, Office of Advanced Scientific Computing Research, Applied Mathematics Program under the award numbers ERKJ345 and ERKJE45; and was performed at the Oak Ridge National Laboratory, which is managed by UT-Battelle, LLC under Contract No. De-AC05-00OR22725. 
U.S.\ Government retains a non-exclusive, royalty-free license to publish or reproduce
the published form of this contribution, or allow others to do so, for U.S.\ Government
purposes. 
MG was also supported by U.S. Department of Energy grant DE-SC0021077 as part of the AEOLUS Multifaceted Mathematics Integrated Capability Center.
}

\begin{abstract}
A nonlocal Cahn--Hilliard model with a nonsmooth potential of double-well obstacle type that promotes sharp interfaces in the solution is presented. To capture long-range interactions between particles, a nonlocal Ginzburg-Landau energy functional is defined which recovers the classical (local) model as the extent of nonlocal interactions vanish. In contrast to the local Cahn--Hilliard problem that always leads to diffuse interfaces, the proposed nonlocal model can lead to a strict separation into pure phases of the substance. Here, the lack of smoothness of the potential is essential to guarantee the aforementioned sharp-interface property. Mathematically, this introduces additional inequality constraints that, in a weak formulation, lead to a coupled system of variational inequalities which at each time instance can be restated as a constrained optimization problem. We prove the well posedness and regularity of the semi-discrete and continuous in time weak solutions, and derive the conditions under which  pure phases are admitted. Moreover, we develop discretizations of the problem based on finite element methods and implicit-explicit time stepping methods that can be realized efficiently. Finally, we illustrate our theoretical findings through several numerical experiments in one and two spatial dimensions that highlight the differences in features of local and nonlocal solutions and also the sharp interface properties of the nonlocal model.
\end{abstract}


\maketitle

\section{Introduction}

The Cahn-Hilliard model was proposed in~\cite{cahn1958} as the model to describe phase separation of a binary alloy. 
Since then, the model and its variants have been widely used in different areas of science such as, e.g., tumor growth, image segmentation and copolymer melts, cf.~\cite{bertozzi2007,fritz2019brinkman,garke2016,oden2010,Ohta1986,wise2008}.

Given a bounded domain $\D\subset\R^n$, $n\leq 3$, and a fixed final time $T>0$, the model is described by the coupled system of equations
\begin{align}
\begin{cases}
\partial_t u-\nabla\cdot(\sigma(u)\nabla w)=0,\\
w=-\varepsilon^2\upDelta u+F^{\prime}(u),
\end{cases}\quad\text{in }\;\; (0,T)\times\D,\label{eq:main}
\end{align}
{together with appropriate boundary conditions. Here,} $u$ is an {\it order parameter} taking the values in $[-1,1]$ and is related to the concentration of a substance, $w$ is a {\it chemical potential}, $F(u)$ is a {\it double-well potential}, $\sigma(\cdot)$ is {\it the mobility}, and $\varepsilon>0$ is the {\it interface parameter}, which is proportional to the thickness of the interface. The double-well potential promotes pure phases and attains its minimum close to pure phases, i.e., when $u=\pm 1$. Typically, a regular potential is employed in~\eqref{eq:main}, given by the fourth-order polynomial
\begin{align}
F(u):={\frac{\cpot}{4}}(u^2-1)^2,\quad{\cpot>0}.\label{potential_regular}
\end{align}
Being the simplest choice, the regular potential~\eqref{potential_regular} is not physically realistic, since it 
may provide a non-feasible solution $|u|>1$. In practice it serves as an approximation of the more complex, but physically more relevant, logarithmic potential {for $ u\in(-1,1)$}
\begin{align}
F(u)&:={\frac{\theta}{2}}((1+u)\log(1+u)+(1-u)\log(1-u))-{\frac{\cpot}{2}} u^2,\;\, {0<\theta<\cpot},\label{potential_log}
\end{align}
or the obstacle potential
\begin{align}
F(u)&:=\left.
\begin{cases}
F_0(u) & \text{if } |u|\leq 1\\
+\infty &\text{if } |u|> 1
\end{cases}
\right\} = F_0(u)+I_{[-1,1]}(u),\label{potential_obstacle}
\end{align}
where $I_{[-1,1]}$ is the convex indicator function of the admissible range $[-1,1]$ and $F_0(u)={\cpot}/{2}(1- u^2)$, $\cpot>0$. In contrast to~\eqref{potential_regular}, the logarithmic potential~\eqref{potential_log} and obstacle potential~\eqref{potential_obstacle} always provide a solution within an admissible range $|u|\leq 1$. However, the logarithmic potential does not allow $u$ to attain pure phases, i.e., $u\in (-1,1)$, whereas the obstacle potential promotes pure states in the model, i.e., $u=1$ or $u=-1$.

Mathematically, the Cahn-Hilliard model~\eqref{eq:main} can be derived as {the} $H^{-1}$-gradient flow of the Ginzburg-Landau energy 
\begin{equation}
\mathcal{E}(u)=\frac{\varepsilon^2}{2}\int_{\D}|\nabla u|^2\d\x+\int_{\D}F(u)\d\x.
\label{GL_local}
\end{equation}
The gradient square term in~\eqref{GL_local} represents short-range interactions between particles. To account for long-term interactions, {one can consider a nonlocal variant of the Cahn-Hilliard problem~\eqref{eq:main}:
\begin{align}
\begin{cases}
\partial_t u-\nabla\cdot(\sigma(u)\nabla w)=0,\\
w=Bu+F^{\prime}(u),
\end{cases}\quad\text{in }\;\; \D\times (0,T),\label{CH_nonlGL}
\end{align}
where $B$ is a nonlocal diffusion operator.}
In contrast to the derivation of the local counterpart~\eqref{eq:main}, {Giacomin and Lebowitz in~\cite{Giacomin1997}} (see also~\cite{Giacomin1998}) provide a rigorous microscopic derivation of the nonlocal model~\eqref{CH_nonlGL} {in the case when $\D\equiv\mathbb{T}^n$ is set to be the $n$-dimensional torus}. More specifically, the macroscopic continuum model is obtained via a hydrodynamic limit of particle models that are dynamic versions of lattice gases undergoing long interactions. 
{In this case the nonlocal operator $B=\mathcal{B}_{\mathbb{T}}$ is set to
\begin{equation*}
\mathcal{B}_{\mathbb{T}}u(\x)=\int_{\mathbb{T}^n}(u(\x)-u(\xx))\kernel(\x-\xx)\d\xx=\cker u(\x)-(\kernel*u)(\x),\quad\forall\x\in\mathbb{T}^n,
\end{equation*}
where} $\kernel\colon \mathbb{T}^n\to \R$ is {an integrable} convolution kernel that sets up the law for the nonlocal interactions. The kernel is assumed to depend only on the distance \(|x-y|_{\mathbb{T}^n}\) on the torus, the convolution is defined as $(\kernel*u)(\x)=\int_{\mathbb{T}^n}u(\xx)\kernel(\x-\xx)\d\xx$, and $\cker =\int_{\mathbb{T}^n}\kernel(\xx)\d\xx=(\kernel*1)(\x)$, which is independent of $\x$. The corresponding nonlocal free energy functional has the form
\begin{equation*}
\mathcal{E}(u)=\frac{1}{4}\int_{\mathbb{T}^n}\int_{\mathbb{T}^n}(u(\x)-u(\xx))^2\kernel(\x-\xx)\d\x\d\xx+\int_{\mathbb{T}^n}F(u)\d\x.
\end{equation*}
We note that the local model can be considered as an approximation of the nonlocal model for vanishing nonlocal interactions (cf., e.g.,~\cite{davoli2019convective,du2018CH} and references cited therein).
\begin{equation}
\cker u-\kernel*u 
\approx-\varepsilon^2\upDelta u\quad\text{with }\;\;\varepsilon^2=\frac{1}{{2n}}\int_{{\R^n}}\kernel(\xi)|\xi|^2\d\xi.\label{local_nonlocal_approx}
\end{equation}

An adoption of the Giacomin-Lebowitz model to the case when $\D\subset\R^n$ is a bounded domain and not a torus is possible by setting { $B=\mathcal{B}_{\mathcal{R}}$ in~\eqref{CH_nonlGL}}, where
\begin{equation}
\mathcal{B}_{\mathcal{R}}u(\x)=\int_{\D}(u(\x)-u(\xx))\kernel(\x-\xx)\d\xx=\cker(\x)u(\x)-(\kernel*u)(\x),\quad\forall\x\in\D. \label{nonl_oper_regional}
\end{equation}
Here, the kernel \(\gamma\colon \R^n \to \R\) is no longer defined on the torus (which corresponds to a periodic function on \(\R^n\)), but often depends on the distance in \(\R^n\), which implies that $c_\kernel(x) = (\kernel*1)(\x)=\int_{\D}\kernel(\x-\xx)\d\xx$ is no longer constant. The operator $\mathcal{B}_{\mathcal{R}}$ is often referred to as the \textit{regional} nonlocal operator, since it restricts the nonlocal interactions to the region $\D$, and has been {one of the most frequently} analyzed  in the context of nonlocal Cahn-Hilliard models; see, e.g.,~\cite{bates2005neumann,colli2012CH_NS,colli2007,davoli2019convective,gajewski2003,
gal_nonlocal_2017,gal2014longtime} and the references cited therein. While most of these works address integrable kernels, which is also a subject of the current work, singular kernels have also been investigated; see, e.g.,~\cite{coli2018,davoli2019}.
For an overview of early and recent references and extensions of the {local and} nonlocal Cahn-Hilliard model, we refer interested readers to~\cite{bates2006survey,miranville2017review,DuFeng2020}. 
In terms of the notation, the present modification has only minor differences to~\eqref{CH_nonlGL}. However, it has direct implications on the properties of the solutions, as will become apparent shortly. 

In this work, we are interested in the model~\eqref{CH_nonlGL} posed on a bounded domain~$\D$ with $\cker$ being constant throughout $\D$. For this, we consider {$B=\mathcal{B}_{\mathcal{N}}$ to be the following nonlocal operator:}
\begin{equation}
\mathcal{B}_{\mathcal{N}}u(\x)=\int_{\DD}(u(\x)-u(\xx))\kernel(\x-\xx)\d\xx=\cker u(\x)-(\kernel*u)(\x),\quad\forall\x\in\D,\label{nonl_oper_neumann}
\end{equation}
where $\cker=\int_{\DD}\kernel(\x-\xx)\d\xx$ for all $\x\in\D$ and $(\kernel*u)(\x)=\int_{\DD}u(\xx)\kernel(\x-\xx)\d\xx$. Here, $\DI$ is to as the {\em nonlocal interaction domain}, that represents the extent of nonlocal interactions outside of $\D$, and is defined as
\[
\DI:=\{\xx\in\R^n\setminus\D\colon \kerneld\neq 0,\;\; \x\in\D\}.
\]
To define~\eqref{nonl_oper_neumann}, one must know the values of $u$ also on the interaction domain $\DI$. Usually, the local Cahn--Hilliard model~\eqref{eq:main} is complemented with the homogeneous Neumann boundary condition (to guarantee the mass-conservation property), and in the nonlocal setting considered here we impose a homogeneous nonlocal flux condition for $u$ on $\DI$, which is analogous of the Neumann boundary conditions in the local setting:
\begin{equation}
\mathcal{N}u(\x)=\int_{\DD}(u(\x)-u(\xx))\kernel(\x-\xx)\d\xx =0 \quad\forall\x\in\DI.\label{eq:nonlocal_flux}
\end{equation}
The type of the nonlocal operator $\mathcal{B}_{\mathcal{N}}$~\eqref{nonl_oper_neumann} has gained significant attention recently in the context of various applications, cf.~\cite{DElia2020Acta,Dunonlocal2012,Du2019book} and references cited therein. In contrast to~\eqref{nonl_oper_regional}, the nonlocal interactions are not limited to $\D$ and occur in all of $\DD$. If  the kernel has infinite support then $\DD\equiv\R^n$, and the nonlocal interactions occur on the whole of $\R^n$.

\subsection{Model setting}

Finally, we present the model that encompasses both formulations. We denote $\DC$ to be either $\D$ or $\DD$, and we employ a non-smooth obstacle potential~\eqref{potential_obstacle}, and, for simplicity, consider  a constant mobility $\sigma(u)\equiv 1$. While the problem formulation~\eqref{CH_nonlGL} is appropriate for smooth potentials, the non-smooth obstacle potential~\eqref{potential_obstacle} {requires us to introduce} the concept of subdifferentials to define the derivative of $I_{[-1,1]}$. In this case, $F^\prime(u)$ should be replaced
by a generalized differential of $F$, $\partial F(u)=-\cpot u+\partial I_{[-1,1]}(u)$, where $\partial I_{[-1,1]}(u)$ is the subdifferential of the indicator function,
\begin{equation}
\partial I_{[-1,1]}(u)=\begin{cases}
(-\infty,0] & \text{if } u=-1,\\
0 &\text{for } u\in(-1,1),\\
[0,+\infty) &\text{if } u=1.
\end{cases}\label{subdiff}
\end{equation}
Finally, the nonlocal model we are interested to study becomes
\begin{align}
\begin{cases}
\partial_t u-\upDelta w=0,\\
w=\xi u-\kernel*u+\lambda,\quad \lambda\in\partial I_{[-1,1]}(u), 
\end{cases}\label{nonl_subdiff}
\end{align}
where $\xi(\x):=\int_{\DC}\kernel(\x-\xx)\d\xx-\cpot$ and $\kernel*u=\int_{\DC}u(\xx)\kernel(\x-\xx)\d\xx$. Here, we will already impose that $\xi(\x)\geq 0$ on $\D$, which is required to obtain a well-posed problem. We will show that $\xi$ is an appropriate interface parameter in the nonlocal model. We complement~\eqref{nonl_subdiff} with the local and nonlocal flux conditions
\begin{equation}
\frac{\partial w}{\partial n}=0\quad\text{on}\;\;\partial\D,\quad\text{and}\quad\mathcal{N}u(\x)=\int_{\DC}(u(\x)-u(\xx))\kernel(\x-\xx)\d\xx=0 \quad\text{on }\;\;\DC\setminus\D.\label{bc_nonlocal}
\end{equation}
We note that, in the variational form the problem~\eqref{nonl_subdiff} can be restated as a coupled system of equations involving variational inequalities. 
The corresponding nonlocal free energy is given as
\begin{align}
\mathcal{E}(u)=&\ \frac{1}{4}\int_{\DC}\int_{\DC}(u(\x)-u(\xx))^2\kernel(\x-\xx)\d\x\d\xx+\int_{\D}F(u)\d\x
\label{GL_nonl}
\end{align}
where $F$ is defined in~\eqref{potential_obstacle}.

\subsubsection*{Sharp interfaces}
The choice of the obstacle potential in~\eqref{nonl_subdiff} is not only physically motivated, but brings an interesting insight into the mathematical properties of the solution, such as the occurrence of sharp interfaces. 
In particular, as one of the main contributions of this work, we demonstrate that the solution of~\eqref{nonl_subdiff} is discontinuous and can admit sharp interfaces for some non-trivial and \textit{non-vanishing} nonlocal interactions. On the contrary, the local problem~\eqref{eq:main}, which is a diffuse interface model, does not {allow the appearance of} sharp interfaces, other than in the limiting case $\varepsilon\to 0$ which corresponds to \textit{vanishing} local interactions in~\eqref{GL_local}. 
In Figure~\ref{fig:intro} we illustrate the appearance of sharp interfaces for the local and nonlocal solutions of~\eqref{eq:main} and~\eqref{nonl_subdiff} with the obstacle potential~\eqref{potential_obstacle}. 
Here, we compare the solutions obtained with a nonlocal model~\eqref{nonl_subdiff} for $\xi\approx 0$ to the local model~\eqref{eq:main}, where we have replaced $\mathcal{B}_{\mathcal{N}}$ with the Neumann Laplacian $-\varepsilon^2\upDelta$ according to~\eqref{local_nonlocal_approx}. We point out that, in contrast to the local model, we do not need to perform the limit to obtain sharp interfaces in the nonlocal solution. Moreover, the sharp interface case $\xi=0$ for the nonlocal model is well-posed and can be solved numerically with the same time-dsicretization as for $\xi>0$, as we will discuss below.

\begin{figure}
\begin{subfigure}{.36\textwidth}
  \centering
  \includegraphics[width=\textwidth]{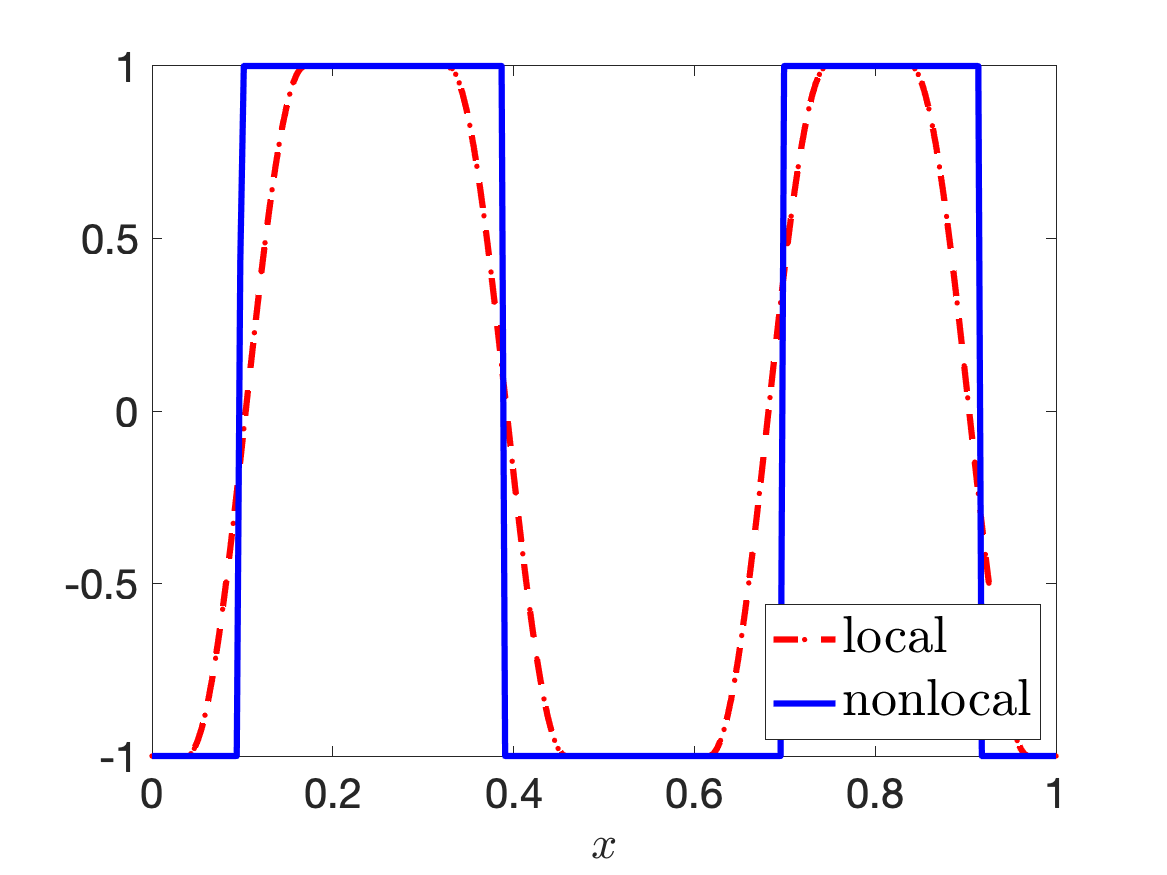}
  \caption{comparison in 1D}
 \end{subfigure}
 \begin{subfigure}{.3\textwidth}
  \centering
    \includegraphics[width=1.2\textwidth]{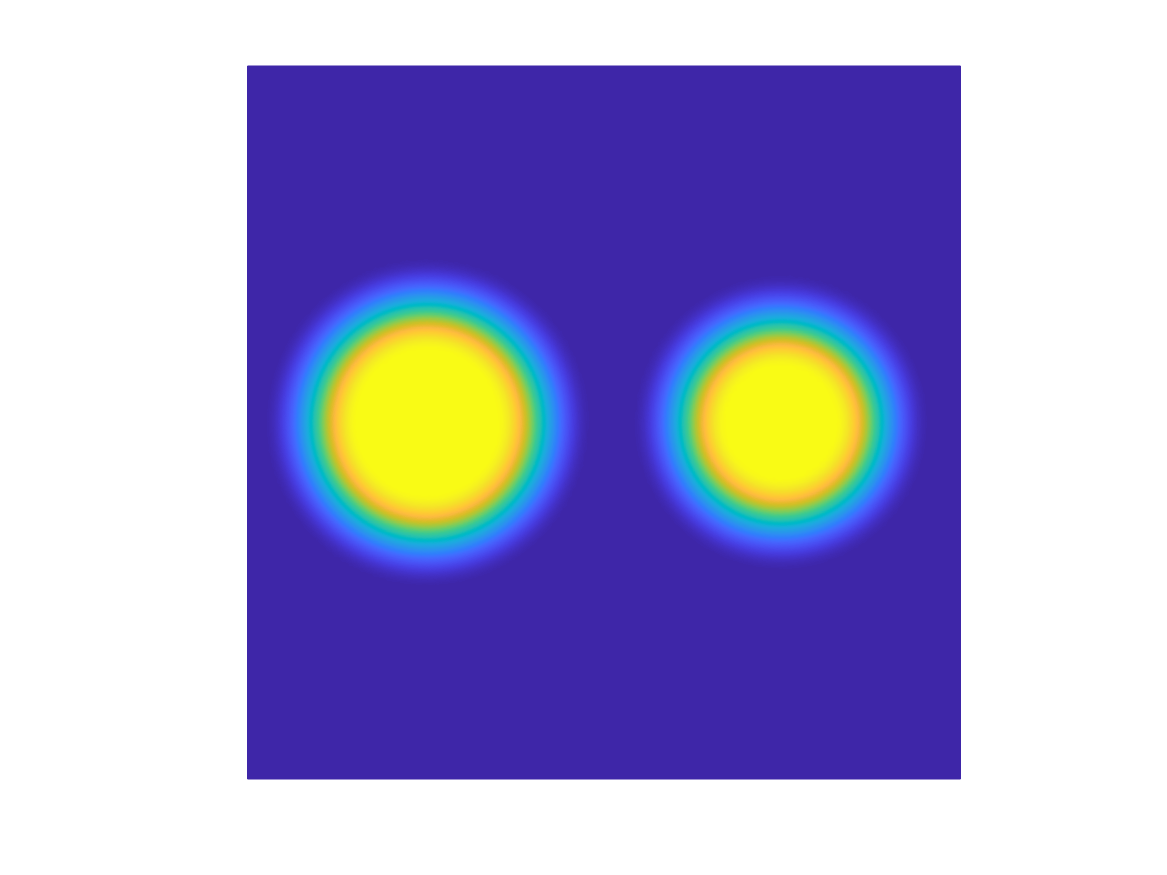}
  \caption{local solution (2D)}
 \end{subfigure}
 \begin{subfigure}{.3\textwidth}
  \centering
     \includegraphics[width=1.2\textwidth]{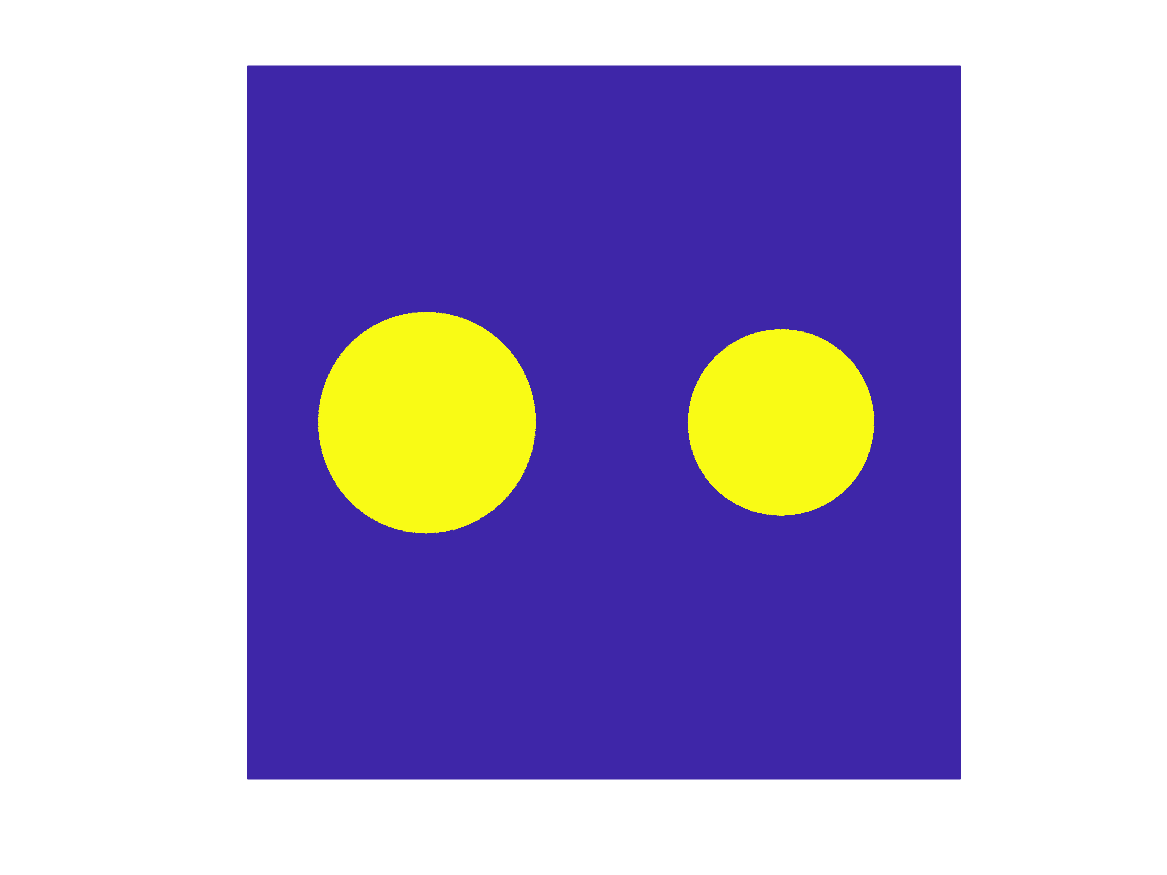}
  \caption{nonlocal solution (2D)}
 \end{subfigure}
 \caption{Comparison of the local and nonlocal solutions in one- and two-dimensions. The local and nonlocal solutions are initialized with the same initial condition, \(\cpot\), and plotted at the same time \(T\). For the nonlocal model we consider the nonlocal operator~\eqref{nonl_oper_neumann} with Gaussian type kernel~\eqref{kernel_gaussian}. The interface parameter \(\varepsilon\) is given by~\eqref{local_nonlocal_approx}, and with this choice the local model corresponds to the local limit of the nonlocal model for vanishing radius of {nonlocal} interactions.}\label{fig:intro}
 \end{figure} 

Mathematically, the appearance of sharp interfaces in the nonlocal model~\eqref{nonl_subdiff} is explained by realizing that the solution of~\eqref{nonl_subdiff} for $\xi>0$ can be obtained as a pointwise projection of $g=w+\kernel*u$ onto the admissible set,
\[
u(t) = P_{[-1,1]}\left\{\frac{1}{\xi}g(t)\right\}\quad\text{in }\;\;\D\times (0,T),\quad\xi>0.
\]
Note that, for $\xi=0$, the above turns into {(see also Theorem~\ref{thm:sharp_interfaces_t})}
\begin{equation*}
u(t) \in \operatorname{sign}\{g(t)\}{=\begin{cases}
\{1\}& \text{if } g(t)>0,\\
\{-1\}& \text{if } g(t)<0,\\
[-1,1]& \text{if } g(t)=0,\end{cases}}\quad\text{in }\;\;\D\times (0,T).\label{sharp_interf_intro}
\end{equation*}
{If} the level set of $\{g=0\}$ is of zero measure, this implies that the solution can admit only pure phases $u=1$ or $u=-1$. 
It becomes apparent that $\xi$ is not constant for the case of the regional operator~\eqref{nonl_oper_regional}, since $\cker$ is variable in $\D$. {Hence}, in general the corresponding solution does not posses sharp interfaces, unless the obstacle potential is equipped with a spatially-dependent $\cpot(\x)$.

\subsection{Contribution and related works}
One of the contributions of this work is to propose a well-posed nonlocal model~\eqref{nonl_subdiff} that admits discontinuous solutions with only pure phases, as described in the previous section. To the best of our knowledge this is the first result in this direction. However, we note that the jump-discontinuous behaviour of the solution closely resembles the bang-bang control principle, where almost all values of the control function lie on the boundary of the admissible set; see, e.g.,~\cite{troltzsch2010optimal}. Indeed, a single time-step of the model~\eqref{nonl_subdiff} can be also understood as a variational inequality, which serves as a necessary optimality condition for a specific bang-bang optimal control problem.

{Moreover, discontinuous solutions for related Allen-Cahn phase transition models have already been {addressed} in, e.g.,~\cite{bates1997,fife1997,BatesChmaj1999,Chen1997,ChmajRen1999}; see also~\cite{FifeReview2003} and the references cited therein. They are based on the same Ginzburg-Landau energy as in~\eqref{GL_nonl} and described by the evolution equation
\begin{align}
\partial_t u+\cker u-\kernel*u+F^{\prime}(u)=0.\label{eq:AC}
\end{align}
 In more recent works~\cite{DuYang2016,du2019,DuYang2017}, the authors also investigated numerically and theoretically the relation between kernel parameters with finite range of nonlocal interactions and the appearance of discontinuities in the steady state solution of the Allen-Cahn equation. In~\cite{DuYang2016} it was reported that the condition under which these solutions may admit discontinuities is related to the same parameter $\xi=\cker-\cpot$, defined for a smooth potential $F$ as in~\eqref{potential_regular} with $\cpot=-\min_{u\in[-1,1]}F^{\prime\prime}(u)=1$. More specifically, for a regular potential these discontinuities can occur for $\xi<0$, which is only possible for the Allen-Cahn equation. However, even though the solution contains a discontinuity, it does not admit only pure phases. That is, $u$ does not take values  only in the set $\{-1,1\}$ and it does not jump from $1$ to $-1$ at a single point. On the other hand, using an obstacle potential in the Allen-Cahn setting can permit exclusively pure phases in the steady-state solutions, $u=\pm 1$. However, the evolution law~\eqref{eq:AC} still does not allow for the discontinuities to migrate over time; cf.~\cite{fife1997,FifeReview2003}. This is due to the fact that the time derivative of $u$ takes values in $L^2(\D)$ for the evolution law~\eqref{eq:AC}.  
In contrast, we prove that the solution of the nonlocal Cahn-Hilliard model with the obstacle potential~\eqref{nonl_subdiff} can admit jump-discontinuities with only pure phases in the solution profile not only at the steady-state but also during the whole time evolution.}

In addition to studying the properties of the solution, we provide a detailed analysis of the nonlocal problem~\eqref{nonl_subdiff}. First, we establish general well-posedness results for an appropriate time-discrete formulation of~\eqref{nonl_subdiff} ({see} Theorem~\ref{thm:existence_semidiscrete}--\ref{thm:uniqueness_semidiscrete}), where only
integrability is required for the kernel. Moreover, we show that under certain conditions the solution is given by an $L^2$-projection on the admissible set. This representation allows us to deduce improved regularity properties of the solution ({see} Theorem~\ref{thm:reg_tk}) and the appearance of sharp-interfaces for $\xi=0$ {(see Theorem~\ref{thm:properties_tk})}. We derive a convergence result towards a continuous in time formulation. As a consequence we obtain well-posedness of the time continuous problem ({see} Theorem~\ref{thm:existence_t}--\ref{thm:uniqueness_t}), including an existence result for $\xi=0$. The latter is noteworthy since most of the available results focus on the settings equivalent to $\xi(\x)\geq {\xi}_{\min}>0$, which does not cover the more delicate case of sharp interfaces, cf.~\cite{davoli2019convective,gal_nonlocal_2017}.
One of the works we are aware of that covers the case $\xi=0$ is~\cite{colli2007}, where a model similar to~\eqref{nonl_subdiff} is studied in an abstract setting. However, there well-posedness is derived in the time-continuous framework by performing a limit of an appropriate regularized problem. In contrast, we do not regularize the non-smooth potential, but exploit its limited regularity to characterize the fine properties of the solution, including sharp interfaces. 
For the analysis, we rely solely on the time discretization, which also serves as a basis for the numerical realization introduced here, including the case $\xi=0$.

In addition to the complete analysis of the model, we also discuss efficient space and time approximations of the nonlocal solution based on the finite element method which we complement with corresponding numerical results. Numerical analysis of the local Cahn-Hilliard problem has been an active subject of research; see, e.g.,~\cite{blank2011,blowey_elliott1991,blowey_elliott_1992,bosch2014,brenner2018}. {The approximation of the nonlocal diffuse interface problem has been addressed in various works, e.g.,~\cite{ainsworth2017,akagi2016,du2018CH,du2019}. A design of appropriate time inegration schemes have been investigated in~\cite{GuanWW2014,GuanLWW2014,GuanLW2017,du2018CH,du2019}. The aforementioned works focused mainly on smooth potentials. }An inclusion of the obstacle potential in the model leads to a non-smooth and nonlinear system involving variational inequalities. The solvability of such a system can be realized by the semi-smooth Newton method~\cite{kunisch2002} which has been addressed in the {local} Cahn-Hilliard setting in, e.g.,~\cite{blank2011,hintermuller2011CH}, and which we adopt here {in the nonlocal setting}. 
Lastly, due to the complex structure of the Cahn-Hilliard system, in general, there is great demand for the development of efficient approximation techniques based on, e.g., Krylov solvers~\cite{bosch2014} or model order reduction methods~\cite{grassle2018,grassle2019}. 

The paper is organized as follows. In Section~\ref{sec:prelim} we introduce appropriate functional settings and the variational formulation of the problem we consider. Section~\ref{sec:semidiscrete} is devoted to the well-posedness analysis of the semi-discrete problem, improved regularity, and the derivation of the sharp-interface condition for the corresponding solution. The respective results for the time-continuous problem formulation are given in Section~\ref{sec:continuous}. Then, in Section~\ref{sec:numerics} we discuss finite-element approximations together with time-marching schemes for the nonlocal problem and provide various numerical examples. Finally, we conclude the paper in Section~\ref{sec:conclusion}.


\section{Preliminaries}\label{sec:prelim}
We denote by $L^p(\D)$, $p\in[1,\infty]$, the usual Lebesgue spaces and by $W^{k,p}(\D)$, $k\in\mathbb{N}$, $1\leq p\leq\infty$, we mean the Sobolev space of all functions in $L^p(\D)$ having all distributional derivatives up to order $k$ and endowed with the norm $\norm{u}^p_{W^{p,k}(\D)}=\sum_{|\alpha|\leq k}\norm{D^\alpha u}_{L^p(\D)}^p$. We denote by $L^p(0, T ; X)$, $1\leq p\leq \infty$ the Bochner space of all measurable functions $u:[0, T]\to X$, for which the norms 
\begin{align*}
\norm{u}_{L^p(0,T;X)}=\ \left(\int_{0}^T\norm{u(t)}^p_{X}\d t\right)^{1/p}, \;\;p<\infty,\;\;\text{and}\;\;
\norm{u}_{L^{\infty}(0,T:X)}=\esssup_{t\in (0,T)}\norm{u(t)}_{X}
\end{align*}
are finite. If not specified, we often use the relation $a\leq C b$, with some generic constant $C > 0$ that may be different in every instance, but is independent of $a$ and $b$.

\subsection{Nonlocal operators}
We introduce a kernel function $\kernel:\R^n\to\R^+$, which is radial, has finite support and integrable, i.e.,
\begin{equation}
\begin{cases}
\begin{aligned}
&\kernel(\x)=\hat{\kernel}(|\x|),\;\; \kernel\in L^1(\R^n),\\
&\text{with }\;{\rm supp}(\hat{\kernel})\subset(0,\delta],\;\delta>0,\\
&\text{and there exists $\sigma>0$, such that}\; (0,\sigma)\subset{\rm supp}(\hat{\kernel}).
\label{kernel_cond1}
\end{aligned}\tag{H1}
\end{cases}
\end{equation}
While most of the results in this paper are valid under the above conditions, in certain cases, we will need to require a higher regularity of the kernel, i.e., 
\begin{equation}
\kernel\in W^{1,1}(\R^n).
\label{kernel_cond2}\tag{H2}
\end{equation}
In the following we also employ (by a slight abuse of notation) the two-point version of $\kernel$ given by
\[
\kerneld=\kernel(\x-\xx),\quad\forall\x,\xx\in\R^n, 
\]
which will be used as an integration kernel to define the nonlocal operator $B$. 

We define the nonlocal operator $B$ that encompasses the definitions of  the ``Neumann'' ({\it Case~1})~\eqref{nonl_oper_neumann} and ``regional''  ({\it Case~2})~\eqref{nonl_oper_regional} type of nonlocal operators
\begin{align}
Bu(\x)=\int_{\DC}(u(\x)-u(\xx))\kerneld\d\xx=\cker(\x) u(\x)- (\kernel*u)(\x),
\label{operatorB}
\end{align}
where $\DC=\D$ (``regional'') or $\DC=\DD$ (``Neumann''). For convenience of notation we suppose that $u$ is extended by zero outside of $\DC$, and $(\kernel*u)(\x)$ denotes the convolution on $\DC$ of $\kernel$ with $u$, and we define
\begin{align}
\cker(\x):=\int_{\DC}\kerneld\d\xx,\quad \ckerr:=\int_{\R^n}\kerneld\d\xx, \quad\text{ and }\;\;\ckerrgrad:=\norm{\nabla\kernel}_{L^1(\R^n)},\label{cker}
\end{align}
and note that $0\leq\cker(\x)\leq\ckerr<\infty$ for all $\x\in\D$. Since $\kernel$ has finite support, $\DC$ is a bounded domain, and for $\DC=\DD$, $\cker$ is a constant for all $\x\in\D$ and can be computed explicitly as $\cker=\frac{2\pi^{n/2}}{\Gamma(n/2)}\int_{0}^\delta|\xi|^{n-1}\hat{\kernel}(|\xi|)\d\xi$.

\subsection*{Nonlocal Neumann-type boundary conditions}
The problem with the ``Neumann'' type nonlocal operator for which $\DC=\DD$ and
 $B=\mathcal{B}_{\mathcal{N}}$~(see \eqref{nonl_oper_neumann}) implies that there is a nonlocal flux from $\DI$ to $\D$ prescribed in the model. From a probabilistic perspective in the context of particle-based models such boundary conditions provide information about what happens to a particle upon leaving a domain $\D$, whereas 
 $u(x,t)$ in this case represents the probability density function of the position of a particle subject to a jump-diffusion process.   
Such boundary conditions have been proposed in~\cite{Dunonlocal2012}. A characterization of this condition  in terms of L\'evy flights for unsteady fractional Laplace equation has been studied in~\cite{deng_boundary_2018}. Additional discussions about Neumann type boundary conditions and their variation in the nonlocal setting can be found in, e.g.,~\cite{Cortazar2007,cortazar2008approximate,MengeshaDu2015,
dipierro2017,tao2017nonlocal,DeliaTianYu2020,YouLuTraskYu2021}; see also~\cite{DElia2020Acta} and the references therein. In contrast, for the ``regional'' nonlocal operator $B=\mathcal{B}_{\mathcal{R}}$~\eqref{nonl_oper_regional} we have that $\DC=\D$ and $\DC\setminus\D=\emptyset$, which implies that there is no nonlocal interactions between $\D$ and $\mathbb{R}^n\setminus\D$, and, hence, no nonlocal flux condition such as~\eqref{eq:nonlocal_flux} needs to be imposed. In this case, from the probabilistic point of view the random process of the movement of the particle occurs solely whithin $\D$, and while in the ``Neumann'' case the particle is not allowed to leave $\DD$, here, 
the particle is not allowed to jump out of $\D$; this is referred to as a censored process.

\subsection*{Function spaces}
We set $\Va:=H^1(\D)$, which is endowed with the usual $H^1$-norm, $\norm{v}^2_{\Va}:=\norm{v}^2_{H^1(\D)}=|v|^2_{H^1(\D)}+\norm{v}^2_{L^2(\D)}$, and $|v|_{\Va}:=|v|_{H^1(\D)}=\norm{\nabla v}_{L^2(\D)}$. We also define the space of mean-free functions
\[\Vao:=\{v\in\Va\colon (v,1)_{L^2(\D)}=0\},\]
endowed with the same semi-norm as on $\Va$. Thanks to the local Poincar\'e inequality 
\begin{equation}
\norm{u}_{L^2(\D)}\leq\widetilde{C}_P\norms{u}_{\Va},\quad\forall u\in\Vao,\quad \widetilde{C}_P>0,\label{poincare_local}
\end{equation}
the seminorm on $\Vao$ also defines a norm. We denote by $\Va^\prime=(H^1(\D))^\prime$ the dual space of $\Va$, and by $\duala{\cdot,\cdot}$ a corresponding duality pairing, and for $f\in\Va^\prime$, we denote
\[
\norm{f}_{\Va^\prime}:=\sup_{\substack{v\in\Va \\ \norm{v}_{\Va}\neq0}}\frac{\duala{f,v}}{\norm{v}_{\Va}}. 
\]
We introduce the dual space of $\Vao$, denoted by $\Vao^{\prime}=(\Vao)^{\prime}$, which consists of functionals with mean-value zero:
\begin{equation*}
\Vao^{\prime}:=\left\{f\in\Va^{\prime}\colon \duala{f,1}=0\right\}.
\end{equation*}
We note that for any $f\in\Vao^\prime$, we have that $\norm{f}_{\Vao^\prime}=\norm{f}_{\Va^\prime}$. Indeed, for $v\in\Va$, such that $\norm{v}_{\Va}\neq 0$, we obtain
\begin{align*}
\norm{f}_{\Va^\prime}=\sup_{v\in\Va}\frac{|\duala{f,v}|}{\norm{v}_{\Va}}
=\sup_{v\in\Vao,\; M} \frac{|\duala{f,v+M}|}{(\norm{v}^2_{\Va}+M^2)^{1/2}}=\sup_{v\in\Vao}\frac{|\duala{f,v}|}{\norm{v}_{\Va}}.
\end{align*}

Let $\G$ denote a Green's operator for the inverse of the Laplacian with zero Neumann boundary conditions. That is, given $f\in\Va^\prime$, we define $\G f\in\Vao$ as the unique solution of 
\begin{align}
(\nabla\G f,\nabla v)_{L^2(\D)}&=\duala{f,v},\quad\forall v\in\Vao.
\label{greens_problem}
\end{align}
We note that, if $f\in\Vao^\prime$, then~\eqref{greens_problem} also holds true for all $v\in\Va$:
\begin{align*}
(\nabla\G f,\nabla v)_{L^2(\D)}&=\duala{f,v},\quad\forall v\in\Va.
\end{align*}
The existence and uniqueness of $\G f$ follows from the Poincar\'e inequality and the Lax-Milgram lemma, and it holds that
\begin{equation}
\norm{f}^2_{\Va^\prime}=\norms{\G f}^2_{\Va}=\norm{\nabla\G f}^2_{L^2(\D)}=\duala{f,\G f}.
\label{greens_prop1}
\end{equation}
If $f\in\Va^\prime\cap L^2(\D)$, then $\norm{f}^2_{\Va^\prime}=(\G f,f)_{L^2(\D)}$. Moreover, owing to~\eqref{greens_prop1}, it also holds that
\begin{equation}
\duala{f^\prime(t),\G f(t)}=\frac{1}{2}\frac{\d}{\d t}\norm{f(t)}^2_{\Va^\prime},\;\;\text{for a.e. } t\in(0,T), \;\;\forall f\in H^1(0,T;\Va^\prime).
\label{greens_prop2}
\end{equation}

For the nonlocal space $\Vb$ we distinguish between two cases:
\begin{equation*}
\Vb:=\begin{cases}
\left\{v\in L^2(\DD)\colon\NM v=0\;\text{on }\DI\right\},\quad&\text{for  \textit{Case~1}},\\
L^2(\D),&\text{for  \textit{Case~2}},
\end{cases}
\end{equation*}
where an inclusion of the nonlocal flux conditions in the function space setting in {\it Case~1} is important as it will become apparent {shortly}. 
We recall the generalized \textit{nonlocal Green's first identity}~\cite{Dunonlocal2012}:
\begin{multline}
(Bu,v)_{L^2(\D)}=\int_{\D}v(\x) \int_{\DC}(u(\x)-u(\xx))\kerneld\d\xx\\
=\frac{1}{2}\int_{\DC}\int_{\DC}(u(\x)-u(\xx))(v(\x)-v(\xx))\kerneld\d\xx\d\x+\int_{\DC\setminus\D}v(\x)\NM u(\x)\d\x.\label{Greens}
\end{multline}
Next, we define the inner product and semi-norm on $\Vb$:
\begin{align*}
(u,v)_{\Vb}&:=\frac{1}{2} \int_{\DC}\int_{\DC}(u(\x)-u(\xx))(v(\x)-v(\xx))\kerneld\d\x\d\xx,\quad \norms{v}^2_{\Vb}:=(v,v)_{\Vb},\
\end{align*}
and endow $\Vb$ with the norm $\norm{v}_{\Vb}^2:=\norms{v}_{\Vb}^2+\norm{v}^2_{L^2(\D)}$.
We also define a bilinear form:
\begin{equation}
b(u,v):=(u,v)_{\Vb}.
\label{bilform_b}
\end{equation}
\begin{proposition}
For all $u,v\in\Vb$ we obtain that 
\begin{equation}
b(u,v)=(\cker u,v)_{L^2(\D)}-(\kernel*u,v)_{L^2(\D)},\label{eq:split_b}
\end{equation}
where $\cker$ is defined in~\eqref{cker}.
\end{proposition}
\begin{proof}
{The proof follows directly from~\eqref{eq:split_b},~\eqref{operatorB} and an identity~\eqref{Greens}.}
\end{proof}
We recall the Young's inequality for products: 
\begin{equation}
ab\leq \frac{1}{2\epsilon}a^2+\frac{\epsilon}{2}b^2,\quad\forall a,b\geq 0,\quad\epsilon>0,
\label{youngs}
\end{equation}
and also make use of the Young's inequality for the convolution: letting $f\in L^p(\R^n)$ and $g\in L^q(\R^n)$ with $1\leq p\leq\infty$, $1\leq q\leq\infty$, and ${1}/{r}={1}/{p}+{1}/{q}-1\geq 0$, then $f*g\in L^r(\R^n)$ and
\begin{equation}
\norm{f*g}_{L^r(\R^n)}\leq\norm{f}_{L^p(\R^n)}\norm{g}_{L^q(\R^n)}.
\label{youngs_conv}
\end{equation}

\subsection{Exterior problem}
We note that for {\it Case~1} the nonlocal Neumann type volume constraints could be considered as an ``exterior" nonlocal problem posed on $\DI$ with non-homogeneous Dirichlet-type volume constraints imposed on $\D$. In particular, we consider the  problem 
\begin{equation}
\begin{aligned}
\NM u(\x) &= 0\quad\text{on }\DI,\\
u &= g\quad\text{on } \D.\label{exterior_strong}
\end{aligned}
\end{equation}
We define a nonlocal space $\Vb^0$ incorporating homogeneous volume constraints as 
\[
\Vb^0:=\big\{v\in L^2(\DD)\colon b(v,v)<\infty,\;\; v=0\;\; \text{on }\; \D\big\},\quad \norms{v}^2_{\Vb^0}:=b(v,v),
\]
where $b(\cdot,\cdot)$ is defined as in~\eqref{bilform_b} with $\DC=\DD$. 
For $\hat{u}:=u-g\in\Vb^0$, where $g$ is extended by zero outside of $\D$, and, by a slight abuse of notation, denoted by the same symbol,
we consider the corresponding weak formulation of the above problem:
\begin{equation}
b(\hat{u},v)=-(\NM g,v)_{L^2(\DI)}\quad\forall v\in\Vb^0.
\label{exterior_form}
\end{equation}
We recall a nonlocal Poincar\'e-type inequality (see, e.g.,~\cite[Proposition~1]{mengeshadu2013}):
\begin{equation}
\norm{u}_{L^2(\DI)}\leq C_P\norms{u}_{\Vb^0},\quad\forall u\in\Vb^0, \label{exterior_poincare}
\end{equation}
with $C_P=C_P(\kernel,\DC)$, and note that $b(\cdot,\cdot)$ is continuous and coercive on~$\Vb^0$
\[
b(u,u)\geq C_P^{-2}\norm{u}^2_{L^2(\DI)},\quad
|b(u,v)|\leq 4\ckerr\norm{u}_{L^2(\DI)}\norm{v}_{L^2(\DI)},\quad\forall u,v\in\Vb^0.
\]
Then, by the Lax-Milgram argument there exist a unique solution of~\eqref{exterior_form} with
\begin{equation}
\norm{u}_{L^2(\DI)}= \norm{\hat{u}}_{L^2(\DI)}\leq {C_P^2}\ckerr\norm{g}_{L^2(\D)},
\label{exterior_bdd}
\end{equation}
where the last inequality above has been obtained by using the definition of $\NM$,~\eqref{bc_nonlocal}, Young's inequality and the fact that $g=0$ on $\R^n\setminus\D$. 
Next, we show that, for both cases {\it Case~1} and {\it Case~2}, $\Vb\cong L^2(\D)$.
\begin{proposition}
The space $\left(\Vb,\norm{\cdot}_{\Vb}\right)$ is a Hilbert space that is equivalent to~$L^2(\D)$.
\end{proposition}
\begin{proof}
To show that $\Vb$ is equivalent to $L^2(\D)$ (which immediately implies that $\Vb$ is a Hilbert space) it is suffices to show the norm equivalence between these spaces. First, it is clear that $\norm{v}_{L^2(\D)}\leq\norm{v}_{\Vb}$. On the other hand, for all $u\in\Vb$ using~\eqref{eq:split_b} and Young's inequality~\eqref{youngs_conv} we obtain
\begin{multline*}
\norm{u}^2_{\Vb}=\norms{u}_{\Vb}+\norm{u}_{L^2(\D)}=(\cker u,u)_{L^2(\D)}-(\kernel*u,u)_{L^2(\D)}+\norm{u}^2_{L^2(\D)}\\
\leq (\ckerr+1)\norm{u}^2_{L^2(\D)}+\ckerr\norm{u}_{L^2(\DC)}\norm{u}_{L^2(\D)}\\
\leq\left((2\ckerr+1)\norm{u}_{L^2(\D)}+\ckerr\norm{u}_{L^2(\DC\setminus\D)}\right)\norm{u}_{L^2(\D)}\leq C\norm{u}_{\Vb}\norm{u}_{L^2(\D)},
\end{multline*}
where in the last inequality we exploit the fact that for {\it Case~2}, $\DC\setminus\D=\emptyset$, and for {\it Case~1}, $u\in\Vb$ is a solution of the exterior problem~\eqref{exterior_strong} and~\eqref{exterior_bdd} holds true. Then, dividing by $\norm{u}_{\Vb}$ concludes the proof. 
\end{proof}
We denote by $\Vb^\prime$ the dual space of $\Vb$, $\Vb^\prime=(\Vb)^\prime$. Clearly, with the previous result at hand, we obtain that the embedding 
\[\Vb\hookrightarrow L^2(\D)\hookrightarrow\Vb^{\prime}
\]
forms a Gelfand triple, i.e., it is continuous and dense. 

\subsection{Variational formulation}
Now, we are ready to present a weak formulation of the problem~\eqref{nonl_subdiff}. To incorporate inequality constraints $|u|\leq 1$ in the variational formulation, we introduce the set
\begin{equation}
\mathcal{K}:=\{v\in\Vb\colon |v|\leq 1\;\; \text{a.\ e. in}\;\; \D\},\label{eq:set_K}
\end{equation}
and the set $\mathcal{K}_m:=\{v\in\mathcal{K}\colon (v,1)_{L^2(\D)}=m\}$ of functions with mass $m\in(-|\D|,|\D|)$, {where $|\D|$ is the volume of $\D$}. We also define a positive cone 
\[
\M=\{v\in \Vb^\prime\colon v\geq 0\;\text{a.\ e. on}\;\D\}.
\]
We introduce space-time cylinders $\Q:=\D\times(0,T)$ and $\QC:=\DC\times(0,T)$, $T>0$.
Then, we consider the following problem: Find $u(t)\in\mathcal{K}$, $w(t)\in\Va$ such that 
\begin{equation}
\begin{aligned}
\duala{\frac{\partial u(t)}{\partial t},\phi}+(\nabla w(t),\nabla\phi)_{L^2(\D)}&=0,\quad\forall\phi\in\Va,\\
 b(u(t),\psi-u(t))+(F^{\prime}_0(u(t))-w(t),\psi-u(t))_{L^2(\D)}&\geq 0,\quad\forall\psi\in\mathcal{K},
\end{aligned}
 \label{CH_VI}\tag{$\mathcal{P}$}
\end{equation}
subject to $u(0)=u_0\in\mathcal{K}_m$.
In many instances, it is convenient to work with an equivalent formulation, where we eliminate the first equation in~\eqref{CH_VI} by setting $w(t)=-\G(\partial_t u(t))-\mu(t)$, where $\G$ is a Green's operator defined in~\eqref{greens_problem} and {$\mu(t)=-\frac{1}{|\D|}\int_\D w(t)\d\x\in\R$ is the negative mean-value of $w(t)$}.  
In addition, to incorporate the inequality constraints~\eqref{eq:set_K}, we introduce additional Lagrange multipliers $\lambda_{\pm}(t)\in L^2(\D)$. Then, we arrive at the following problem having a saddle point structure: Find $u(t)\in\Vb$, $\lambda(t)\in\Vb^\prime$, $\lambda(t):=\lambda_{+}(t)-\lambda_{-}(t)$, $\lambda_{\pm}(t)\in\M$, $\mu(t)\in\R$ such that 
\begin{equation}
\begin{aligned}
(\G(\partial_t u(t))+Bu(t)+F^{\prime}_0(u(t))+\lambda(t)+\mu(t),\psi)_{L^2(\D)}&=0,\;\;\forall\psi\in\Vb,\\
(\eta-\lambda_{\pm}(t),1\mp u)_{L^2(\D)}&\geq 0,\;\;\forall\eta\in\M.
\end{aligned}\label{CH_VIG}
\end{equation}
We note that the inequalities in~\eqref{CH_VIG}, which refer to complementarity conditions, can be simplified further to
\begin{equation}
(\eta,u(t)\pm 1)_{L^2(\D)}\geq 0,\quad (\lambda_{\pm}(t),u(t)\mp 1)_{L^2(\D)}=0,\quad\forall\eta\in\M.
\label{CH_VIG_complem}
\end{equation}
The above equations are obtained by simply taking $\eta=2\lambda_{\pm}(t)$ and $\eta =0$ in~\eqref{CH_VIG}.

We note that the problem~\eqref{CH_VIG} is also equivalent to its time-integrated version: Find $u\in L^2(0,T;\mathcal{K})$ such that
\begin{equation}
\begin{aligned}
(\G(\partial_t u)+Bu+F^{\prime}_0(u)+\lambda+\mu,\psi)_{L^2(\Q)} =0,\;\;\forall\psi\in L^2(0,T;\Vb),\\
(\eta-\lambda_{\pm},1\mp u)_{L^2(\Q)}\leq 0,\;\;\forall\eta\in L^2(0,T;\M),
\end{aligned}\label{CH_VIGt}
\end{equation}
where $L^2(0,T;M):=\{v\in L^2(0,T;\Vb^\prime)\colon v\geq 0\;\text{a.\ e. in}\; \Q\}$ and $L^2(0,T;\mathcal{K}):=\{v\in L^2(0,T;\Vb)\colon |v|\leq 1\;\text{a.\ e. in}\; \Q\}.$

\section{Semi-discrete problem: Existence, uniqueness and a sharp interface condition}\label{sec:semidiscrete}
In this section we present the semi-discretization of the problem~\eqref{CH_VI} and establish existence, uniqueness and the main properties of the solution. In the course of our analysis we introduce and make use of different equivalent characterizations of the problem. 

For $T>0$ and $K\in\mathbb{N}$, we define $\tt=T/K$ and $u^k:=u(t^k)$, $w^k:=w(t^k)$ for $k=1,\dots,K$. 
Following~\cite{blowey_elliott_1992} we consider an implicit time stepping scheme. Given $u^0\in\mathcal{K}_m$, we consider for $k=1,\dots,K$ the following problem: Find $(u^k,w^k)\in\mathcal{K}\times\Va$ such that
\begin{equation}
\begin{aligned}
\frac{1}{\tt}(u^k-u^{k-1},\phi)_{L^2(\D)}+(\nabla w^k,\nabla\phi)_{L^2(\D)}&=0,\quad\forall\phi\in\Va,\\
b(u^k,\psi-u^k)+(F^{\prime}_0(u^{k})-w^k,\psi-u^k)_{L^2(\D)}&\geq 0,\quad\forall\psi\in\mathcal{K}.
\end{aligned}\tag{$\mathcal{P}^k$}\label{CH_VI_tk}\end{equation}

We point out that the above scheme is mass conserving. Indeed, taking $\phi=1$ in~\eqref{CH_VI_tk} we obtain that $(u^{k-1},1)=(u^{k},1)=m$. 

If $u^k\in\mathcal{K}_m$ is a solution of~\eqref{CH_VI_tk}, then setting $w^k=-\G ({1}/{\tt}(u^k-u^{k-1}))-\mu^k$, $\mu^k\in\R$, and restricting the test functions to $\psi\in\mathcal{K}_m$, we can restate the coupled system~\eqref{CH_VI_tk} as a single variational inequality: Given $u^0\in\mathcal{K}_m$, find $u^k\in\mathcal{K}_m$ such that for $k=1,\dots,K$ it holds that
\begin{equation}
\frac{1}{{\tt}}\left(\G\left(u^k-u^{k-1}\right),\psi-u^k\right)_{L^2(\D)}+b(u^k,\psi-u^k)+(F^{\prime}_0(u^{k}),\psi-u^k)_{L^2(\D)}\geq 0,\;\;\forall\psi\in\mathcal{K}_m.
\tag{$\mathcal{Q}^k$}\label{CH_VI_greens}
\end{equation}
Furthermore, the above problem~\eqref{CH_VI_greens} can be equivalently posed as a system of complementarity conditions: For a given $u^0\in\mathcal{K}_m$, find $(u^k,\lambda_{\pm}^k,\mu^k)\in\Vb\times\M\times\R$ such that for $k=1,\dots,K$, it holds that
\begin{equation}
\begin{aligned}
\left(\frac{1}{{\tt}}\G(u^k-u^{k-1})+Bu^k+F^{\prime}_0(u^k)+\lambda^k+\mu^k,\psi\right)_{L^2(\D)}&=0,\;\;\forall\psi\in\Vb,\\
(\eta-\lambda_{\pm}^k,1\mp u^k)_{L^2(\D)}&\geq 0,\;\;\forall\eta\in\M.
\end{aligned}\label{CH_VIG_tk}
\end{equation}
The variational inequality has a direct connection to a minimization problem.

\subsubsection*{Minimization problem}
For $u\in\Vb$ we define the functional
\begin{align*}
\J(u):=&\frac{1}{2}\norms{u}_{\Vb}^2+\int_{\D}F_0(u)\d\x+\frac{1}{2\tt}\norm{\nabla\G(u-u^{k-1})}^2_{L^2(\D)}\\
=&\frac{1}{2}(\xi u,u)_{L^2(\D)}-\frac{1}{2}(\kernel*u,u)_{L^2(\D)}+\frac{\cpot}{2}|\D|+\frac{1}{2\tt}\norm{u-u^{k-1}}^2_{\Va^\prime},
\end{align*}
where $\xi(\x):=\cker(\x)-\cpot$.
Then, we consider the following constrained minimization problem:
\begin{equation}
\min_{u\in\mathcal{K}} \J(u)
 \quad\text{subject to}\int_{\D}u\d\x=m,\quad m\in (-|\D|,|\D|).
\tag{$\mathcal{J}^k$}\label{min_problem}
\end{equation}

\begin{proposition}\label{prop:JQ}
If $u^k\in\mathcal{K}_m$ is a solution of~\eqref{min_problem}, then it solves the variational inequality~\eqref{CH_VI_greens}. Conversely, if $u^k\in\mathcal{K}_m$  is a solution of~\eqref{CH_VI_greens} and the conditions on the time step $\tt$ given in Theorem~\ref{thm:uniqueness_semidiscrete} hold true, then $u^k\in\mathcal{K}_m$ is also a solution to~\eqref{min_problem}. 
\end{proposition}
\begin{proof}
Since $\J$ is convex and G\^ateaux differentiable in $\Vb$ (see Theorem~\ref{thm:uniqueness_semidiscrete}), it follows that the variational inequality~\eqref{CH_VI_greens} is a necessary and sufficient first-order optimality condition for the minimization problem~\eqref{min_problem}; see~\cite[Lemma~2.21]{troltzsch2010optimal}.
\end{proof}
Next, we state a useful property for the function belonging to the set $\mathcal{K}$, which will be used to prove the existence result in the next section.
\begin{proposition}[Maximum principle]\label{prop:max_principle}
Let $u\in\mathcal{K}$, then it follows that $|u|\leq 1$ in $\DC$, and hence $u\in L^{\infty}(\DC)$.
\end{proposition}
\begin{proof}
The proof follows the same lines as in the proof of Theorem~4.4 in~\cite{BurkovskaGunzburger2019}. 
\end{proof}

\subsection{Existence and uniqueness}
To show the existence and uniqueness of the semi-discrete problem~\eqref{CH_VI_tk}, we first discuss the existence and uniqueness of the solution of~\eqref{CH_VI_greens}, and then demonstrate the equivalence of problem formulations~\eqref{CH_VI_tk} and~\eqref{CH_VI_greens}. 

We collect in Proposition~\ref{prop:kernel_prop} some properties of the kernel which will be useful later.
 
\begin{proposition}\label{prop:kernel_prop}
For the kernel $\kernel$ satisfying~\eqref{kernel_cond1} the following properties hold true.
\begin{enumerate}
\item[(i)] For $\eta>0$ there exists a family of functions $\kernel_\eta:\R^n\to\R^+$ satisfying~\eqref{kernel_cond1}--\eqref{kernel_cond2}, such that
\begin{equation}
\norm{\nabla\kernel_\eta}_{L^{1}(\R^n)}\leq C_\eta,\quad\text{and}\quad\norm{\kernel-\kernel_\eta}_{L^1(\R^n)}\leq\eta,\quad C_\eta>0.
\end{equation}
\item[(ii)] The sequence $\kernel_\eta*u\to\kernel*u$ converges uniformly in $L^{\infty}(\DC)$ for any $u\in L^{\infty}(\DC)$, and the limiting function is continuous
\begin{equation}
\kernel*u\in C(\DC),\quad\forall u\in L^{\infty}(\DC).
\end{equation}
\item[(iii)] If, in addition,  $\kernel$ satisfies~\eqref{kernel_cond2}, then 
\begin{equation}
\kernel*u\in W^{1,\infty}(\R^n),\quad\forall u\in L^{\infty}(\R^n).
\end{equation}
\end{enumerate}
\end{proposition}
\begin{proof}
The conditions $(i)$ and $(iii)$ follow directly by density argument and Young's inequality~\eqref{youngs_conv}, respectively. Then, using the properties $(i)$, $(iii)$, and the fact that $\kernel*u=\kernel_\eta*u-(\kernel-\kernel_\eta)*u$, we obtain that $\kernel_\eta*u\in W^{1,\infty}(\R^n)$ for any $u\in L^{\infty}(\R^n)$, and
\[
\norm{\kernel*u-\kernel_\eta*u}_{L^{\infty}(\DC)}\leq\norm{\kernel-\kernel_\eta}_{L^1(\R^n)}\norm{u}_{L^{\infty}(\DC)}\leq\eta\norm{u}_{L^{\infty}(\DC)}.
\]
Now, letting $\eta\to 0$ we obtain $(ii)$ and conclude the proof. 
\end{proof}
Next, we establish the existence and uniqueness of the solution of the semi-discrete problem~\eqref{CH_VI_greens}.
\begin{theorem}[Existence]\label{thm:existence_semidiscrete}
Let $\kernel$ satisfy~\eqref{kernel_cond1} and let $\xi(\x):=\cker(\x)-\cpot\geq 0$ for all $\x\in\D$, then there exists a solution of~\eqref{CH_VI_greens}.
\end{theorem}
\begin{proof}
To show the existence of the solution of~\eqref{CH_VI_greens}, it is sufficient to show an existence result for the solution of~\eqref{min_problem}, which can also be equivalently written as $\min_{u\in\mathcal{K}_m}\J(u)$. 
Since $u\in\mathcal{K}_m$, and $\mathcal{K}_m\neq\emptyset$, it follows that $u\in L^{\infty}(\D)$ and $\norm{u}_{L^2(\D)}^2\leq|\D|$.
Then, we have that $\J(u)\geq 0$ for all $u\in\mathcal{K}_m$, and there exists $\hat{\J}\geq 0$ that realizes $\hat{\J}=\inf_{u\in\mathcal{K}}\J(u)$. Hence, there exists a minimizing sequence $\{u_n\}_{n\in\mathbb{N}}\subset\mathcal{K}_m$ for $\J(u)$ such that 
\[
\J(u_n)\rightarrow\hat{\J}:=\inf_{u\in\mathcal{K}_m}\J(u),\quad n\rightarrow\infty.
\]
Since $\{u_n\}\subset\mathcal{K}_m$, it follows by Proposition~\ref{prop:max_principle} that $u_n\in L^{\infty}(\DC)$. Then, by the Banach-Alaoglu theorem, we can extract a weak*-convergent subsequence $\{u_{n_j}\}_{n_j\in\mathbb{N}}\subset L^{\infty}(\DC)$: 
\[
u_{n_j}\stackrel{\ast}{\rightharpoonup}\hat{u},\quad j\rightarrow\infty.
\]
Since $\mathcal{K}_m$ is convex and closed, it is weakly closed and $\hat{u}\in\mathcal{K}_m$. Next, we are going to pass to the limit in $\J(u_{n_j})$. Since $\xi\geq 0$ and using expression~\eqref{eq:split_b}, we can write
\begin{equation}
\J(u_{n_j})=\frac{1}{2}\|\sqrt{\xi}u_{n_j}\|^2_{L^2(\D)}-\frac{1}{2}(\kernel*u_{n_j},u_{n_j})_{L^2(\D)}+\frac{\cpot}{2}|\D|+\frac{1}{2\tt}\norm{u_{n_j}-u^{k-1}}^2_{\Va^\prime},
\label{eq:limJ}
\end{equation}
where the first and the last terms in the expression above are continuous and convex, and, hence, weakly lower semi-continuous, i.e., 
\begin{equation*}
\|\sqrt{\xi}\hat{u}\|^2_{L^2(\D)}\leq\lim_{j\rightarrow\infty}\inf\|\sqrt{\xi}u_{n_j}\|^2_{L^2(\D)},\quad\norm{\hat{u}-u^{k-1}}^2_{\Va^\prime}\leq\lim_{j\rightarrow\infty}\inf\norm{u_{n_j}-u^{k-1}}^2_{\Va^\prime}.
\end{equation*}
Next, we consider the second term in~\eqref{eq:limJ}. Let $\phi_{n_j}:=\kernel*u_{n_j}$ and $\hat{\phi}:=\kernel*\hat{u}$. 
Since, $\hat{u}, u_{n_j}\in L^{\infty}(\DC)$, by Proposition~\ref{prop:kernel_prop} it follows that $\phi_{n_j},\hat{\phi}\in C(\DC)$.
Then, using the fact that $u_{n_j}\stackrel{\ast}{\rightharpoonup}\hat{u}$, $j\rightarrow\infty$, and $\kernel\in L^1(\R^n)$, we obtain that
\[
\int_{\DC}u_{n_j}(\xx)\kernel(|\x-\xx|)\d\xx\rightarrow\int_{\DC}\hat{u}(\xx)\kernel(|\x-\xx|)\d\xx,\quad j\rightarrow\infty,
\]
that is, $\phi_{n_j}(\x)\rightarrow\hat{\phi}(\x)$, $\forall\x\in\D$.
Then, by the Lebesgue dominated convergence theorem, we obtain
\[
\int_{\D}u_{n_j}\phi_{n_j}\d\x=\int_{\D}u_{n_j}\phi\d\x+\int_{\D}u_{n_j}(\phi_{n_j}-\phi)\d\x\rightarrow\int_{\D}\hat{u}\hat{\phi}\d\x,\quad j\rightarrow\infty.
\]
Finally, taking the limit in $\J(u_{n_j})$ we obtain
\[
\lim_{j\rightarrow\infty}\inf \J(u_{n_j})\geq \frac{1}{2}\|\sqrt{\xi}\hat{u}\|^2_{L^2(\D)}-(\hat{u},\kernel*\hat{u})_{L^2(\D)}+\frac{\cpot}{2}|\D|+
\norm{\hat{u}-u^{k-1}}^2_{\Va^\prime}=\J(\hat{u}).
\]
That is, 
\[
\J(\hat{u})\leq\lim_{j\rightarrow\infty}\inf \J(u_{n_j})=\lim_{n\rightarrow\infty}\J(u_n)=\inf_{u\in\mathcal{K}_m}\J(u)=\hat{\J}.
\]
On the other hand, by definition of $\hat{\J}$, it follows that $\hat{\J}\leq \J(\hat{u})$, and, hence $\hat{\J}=\J(\hat{u})$.
\end{proof}

\begin{theorem}[Uniqueness]\label{thm:uniqueness_semidiscrete}
Let $\kernel$ satisfy~\eqref{kernel_cond1} and $\xi(\x):=\cker(\x)-\cpot> 0$ for all $\x\in\overline{\D}$. Then, the solution of~\eqref{CH_VI_greens} is unique for $\tt<4C_\eta^{-2}(\xi/(1+C_P^4\ckerr^2)-\eta)$ (Case~1), and $\tt<4(\xi-\eta)/C_\eta^2$ ({\it Case~2}), where $C_\eta$, $0<\eta<\xi$ are as in Proposition~\ref{prop:kernel_prop}. 
\end{theorem}
\begin{proof}
We prove it by contradiction. Let $u_1^k$, $u_2^k\in\mathcal{K}_m$ be two solutions of~\eqref{CH_VI_greens}, and let $\theta^k:=u_1^k-u_2^k\in\mathcal{K}_0$. Taking an addition of~\eqref{CH_VI_greens} tested with $\psi=u^k_2$, when $u_1^k$ is a solution, and vice-versa, and using~\eqref{eq:split_b} we obtain
\begin{multline*}
0\geq b(\theta^k,\theta^k)-\cpot\|\theta^k\|^2_{L^2(\D)}+\frac{1}{\tt}(\G (\theta^k),\theta^k)_{L^2(\D)}\\
=(\xi\theta^k,\theta^k)_{L^2(\D)}-(\kernel*\theta^k,\theta^k)_{L^2(\D)}+\frac{1}{\tt}\norm{\theta^k}^2_{\Vao^\prime}.
\end{multline*}
Invoking Proposition~\ref{prop:kernel_prop} and~\eqref{youngs_conv}, we can estimate the second term above
\begin{multline*}
\norms{(\kernel*\theta^k,\theta^k)_{L^2(\D)}}\leq
\norms{\dualar{\kernel_\eta*\theta^k,\theta^k}}+\norms{\left((\kernel-\kernel_\eta)*\theta^k,\theta^k\right)_{L^2(\D)}}\\
\leq\norm{\nabla\kernel_\eta*\theta^k}_{L^2(\D)}\norm{\theta^k}_{\Va^\prime}+\norm{\kernel-\kernel_\eta}_{L^1(\R^n)}\norm{\theta^k}_{L^2(\DC)}^2\\
\leq C_\eta\norm{\theta^k}_{L^2(\DC)}\norm{\theta^k}_{\Va^\prime}+\eta\norm{\theta^k}_{L^2(\DC)}^2\\
\leq \left({C_\eta^2\tt}/{4}+\eta\right)\norm{\theta^k}_{L^2(\DC)}^2+\frac{1}{\tt}\norm{\theta^k}_{\Va^\prime}^2,
\end{multline*}
where the last estimate is obtained by using Young's inequality~\eqref{youngs} with $\epsilon=2/(C_\eta\tt)$. 
Since $\theta^k\in\mathcal{K}_0$, we have that $\norm{\theta^k}_{\Va^\prime}=\norm{\theta^k}_{\Vao^\prime}$. 
By combining the previous estimates we obtain
\begin{equation}
(\xi\theta^k,\theta^k)_{L^2(\D)}-(C_\eta^2\tt/4+\eta)\norm{\theta^k}_{L^2(\DC)}^2\leq 0.\label{eq:tmp_xi}
\end{equation}
For \textit{Case~2}, the above simplifies to $((\xi-C_\eta^2\tt/4-\eta)\theta^k,\theta^k)_{L^2(\D)}\leq 0$, and, since $\tt<4(\xi-\eta)/C_\eta^2$, it follows that $\theta^k=0$, and hence $u_1^k=u_2^k$.
For \textit{Case~1}, we denote $\theta^k_\D:=\theta^k|_{\D}$, and using the fact that $\theta^k\in\Vb$, we obtain that $\theta^k$ solves the exterior problem~\eqref{exterior_form} on $\DI$ with $g=\theta^k_\D$ on $\D$, and from~\eqref{exterior_bdd}, we get
\[
\norm{\theta^k}_{L^2(\DI)}\leq C_P^2\ckerr\norm{\theta_\D^k}_{L^2(\D)}=C_P^2\ckerr\norm{\theta^k}_{L^2(\D)}.
\]
Then, the above estimate together with~\eqref{eq:tmp_xi} brings us to
\[
\left(\left(\xi-(C_\eta^2\tt/4+\eta)(1+C_P^4\ckerr^2)\right)\theta^k,\theta^k\right)_{L^2(\D)}\leq 0,
\]
which implies for $\tt<4C_\eta^{-2}(\xi/(1+C_P^4\ckerr^2)-\eta)$, that $\theta^k=0$, and hence $u_1^k=u_2^k$.
\end{proof}
\begin{corollary}
If under the assumptions of Theorem~\ref{thm:uniqueness_semidiscrete}, the kernel $\kernel$ additionally  satisfies~\eqref{kernel_cond2}, then $\eta=0$, $C_{\eta}=\ckerrgrad$, and the solution of~\eqref{CH_VI_greens} is unique if $\tt<4\xi/(\ckerrgrad^2(1+C_P^4\ckerr^2))$ (Case~1), and $\tt<4\xi/\ckerrgrad^2$ ({\it Case~2}).
\end{corollary}

We have shown existence and uniqueness for~\eqref{CH_VI_greens}. Next, we show that~\eqref{CH_VI_greens} and~\eqref{CH_VI_tk} are equivalent. 
\begin{proposition}\label{prop:JP}
Under conditions of Theorem~\ref{thm:uniqueness_semidiscrete}, problems~\eqref{CH_VI_greens} and~\eqref{CH_VI_tk} admit unique solutions and are equivalent.
\end{proposition}
\begin{proof}
Since the problems~\eqref{CH_VI_greens} and~\eqref{min_problem} are equivalent, it suffices to show that~\eqref{min_problem} and~\eqref{CH_VI_tk} are equivalent. First, we show the existence of the Lagrange multiplier $\mu^k\in\R$ by verifying a constraint qualification condition~\cite[Theorem~6.3]{troltzsch2010optimal}. For $u\in\mathcal{K}$ and $m\in(|\D|,|\D|)$, we define the equality constraint $G(u)=(u,1)-m$. If $u^k\in\mathcal{K}_m$ is a unique minimizer of~\eqref{min_problem}, then $G(u^k)=0$, and we define the tangent cone
\[
C(u^k)=\{\alpha(u-u^k)\colon \alpha\geq 0,\; u\in\mathcal{K}\}.
\]
Taking into account that $G^\prime(u)=1$, we derive that
\[
S:=G^\prime({u}^k)\mathcal{C}({u}^k)=\left\{\alpha\int_{\D}(u-{u}^k)\d\x,\quad\alpha\geq 0,u \in \mathcal{K}\right\}.
\]
Since $u\in\mathcal{K}$ and $u^k\in\mathcal{K}_m$, it follows that $S\equiv\R$, {and a} constraint qualification is fulfilled. {By}~\cite[Theorem~6.3]{troltzsch2010optimal}, there exist $\mu^k\in\R$ such that
\[
(\nabla_{u}L({u}^k,\mu^k),v-{u}^k)_{L^2(\D)}\geq 0,\quad\forall v\in\mathcal{K},
\]
where $L(u,\mu):=\J(u)+\mu G(u)$ is the Lagrange function, and 
\[
\nabla_u L({u}^k,\mu)=\nabla \J({u}^k)+\mu\nabla G({u}^k)= B{u}^k+F_0^\prime({u}^k)+\frac{1}{\tt}\G({u}^k-u^{k-1})+{\mu^k}.
\]
Setting $w^k:=-\frac{1}{\tt}\G({u}^k-u^{k-1})-{\mu^k}$, and noting that $\frac{1}{\tt}\G({u}^k-u^{k-1})\in\Vao$, $\mu^k\in\R$, we obtain that $w^k\in\Va$, and $({u}^k,w^k)\in\mathcal{K}\times\Va$ is a solution of~\eqref{CH_VI_tk}. 
  \end{proof}

\subsection{Properties of the solution}\label{sec:sec:regularity_tk}
In this section we prove that, under certain conditions, the solution of the nonlocal Cahn-Hilliard problem~\eqref{CH_VI_tk} at each time step admits discontinuous solutions that imply sharp interfaces in the model. 

From~\eqref{CH_VIG_tk} and~\eqref{eq:split_b} we obtain that the following relation holds for $k=1,\dots,K$ a.e. in $\D$:
\begin{equation}
w^k= Bu^k+F^{\prime}_0(u^k)+\lambda^k =\xi u^k-\kernel*u^k+\lambda^k, 
\label{eq:w1}
\end{equation}
where, we recall $\xi(\x)=\cker(\x)-\cpot$. We define for $k=1,\dots,K$
\begin{equation*}
{g^k:=w^k+\kernel*u^k=-\frac{1}{\tt}\G({u}^k-u^{k-1})-{\mu^k}+\kernel*u^k,}
\end{equation*}
and from~\eqref{eq:w1} it follows that $g^k=\param u^k+\lambda^k$. 
Hence, for $\param(\x)>0$, $\x\in\overline{\D}$, {since $\lambda^k\in\partial{I}_{[-1,1]}(u^k)$, where we recall
\begin{equation}
\lambda^k\in\partial I_{[-1,1]}(u^k)=\begin{cases}
(-\infty,0] \quad&\text{if } \;\;u^k=-1,\\
0 \quad&\text{for } u^k\in(-1,1),\\
[0,+\infty) \quad&\text{if } \;\;u^k=1,
\end{cases}\label{eq:subdiff_tk}
\end{equation}
the solution $u^k$} at each time step $t^k$ can be {calculated} as a pointwise projection of $g^k/\param$ onto $[-1,1]$:
\begin{align}
u^k=P_{[-1,1]}\left(\frac{1}{\param}g^k\right)=\begin{cases}
1,\quad&\text{if }\;\; g^k\geq\param,\\
-1,\quad&\text{if }\;\; g^k\leq -\param,\\
{g^k}/{\param},\quad&\text{if }\;\; g^k\in(-\param,\param).
\end{cases}
\label{proj_u}
\end{align}
The projection formula~\eqref{proj_u} provides a crucial insight into the properties of the solution, such as {the} regularity stated below.
\begin{theorem}[Improved regularity]\label{thm:reg_tk}
Let $(u^k,w^k)\in\mathcal{K}\times\Va$ be the solution pair of~\eqref{CH_VI_tk}, and the kernel $\kernel$ satisfies~\eqref{kernel_cond1}--\eqref{kernel_cond2}. If $\xi(\x)>0$ for all $\x\in\overline{\D}$, then $u^k\in H^1(\D)$, for all $K=1,\dots,K$, and 
\begin{equation}
\norm{\nabla u^k}_{L^2(\D)}\leq\norm{\nabla(\xi^{-1}g^k)}_{L^2(\D)}\leq C\norm{g^k}_{H^1(\D)},\label{reg_tk_stability}
\end{equation}
where $C>0$ depends only on $\xi$ and $\kernel$. 
\end{theorem}
\begin{proof}
Since $\kernel\in W^{1,1}(\R^n)$ and $u^k\in L^{\infty}(\R^n)$ (by Proposition~\ref{prop:max_principle}), invoking Proposition~\ref{prop:kernel_prop} it follows that $\kernel*u^k\in W^{1,\infty}(\R^n)$. Since $w^k\in H^1(\D)$, we obtain that $g^k=w^k+\kernel*u^k\in {H^1(\D)}$.   Invoking the projection formula~\eqref{proj_u} and using the stability of the $L^2$-projection in $H^1$, we obtain the following estimate
\begin{equation}
\norm{\nabla u^k}_{L^2(\D)}\leq\norm{\nabla(\xi^{-1}g^k)}_{L^2(\D)}.\label{eq:stability_estimate}
\end{equation}
We note that for {\it Case~1}, $\xi(\x)$ is constant for all $\x\in\D$. Therefore, from the above estimate we immediately deduce that $\norms{u^k}_{H^1(\D)}\leq{\xi^{-1}}\norms{g^k}_{H^1(\D)}\leq C$, and, hence, $u\in H^1(\D)$. For {\it Case~2}, since $\kernel\in W^{1,1}(\R^n)$, we have that $\xi,\xi^{-1}\in W^{1,\infty}(\D)$. Indeed, since $\xi(\x)>0$ in $\overline{\D}$, it follows that there exists $\xi_{\min}>0$, such that $\xi(\x)\geq\xi_{\min}>0$, and 
\begin{equation}
\norm{\nabla(\xi^{-1})}_{L^{\infty}(\D)}\leq\xi_{\min}^{-2}\norm{\nabla\xi}_{L^{\infty}(\D)}\leq\xi_{\min}^{-2}\norm{\nabla\kernel}_{L^1(\R^n)}<\infty.\label{eq:lipschitz_xi}
\end{equation}
Then, using the product rule, the estimate~\eqref{eq:stability_estimate} reduces to
\begin{multline*}
\norms{u^k}_{H^1(\D)}\leq\norms{\xi^{-1}g^k}_{H^1(\D)}\leq\ckerr\norms{g^k}_{H^1(\D)}
+\norm{\nabla(\xi^{-1})}_{L^{\infty}(\D)}\norm{g^k}_{L^2(\D)}\\
\leq C\norm{g^k}_{H^1(\D)},
\end{multline*}
and, thus, $u^{{k}}\in H^1(\D)$. 
\end{proof}
\begin{corollary}
Let $\kernel$ satisfy~\eqref{kernel_cond1}--\eqref{kernel_cond2}. Then, for $\xi(\x)\geq 0$ for all $\x\in\D$, the Lagrange multiplier in~\eqref{CH_VIG_tk} fulfills $\lambda^k\in H^1(\D)$, $k=1,\dots,K$.
\end{corollary}

In contrast to the previous result, we show next that if $\xi=0$, the solution $u^k$ is generally not smoother {than $u^k\in L^{\infty}(\D)$}, and furthermore, it can posses jump-discontinuities, which {allow} that the solution can admit only pure phases.
\begin{theorem}[Sharp interfaces]\label{thm:properties_tk}
Let $(u^k,w^k)\in\mathcal{K}\times\Va$ be the solution pair of~\eqref{CH_VI_tk}, and the kernel $\kernel$ satisfies~\eqref{kernel_cond1}. If $\param(\x):=\cker(\x)-\cpot=0$, $\forall\x\in\D$, then it holds for all $k=1,\dots,K$ a.e. in $\D$:
\begin{equation}
u^k\in\begin{cases}
\{1\},\quad&\text{if }\;\; g^k > 0,\\
\{-1\},\quad&\text{if }\;\; g^k < 0,\\
[-1,1],\quad&\text{if }\;\; g^k = 0,
\end{cases}
\end{equation}
where $g^k=w^k+\kernel*u^k$. Thus, in the case that the function $g^k$ assumes the values zero only on a set of measure zero, i.e., $|\{g^k=0\}|=0$, the variable $u^k$ assumes only the extreme values $-1$ and $1$ almost everywhere. That is, the solution $u^k$, $k=1,\dots,K$ is discontinuous and {consists} only of pure phases $u^k=1$ and $u^k=-1$.
\end{theorem}
\begin{proof}
We recall that from~\eqref{eq:w1}, $g^k=\param u^k+\lambda^k$, where $\lambda^k=\lambda_{+}^k-\lambda_{-}^k$, $\lambda_{\pm}^k\geq 0$. Then, if $\param=0$, we obtain that $g^k=\lambda^k$, where we recall that for $k=1,\dots,K$ { $\lambda^k\in\partial{I}_{[-1,1]}(u^k)$~\eqref{eq:subdiff_tk} holds true.}
Hence, if $g^k>0$ we get $\lambda^k>0$ and $u^k=1$. Similarly, if $g^k<0$, then $\lambda^k<0$ and $u^k=-1$. Finally, for $g^k=0$ we obtain that $\lambda^k=0$ and $u^k\in[-1,1]$, and we conclude the proof.
\end{proof}
\begin{remark}
There are settings  for which we can guarantee that $|\{g^k=0\}|=0$. In particular, if we 
consider the steady state problem in $\D\subset\R^1$, i.e., for $K\to \infty$, $w^k \to-\hat{\mu}$ and $u^k\to\hat{u}$, where $\hat{u}$ and $\hat{\mu}$ are the steady state solution and mean-value, respectively. Then, taking, e.g., an analytic convolution kernel $\kernel$ we obtain that $\hat{g}:= \kernel*\hat{u}-{\hat{\mu}}$ is analytic and not constantly equal to zero, {unless $\hat{u}\equiv m$ everywhere}. Hence, {if $\hat{u}\not\equiv m$}, $|\{\hat{g}=0\}|=0$ always holds true and the solution admits only pure phases.
\end{remark}


\section{Continuous problem: Existence, uniqueness and a sharp interface condition}\label{sec:continuous}
In this section, we analyze the continuous in time problem~\eqref{CH_VI}. To derive the corresponding existence and uniqueness results, we analyze the semidiscrete problem defined in the previous section by taking the limit as $\tt\to 0$. We demonstrate that the projection formula~\eqref{proj_u} also remains valid for the continuous case. 

\subsection{Existence of a solution}
Let ${X}$ be either $\Vb$, $L^2(\D)$ or $\Va$, and recall that $\Q=(0,T)\times{\D}$ and $\QC=(0,T)\times\DC$. Then, for $K\in\mathbb{N}$, $\tt=T/K$, $T>0$, and a given sequence of functions $\{z^k\}_{k=1}^K\subset X$, we introduce piecewise constant and piecewise linear interpolants: 
\begin{equation}
\begin{aligned}
\inp{z}(t):=z^k,\quad \inpl{z}(t):=z^{k-1},\quad t\in(t_{k-1},t_k],\quad k=1,\dots,K,\\
\inpt{z}(t):=\frac{t-t_{k-1}}{\tt} z^{k}+\frac{t_k-t}{\tt}z^{k-1},\quad t\in[t_{k-1},t_k],\quad k=1,\dots,K.
\end{aligned}\label{eq:t_interpolant}
\end{equation}
In terms of the above notations, it follows that $\partial_t\inpt{z}=(z^{k}-z^{k-1})/\tt$, $t\in[t_{k-1},t_k]$, $k=1,\dots,K$. Then by setting $\inp{w}(t):=-\G(\partial_t\inpt{z}(t))-{\inp{\mu}(t)}$, the semi-discrete problem~\eqref{CH_VIG_tk} can be recast for a.e. $t\in (0,T)$ as follows: 
\begin{subequations}
\begin{align}
(\G(\partial_t\inpt{u}(t))+\xi\inp{u}(t)-\kernel*\inp{u}(t)+\inp{\lambda}(t)+\inp{\mu}(t),\psi)_{L^2(\D)}&=0,\;\;\forall\psi\in\Vb,\label{CH_VI_t1}\\
(\eta-\inppm{\lambda}(t),1\mp\inp{u}(t))_{L^2(\D)}&\geq 0,\;\;\forall\eta\in\M.\label{CH_VI_t2}
\end{align}\label{CH_VI_t}
\end{subequations}
Next, we establish the existence result, and invoking the projection formula~\eqref{proj_u} derive an improved regularity of the solution.

\begin{theorem}[Existence and improved regularity]\label{thm:existence_t}
Let $\kernel$ satisfy~\eqref{kernel_cond1}--\eqref{kernel_cond2}, and $\xi(\x):=\cker(\x)-\cpot\geq 0$ for all $\x\in{\D}$, then there exists a solution pair $(u(t),w(t))$ of~\eqref{CH_VI}, such that 
\[u\in L^{\infty}(\QC),\quad \partial_t u\in L^2(0,T;\Va^{\prime}), \quad\text{and}\quad w\in L^2(0,T;\Va).
\]
Moreover, if $\xi(\x)>0$ for all $\x\in\overline{\D}$, then 
\[
u\in W(0,T)\cap L^{\infty}(\QC) \quad\text{and}\;\; w\in L^2(0,T;\Va),
\]
where 
$W(0,T):=\left\{v\in L^2(0,T; \Va)\colon \partial_t v \in L^2(0,T;\Va^{\prime})\right\}$.
\end{theorem}
\begin{proof}

\textbf{A~priori estimates.}
Since $|\inp{u}(t)|\leq 1$ {for $t\in (0,T)$ a.e.}, it immediately follows that $u\in L^{\infty}(\QC)$.
Next, we show that $\partial_t\inpt{u}\in L^2(0,T;\Va^\prime)$. From the energy minimization~\eqref{min_problem} and Proposition~\ref{prop:JQ} it follows that $\J(u^k)\leq \J(u^{k-1})$, $k=1,\dots,K$. That is,
\begin{align*}
\frac{1}{2\tt}\norm{u^k-u^{k-1}}^2_{\Va^\prime}\leq\frac{1}{2}\norms{u^{k-1}}^2_{\Vb}-\frac{1}{2}\norms{u^k}_{\Vb}^2-\frac{\cpot}{2}\norm{u^{k-1}}^2_{L^2(\D)}+\frac{\cpot}{2}\norm{u^k}^2_{L^2(\D)}.
\end{align*}
Then, utilizing the above expression we deduce
\begin{multline}
\norm{\partial_t\inpt{u}}^2_{L^2(0,T;\Va^\prime)}=\int_0^T\norm{\partial_t\inpt{u}}^2_{\Va^\prime}\d t=\sum_{k=1}^K\int_{t_{k-1}}^{t_k}\frac{1}{\tt^2}\norm{u^k-u^{k-1}}_{\Va^\prime}^2\d t\\
\leq \norms{u^0}^2_{\Vb}+\cpot\norm{u^K}^2_{L^2(\D)}\leq C. \label{eq:time_der}
\end{multline}

\textbf{Case~$\xi>0$.}
We consider~\eqref{CH_VI_t} with $\psi={\intm{m}}-u(t)\in\Vb$, where {$\intm{m}=\frac{1}{|\D|}\int_\D u(\x)\d\x={m}/{|\D|}$, $\intm{m}\in(-1,1)$}, and taking into account that 
$(\mu(t),{\intm{m}}-u(t))_{L^2(\D)}=0$ for $t\in (0,T)$ {a.e.}, we arrive at
\begin{align*}
-(\G(\partial_t\inpt{u}(t)),\inp{u}(t))_{L^2(\D)}+\left(\xi\inp{u}(t)-\kernel*\inp{u}(t)+\inp{\lambda}(t),\intm{m}-\inp{u}(t)\right)_{L^2(\D)}=0.
\end{align*}
From the complementarity conditions~\eqref{CH_VI_t2} it follows that $\inp{u}(t)=\pm 1$ if $\inppm{\lambda}(t)>0$. Then, using repeatedly Cauchy and Young's inequalities, and~\eqref{poincare_local}, we obtain
\begin{multline*}
0\leq (\inpp{\lambda},1-\intm{m})_{L^2(\D)}+(\inpm{\lambda}(t),1+\intm{m})_{L^2(\D)}=-(\inp{\lambda}(t),\intm{m}-\inp{u}(t))_{L^2(\D)}\\
\leq \norms{\dualar{\G(\partial_t\inpt{u}(t)),\inp{u}(t)}}+\norm{\xi\inp{u}(t)-\kernel*\inp{u}(t)}_{L^2(\D)}\norm{\intm{m}-\inp{u}(t)}_{L^2(\D)}\\
\leq\norm{\G(\partial_t\inpt{u}(t))}_{\Va}\norm{\inp{u}(t)}_{\Va^\prime}
+C\left(\norm{\xi\inp{u}(t)}_{L^2(\D)}+\norm{\inp{u}(t)}_{L^2(\DC)}\right)\\
\leq\norm{\G(\partial_t\inpt{u}(t))}^2_{\Va}
+C\left(\norm{\inp{u}(t)}_{\Va^\prime}^2+\norm{\xi\inp{u}(t)}_{L^2(\D)}+\norm{\inp{u}(t)}_{L^2(\DC)}\right)\\
\leq C\left(\norms{\G(\partial_t\inpt{u}(t))}^2_{\Va}+1\right).
\end{multline*}
Now integrating the last expression from $(0,T)$ and invoking the estimate~\eqref{eq:time_der}, we obtain that
\begin{equation*}
\int_0^T|\inp{\lambda}(t)|\d t\leq \int_{0}^T(\inpp{\lambda}(t)+\inpm{\lambda}(t))\d t \leq C\left(\int_{0}^T\norm{\partial_t\inpt{u}(t)}^2_{\Va^\prime}\d t+T\right)\leq C,
\end{equation*}
and, hence $\inp{\lambda}\in L^2(0,T;L^1(\D))$. Therefore, from~\eqref{CH_VI_t1} it follows that 
{
\[
\inp{\mu}=-\G(\partial_t\inpt{u})-\xi\inp{u}+\kernel*\inp{u}-\inp{\lambda}\in L^2(0,T;\R),
\]} and, hence, $-\G(\partial_t\inpt{u})-\inp{\mu}=:\inp{w}\in L^2(0,T;\Va)$, with the uniform bounds $\norm{\inp{\mu}}^2_{L^2(0,T;\R)}\leq C$ and $\norm{\inp{w}}^2_{L^2(0,T;\Va)}\leq C$. From the derived regularity estimates, and invoking again~\eqref{CH_VI_t1}, we immediately conclude that $\inp{\lambda}\in L^2(\Q)$, and $\norm{\inp{\lambda}}_{L^2(\Q)}\leq C$. 

Next, we are going to invoke the higher regularity results from Theorem~\ref{thm:reg_tk}. We introduce an interpolant of $g^k$, defined in Section~\ref{sec:sec:regularity_tk}, $\inp{g}(t)=\inp{w}(t)+\kernel*\inp{u}(t)$. From the projection formula~\eqref{proj_u}, we obtain that $\inp{u}=P_{[-1,1]}(\xi^{-1}\inp{g}(t))$ holds a.e. for $t\in (0,T)$, and from~\eqref{reg_tk_stability} it holds that
\[
\norm{\nabla\inp{u}(t)}_{L^2(\D)}^2\leq C\norm{\inp{g}(t)}^2_{H^1(\D)}.
\]  
Integrating the above inequality from $(0,T)$, and invoking~\eqref{eq:time_der},~\eqref{poincare_local}, and~\eqref{youngs_conv}, leads to
\begin{multline*}
\int_{0}^T\norm{\nabla\inp{u}(t)}_{L^2(\D)}^2\d t\leq C\int_{0}^T\norm{\inp{g}(t)}^2_{H^1(\D)}\d t\\
\leq 
C\left(\int_{0}^T\norms{\G(\partial_t\inpt{u}(t))}^2_{\Va}\d t +\int_{0}^T\norm{\inp{u}(t)}_{L^2(\DC)}^2\d t+\norm{\inp{\mu}}^2_{L^2(0,T;\R)}\right)\\
\leq C\left(\norm{\partial_t\inpt{u}}_{L^2(0,T;\Va^{\prime})}^2+\norm{\inp{u}}^2_{L^{\infty}(\QC)}+\norm{\inp{\mu}}^2_{L^2(0,T;\R)}\right)\leq C,
\end{multline*}
and, hence, $\inp{u}\in L^2(0,T;\Va)$. Since, $\xi\in W^{1,\infty}(\D)$~\eqref{eq:lipschitz_xi}, it follows that $\xi\inp{u}\in L^2(0,T;\Va)$, and, hence, from~\eqref{CH_VI_t} we immediately deduce that 
{
\[
\inp{\lambda}=-\G(\partial_t\inpt{u})-\xi\inp{u}+\kernel*\inp{u}-\inp{\mu}\in L^2(0,T;\Va).
\]}
Collecting the above estimates, we arrive at the following a priori energy bound
\[
\norm{\partial_t\inpt{u}}^2_{L^2(0,T;\Va^\prime)}+\norm{\inp{u}}^2_{L^2(0,T;\Va)}+\norm{\inp{\lambda}}^2_{L^2(0,T;\Va)}+\norm{\inp{\mu}}^2_{L^2(0,T;\R)}\leq C.
\]

\textbf{Case~$\xi=0$.} Following the same steps as above we obtain that $\inp{\lambda}\in L^2(\Q)$, $\inp{\mu}\in L^2(0,T;\R)$, and $\inp{w}\in L^2(0,T;\Va)$ with the corresponding uniform norm bounds. However, in contrast to the the previous case, here the projection formula~\eqref{proj_u} does not hold, and, in general, we can not expect a solution to have an improved regularity. Nevertheless, we still can derive an improved regularity for the Lagrange multiplier $\inp{\lambda}$. Indeed, using that $\kernel*\inp{u}\in L^2(0,T;\Va)$ for $\kernel\in W^{1,1}(\R^n)$, $\inp{u}\in L^{\infty}(\QC)$,  and $\G(\partial_t\inpt{u})\in L^2(0,T;\Va)$, $\inp{\mu}\in L^2(0,T;\R)$, we conclude from~\eqref{CH_VI_t1}, that $\inp{\lambda}\in L^2(0,T;\Va)$, and we have the following estimate:
\[
\norm{\partial_t\inpt{u}}^2_{L^2(0,T;\Va^\prime)}+\norm{\inp{u}}^2_{L^2(0,T;\Vb)}+\norm{\inp{\lambda}}^2_{L^2(0,T;\Va)}+\norm{\inp{\mu}}^2_{L^2(0,T;\R)}\leq C.
\]

\textbf{Limit.} From the a~priori estimates and the Banach-Alaouglu theorem, there exist functions $u$, $\lambda$ and $\mu$ such that
\begin{align*}
&(\xi\geq 0) &&\inp{u}\rightharpoonup u \quad &\text{{weakly-* in}}\quad  &L^{\infty}(\QC),\\
&(\xi >0) &&\inp{u}\rightharpoonup u \quad &\text{{weakly-* in}}\quad  &L^{\infty}(0,T;\Va),\\
& &&\partial_t\inpt{u}\rightharpoonup \partial_t u \quad &\text{{weakly in}}\quad  &L^2(0,T;\Va^\prime),\\
& &&\inp{\mu}\rightharpoonup \mu\quad &\text{{weakly in}}\quad  &L^2(0,T;\R),\\
& &&\inp{\lambda}\rightharpoonup \lambda \quad &\text{{weakly in}}\quad  &L^2(0,T;\Va),
\end{align*}
as $\tt\to 0$ (or equivalently, as $K\to\infty$), where by a slight abuse of notation we keep the same subscript for convergent sub-sequences. Next, we are going to pass to the limit $\tt\to 0$ in an equivalent time-integrated variant of~\eqref{CH_VI_t}. That is, we consider
\begin{equation}
\begin{aligned}
(\G(\partial_t\inpt{u})+\xi\inp{u}-\kernel*\inp{u}+\inp{\lambda}+\inp{\mu},\psi)_{L^2(\Q)}=0,\;\;\forall\psi\in L^2(0,T;\Vb),\\
(\eta-\inppm{\lambda},1\mp\inp{u})_{L^2(\Q)}\geq 0,\;\;\forall\eta\in L^2(0,T;\M),
\end{aligned}\label{CH_VI_T}
\end{equation}
where now, letting $\tt\to 0$ and using the linearity of the first equation in~\eqref{CH_VI_T}, we obtain that the limiting functions $u$, $\lambda$ and $\mu$ satisfy the first equation in~\eqref{CH_VIGt}. 
To be able to pass to the limit in the {variational} inequality constraints~\eqref{CH_VI_t2}, and to deal with the induced nonlinearity, we need to upgrade a weak convergence to a strong one. Using the Aubin-Lions compactness lemma we obtain that the embedding $L^2(0,T;L^2(\D))\cap H^1(0,T;\Va^\prime)\hookrightarrow L^2(0,T;\Va^\prime)$ is compact, and hence, we obtain that
\[
\inpt{u}\rightarrow u\quad\text{strongly in}\;\; L^2(0,T;\Va^\prime). 
\]
Next, we show we show that $\inp{u}\to u$ strongly in $L^2(0,T;\Va^\prime)$. First, we note that the following estimate holds true for piecewise linear and piecewise constant interpolants, see, e.g.,~\cite[Proposition~3.9]{coli2018}:
\[
\norm{\inpt{u}-\inp{u}}_{L^2(0,T;\Va^\prime)}\leq C\tt\norm{\partial_t\inpt{u}}_{L^2(0,T;\Va^\prime)}. 
\]
Then, using the above estimate and a triangle inequality we arrive at
\[
\norm{\inp{u}-u}_{L^2(0,T;\Va^\prime)}\leq\norm{\inp{u}-\inpt{u}}_{L^2(0,T;\Va^\prime)}+\norm{\inpt{u}-u}_{L^2(0,T;\Va^\prime)}\to 0,
\]
as $\tt\to 0$ (or equivalently, as $K\to\infty$). Similarly as in~\eqref{CH_VIG_complem} we can decompose the complementarity conditions~\eqref{CH_VI_T} as
\[
\int_{0}^T(\eta(t),1\mp\inp{u}(t))_{L^2(\D)}\d t\geq 0\quad\text{and}\quad\int_{0}^T
(\inppm{\lambda}(t),1\mp\inp{u}(t))_{L^2(\D)}\d t =0.\]
Now passing to the limit $\tt\to 0$ in the above expressions and using the fact that the inner product of strong and weak convergences converges, we obtain
\begin{align*}
0&\geq (\eta,\inp{u}\mp 1)_{L^2(\Q)}\to (\eta,{u}\mp 1)_{L^2(\Q)},\\
0&=\dual{\inppm{\lambda},1\mp \inp{u}}_{L^2(0,T;\Va)\times L^2(0,T;\Va^\prime)} \to \dual{{\lambda}_{\pm},1\mp {u}}_{L^2(0,T;\Va)\times L^2(0,T;\Va^\prime)},
\end{align*}
that leads that the limiting functions $\lambda=\lambda_{+}-\lambda_{-}$ and $u$ satisfy the complementarity conditions in~\eqref{CH_VIGt}, and since the problems~\eqref{CH_VIGt} and~\eqref{CH_VI} are equivalent, this concludes the proof. 
\end{proof}

\subsection{Uniqueness and continuous dependence result}
\begin{theorem}\label{thm:uniqueness_t}
Let $\kernel$ satisfy~\eqref{kernel_cond1}--\eqref{kernel_cond2}, and $\xi(\x):=\cker(\x)-\cpot>0$ for all $\x\in\overline{\D}$, then the solution $u\in W(0,T)\cap L^\infty(\QC)$ of the problem~\eqref{CH_VI} is unique. Moreover, if $u_1$ and $u_2$ are two solutions of~\eqref{CH_VI} with the initial conditions $u_1^0=u_1(0)$ and $u_2^0=u_2(0)$, respectively, then we have the following continuous dependence result 
\begin{equation}
\norm{u_1(t)-u_2(t)}^2_{\Va^\prime}\leq e^{C T}\norm{u_1^0-u_2^0}^2_{\Va^\prime},\quad\text{for a.e. }t\in(0,T),
\label{eq:cont_depend}
\end{equation}
where a constant $C>0$ is independent of the initial condition. 
\end{theorem}
\begin{proof}
Let $(u_1,w_1)$ and $(u_2,w_2)$ be two solutions of~\eqref{CH_VI} with the corresponding initial conditions $u^0_1=u_1(0)$ and $u_2^0=u_2(0)$. We define $u(t):=u_1(t)-u_2(t)$, $w(t):=w_1(t)-w_2(t)$ and $u(0)=u_1(0)-u_2(0)$. Next, we consider~\eqref{CH_VI} for solutions $(u_1,w_1)$ and $(u_2,w_2)$. Taking a difference and choosing the test function $\phi=\G u(t)$, noting that $u(t)\in \Vb\subset\Va^\prime$, and exploiting~\eqref{greens_prop2}, we arrive at
\begin{equation}
\frac{\d}{\d t}\norm{u(t)}^2_{\Va^\prime}+(w(t),u(t))_{L^2(\D)}=0,\quad\text{a.e. for } t\in(0,T).  \label{eq:uniq_proof1}
\end{equation}
We consider the second inequality in the problem~\eqref{CH_VI} for $(u_1,w_1)$ and $(u_2,w_2)$ tested with $\phi=u_2$ and $\phi=u_1$, respectively. Then, {taking the difference of the two} inequalities and accounting for~\eqref{eq:split_b} we obtain that
\begin{equation}
(\xi u(t),u(t))_{L^2(\D)}-(\kernel* u(t),u(t))_{L^2(\D)}-(w(t),u(t))_{L^2(\D)}\leq 0. \label{eq:uniq_proof2}
\end{equation}
Using similar techniques as in the proof of Theorem~\ref{thm:uniqueness_semidiscrete} we can estimate the second term in the above inequality:
\begin{multline}
\norms{(\kernel*u(t),u(t))_{L^2(\D)}}=\norms{\dualar{\kernel*u(t),u(t)}}\leq \norm{\nabla\kernel*u(t)}_{L^2(\D)}\norm{u(t)}_{\Va^\prime}\\
\leq \ckerrgrad\norm{u(t)}_{L^2(\DC)}\norm{u(t)}_{\Va^\prime}
\leq \ckerrgrad\left(\frac{\epsilon}{2}\norm{u(t)}^2_{L^2(\DC)}+\frac{1}{2\epsilon}\norm{u(t)}^2_{\Va^\prime}\right)\\
\leq\frac{\ckerrgrad\epsilon}{2}(1+C_P^4\ckerr^2)\norm{u(t)}^2_{L^2(\D)}+\frac{\ckerrgrad}{2\epsilon}\norm{u(t)}_{\Va^\prime}^2,\label{eq:uniq_proof3}
\end{multline}
with $\epsilon>0$. Since, $\xi(\x)\geq 0$ in $\overline{\D}$, we obtain that there exists $\xi_{\min}>0$, such that $\xi(\x)\geq\xi_{\min}>0$. Then, combining~\eqref{eq:uniq_proof2} and~\eqref{eq:uniq_proof3} we obtain
\[
\left(\xi_{\min}-\frac{\ckerrgrad\epsilon}{2}(1+C_P^4\ckerr^2)\right)\norm{u(t)}^2_{L^2(\D)}-\frac{\ckerrgrad}{2\epsilon}\norm{u(t)}^2_{\Va^\prime}-(w(t),u(t))_{L^2(\D)}\leq 0.
\]
Choosing $\epsilon=\xi_{\min}\ckerrgrad^{-1}(1+C_P^4\ckerr^2)^{-1}$ and 
adding the above inequality to~\eqref{eq:uniq_proof1} leads to
\[
\frac{\d}{\d t}\norm{u(t)}^2_{\Va\prime}+\frac{\xi_{\min}}{2}\norm{u(t)}^2_{L^2(\D)}\leq C\norm{u(t)}^2_{\Va^\prime}.
\]
Now, applying the Gronwall lemma we obtain the desired estimate and conclude the proof. 
\end{proof}

\subsection{Properties of the solution}
Similarly as in the time discrete case discussed in Section~\ref{sec:sec:regularity_tk}, for $\xi(\x)>0$, $\forall\x\in\overline{\D}$, we obtain the projection formula that holds a.e. in $(0,T)\times\D$:
\begin{align}
u(t)=P_{[-1,1]}\left(\frac{1}{\param}g(t)\right)=\begin{cases}
1,\quad&\text{if }\;\; g(t)\geq\param,\\
-1,\quad&\text{if }\;\; g(t)\leq -\param,\\
({1}/{\param})g(t),\quad&\text{if }\;\; g(t)\in(-\param,\param),
\end{cases}
\label{proj_ut}
\end{align}
where, we recall, $g(t):=w(t)+\kernel*u(t)$. Analogously, for $\xi(\x)=0$, $\x\in\D$, the solution can admit sharp interfaces.
\begin{theorem}[Sharp interfaces]\label{thm:sharp_interfaces_t}
Let $(u(t),w(t))$ be the solution pair of~\eqref{CH_VI}, and $\kernel$ satisfy~\eqref{kernel_cond1}. If $\param(\x):=\cker(\x)-\cpot=0$, $\forall\x\in\D$, then it holds a.e. in $(0,T)\times\D$:
\begin{equation}
u(t)\in\begin{cases}
\{1\},\quad&\text{if }\;\; g(t) > 0,\\
\{-1\},\quad&\text{if }\;\; g(t) < 0,\\
[-1,1],\quad&\text{if }\;\; g(t)= 0,
\end{cases}
\end{equation}
where $g(t)=w(t)+\kernel*u(t)$, and, if $|\{g(t)=0\}|=0$ a.e. $t\in(0,T)$, then $u(t)$ assumes only phases $-1$ and $1$ a.e. in $(0,T)\times\D$.
\end{theorem}


\section{Numerical results}\label{sec:numerics}
In this section we present numerical experiments for one and two-dimensional cases for the nonlocal operator~\eqref{operatorB} defined for {\it Case~1} and {\it Case~2}.

\subsection{Discretization}
Let $\D\subset\R^n$, $d\geq 1$, be a polygonal domain. 
We partition $\DD$ into a shape regular quasi-uniform triangulation $\{\mathcal{T}_h\}_h$ such that there exists a proper triangulation $\{\mathcal{T}^\prime_h\}_h\subset\{\mathcal{T}_h\}_h$ of $\D$, that respects the boundary of $\D$. We denote by $h$ the maximum diameter of the elements $K\in\mathcal{T}_h$ and set $\overline{\D^h\cup\DI^h}=\cup_{K\in\mathcal{T}_h}\overline{K}$ and $\overline{\D^h}=\cup_{K\in\mathcal{T}^\prime_h}\overline{K}$.

We employ the implicit Euler time stepping scheme {\eqref{CH_VI_tk} together} with  piecewise-linear continuous finite elements for spatial discretization. We define
\begin{align*}
\Sho&=\{v_h\in C^0({\overline{\D}})\colon 
v_h|_{K}\in\mathcal{P}_1(K),\; \forall K\in\mathcal{T}_h\},\\
\Sh&=\{v_h\in C^0({\overline{\DC}})\colon 
v_h|_{K}\in\mathcal{P}_1(K)\quad \forall K\in\mathcal{T}_h\}.
\end{align*}
Let $\mathcal{J}_h^{\D}$ denote the set of nodes corresponding to the triangulation of $\D$ (including the boundary nodes), $\mathcal{J}_h^{I}$ the set of nodes in $\DD\setminus\overline{\D}$, and $\mathcal{J}_h^{\DC}$ the set of all nodes. Also, we set $p_j\in\mathcal{J}_h^*$, $*\in\{\D,I,\DC\}$ to be the coordinates of the corresponding nodes.
Then, we can represent $\Sho=\spann\{\phi_p, \;p\in\mathcal{J}_h^{\D}\}$ and $\Sh=\spann\{\phi_p, \;p\in\mathcal{J}_h^{\DC}\}$, where $\phi_i$ are the nodal Lagrange basis 
functions, $\phi_i(\x_j)=\delta_{i,j}$.

Similarly as in~\cite{blank2011} we employ a mass-lumping approach, where 
by $(\cdot,\cdot)_h$ and $(\cdot,\cdot)_{\widetilde{h}}$ we denote a mass-lumped $L^2(\D)$ and $L^2(\DC\setminus\D)$ inner products, respectively. To adapt the mass-lumping procedure for the nonlocal term, we employ the trapezoidal quadrature rule to assemble $b(\cdot,\cdot)$. In this way, the resulting discretized bilinear form preserves the mass-lumping property. Indeed, let $I^{\x}_h[\cdot]$ be the nodal interpolant {for the variable} $\x$ (respectively {for} $\xx$, $I^{\xx}_h$) and let $\cker^h(\x):=\int_{\DC}I_h^{\xx}[\kernel(\x,\xx)]\d\xx$. For any $\phi,\psi\in\Sh$ we introduce the discrete convolution
\begin{equation*}
(\kernel\circledast\phi)(\x):=\int_{\DC}I_h^{\xx}[\kernel(\x,\xx)\phi(\xx)]\d\xx=\sum_{k\in\mathcal{J}_h^{\DC}}\widetilde{m}_k\kernel(\x,\xx_k)\phi(\xx_k),\quad \widetilde{m}_k =\int_{\DC}\phi_k(\x)\d\x
\end{equation*}
and {obtain the mass-lumped} bilinear form
\begin{multline*}
b_h(\phi,\psi):=\ \frac{1}{2}\int_{\DC}\int_{\DC}I_h^{\x}I_{h}^{\xx}\left[(\phi(\x)-\phi(\xx))(\psi(\x)-\psi(\xx))\right]\d\x\d\xx\\
=\ \int_{\D}I_h^{\x}\left[\cker^h\phi(\x)\psi(\x)\right]\d\x-\int_{\D}I_h^{\x}\left[\psi(\x)(\kernel\circledast\phi)(\x)\right]\d\x+\int_{\DC\setminus\D}I_h^{\x}\left[\psi(\x)\mathcal{N}_h\phi(\x)\right]\d\x\\
=(\cker^h\phi,\psi)_h-(\kernel\circledast\phi,\psi)_h+(\mathcal{N}_h\phi,\psi)_{\widetilde{h}},
\end{multline*}
where
\[
\mathcal{N}_h\phi(\x):=\int_{\DC}I_h^{\xx}\left[(\phi(\x)-\phi(\xx))\kernel(\x,\xx)\right]\d\xx=\cker^h\phi(\x)-(\kernel\circledast\phi)(\x).
\]
From the above expression we see that mass-lumping for the nonlocal term holds on the discrete level. 
Now, letting $\psi_j\in\Sh$ to be the Lagrangian basis functions, we obtain that 
$b_h(\phi,\psi_j)=(\cker^h\phi,\psi_j)_h-(\kernel\circledast\phi,\psi_j)_h$ for all $j\in\mathcal{J}_h^{\D}$ and $b_h(\phi,\psi_j)=(\mathcal{N}_h\phi,\psi_j)_{\widetilde{h}}$ for all $j\in\mathcal{J}_h^I$.

Then, the fully discrete problem reads as follows: For a given $u_h^0\in\Sh$ we seek $(w_h^k,u_h^k,\lambda_{h}^k)\in\Sho\times\Sh\times\Sho$, $k=1,\dots,K$, such that the following holds
\begin{equation}
\begin{aligned}
\frac{1}{\tt}(u_h^k-u_h^{k-1},\phi)_h+(\nabla w_h^k,\nabla\phi)_{L^2(\D)}&=0,\;\;\forall \phi\in\Sho,\\
(w_h^k,\psi)_h- {b_h(u_h^k,\psi)}+\cpot(u_h^k,\psi)_h-(\lambda_h^k,\psi)_h&=0,\;\;\forall\psi\in{\Sh},\\
\lambda_h^k=\lambda^k_{h,+}-\lambda^k_{h,-},\quad \lambda^k_{h,+}\geq 0,\quad\lambda^k_{h,-}\geq 0,\quad |u_h^k|&\leq 1,\\
\lambda_{h,\pm}^k(p_j)(u_h^k(p_j)\mp1)&=0,\;\forall p_j\in\mathcal{J}_h^{\D}.
\end{aligned}\label{CH_VI_disc}
\end{equation}
Thanks to the mass-lumping property of $b_h(\cdot,\cdot)$, we can also obtain the projection formula~\eqref{proj_ut}, for the fully discrete solution $u_h^k\in\Sh$ {of~\eqref{CH_VI_disc}}:
\[
u_h^k(p_j)=P_{[-1,1]}\left(\frac{1}{\xi_h(p_j)}g_h^k(p_j)\right),\quad\quad p_j\in\mathcal{J}_h^\D,\;\; k=1,\dots,K,
\]
where $\xi_h(p_j)=\cker^h(p_j)-\cpot>0$ and $g_h^k:=w_h^k+\kernel_h\circledast u^k_h$. The availability of the above formula is useful, as it provides an insight into the stability properties of the discrete solution. 

We note that the discrete system~\eqref{CH_VI_disc} is computationally demanding due to the inclusion of the nonlocal term $b_h(\cdot,\cdot)$ on the left-hand side of the equation. This leads to inverting a large and dense matrix at each time step in~\eqref{CH_VI_disc}. 
To overcome it, one can consider instead an implicit-explicit time stepping scheme, where we discretize the local part of $b_h(\cdot,\cdot)$ implicitly, and the remaining nonlocal part explicitly, i.e., {we replace $b_h(u_h^k,\psi)$ in~\eqref{CH_VI_disc} with the approximation}
\begin{equation}
b_h(u_h^k,\psi)\approx(\cker^hu_h^k,\psi)_h-(\kernel^h \circledast u_h^{k-1},\psi)_h.\label{eq:explicit_bh}
\end{equation}
The discrete systems~\eqref{CH_VI_disc} can be solved by the primal-dual-active set strategy (PDAS), which has been already successfully employed for {the} local Cahn-Hilliard variational inequality problem; see, e.g.,~\cite{blank2011,hintermuller2011CH}. 
Under suitable conditions, PDAS is equivalent to the semismooth Newton method, which converges locally superlinearly; for more details, see~\cite{kunisch2002}.

{
We note that, while the present work focuses on solutions that may admit jump discontinuitites, an adoption of the continuous finite element method with the mass-lumping approach allows the avoidance of instabilities associated with the approximation of sharp interfaces. Alternatively, one could also consider a discontinuous Galerkin finite element method, which has been recently explored in nonlocal settings~\cite{ChenGunzburger2011,Tian2015DG,
Du2019DG,DuYin2019DG,Du2020DG,RenAskari2017}.}

{Next} we present several numerical examples which illustrate the behaviour of the nonlocal solution of~\eqref{CH_VI_disc} and a comparative study of the nonlocal vs. local solutions of the Cahn-Hilliard variational inequality. By ``local" we mean the solution of~\eqref{eq:main} with the constant mobility and {with the same obstacle potential $F$ as defined in}~\eqref{potential_obstacle}. 

\subsection{Numerical examples}
The domain $\D$ is set to be $\Omega = (0,1)^d$, $d\in\{1,2\}$ and {$\DC$
is discretized with a quasi-uniform mesh with a number of nodes $N$ that will be  specified in each instance.} For the choice of the kernel $\kernel$, we consider a scaled Gaussian kernel, similar to the one defined in~\cite{du2018CH}:
\begin{equation}
\kerneld=\frac{4\varepsilon^2}{\pi^{n/2}(\delta/3)^{n+2}}e^{\frac{-|\x-\xx|^2}{(\delta/3)^2}},\quad\x\in\R^n,\quad\delta>0.\label{kernel_gaussian}
\end{equation}
{For} $\delta\to 0$ the nonlocal operator $B$ recovers the Laplace operator $-\varepsilon^2\upDelta$, see, e.g.,~\cite{du2018CH,tian2020review}. {We additionally truncate the kernel at $|\x-\xx|>\delta$ to ensure that we have a finite extent of nonlocal interactions defined by $\delta$. Due to the exponential decay in the kernel~\eqref{kernel_gaussian}, this truncation introduces a negligible error in the corresponding solution. }
For {\it Case~1} the constant $\cker$ in~\eqref{cker} {is approximately equal} $\cker\approx36\varepsilon^2/\delta^2$. The parameters $\varepsilon,\cpot$ and $\delta$ are specified for each case separately. 

\subsection*{Example 1a}
Let $\D=(0,1)$ and we consider the ``Neumann'' type nonlocal operator $B$, defined as in \textit{Case~1}. We set $T=2$, $\tt=2\cdot 10^{-4}$, $\delta=0.25$, and $\xi=\cker-\cpot=0.008$, that corresponds to $\varepsilon^2= 0.00175$ and $\cpot=1$. The initial condition is chosen as $u_0=0.1(\sin(2\pi\x)+\sin(3\pi\x))$.
In Figure~\ref{fig:1a} we depict the nonlocal and local solutions of the Cahn-Hilliard variational inequality at different time instances.
\begin{figure}[ht!]
\begin{subfigure}{.45\textwidth}
  \centering
  \includegraphics[width=\textwidth]{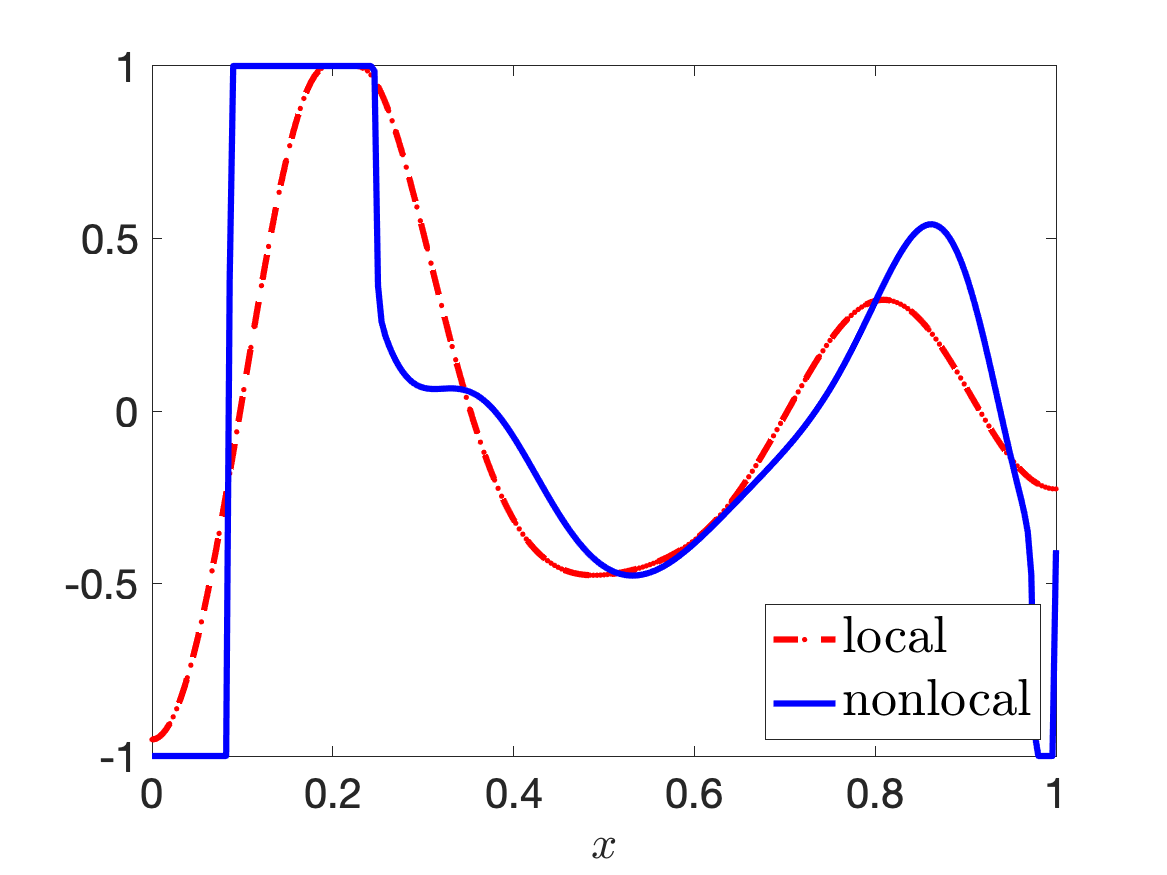}
  \caption*{$t=0.02$}
 \end{subfigure}
 \begin{subfigure}{.45\textwidth}
  \centering
  \includegraphics[width=\textwidth]{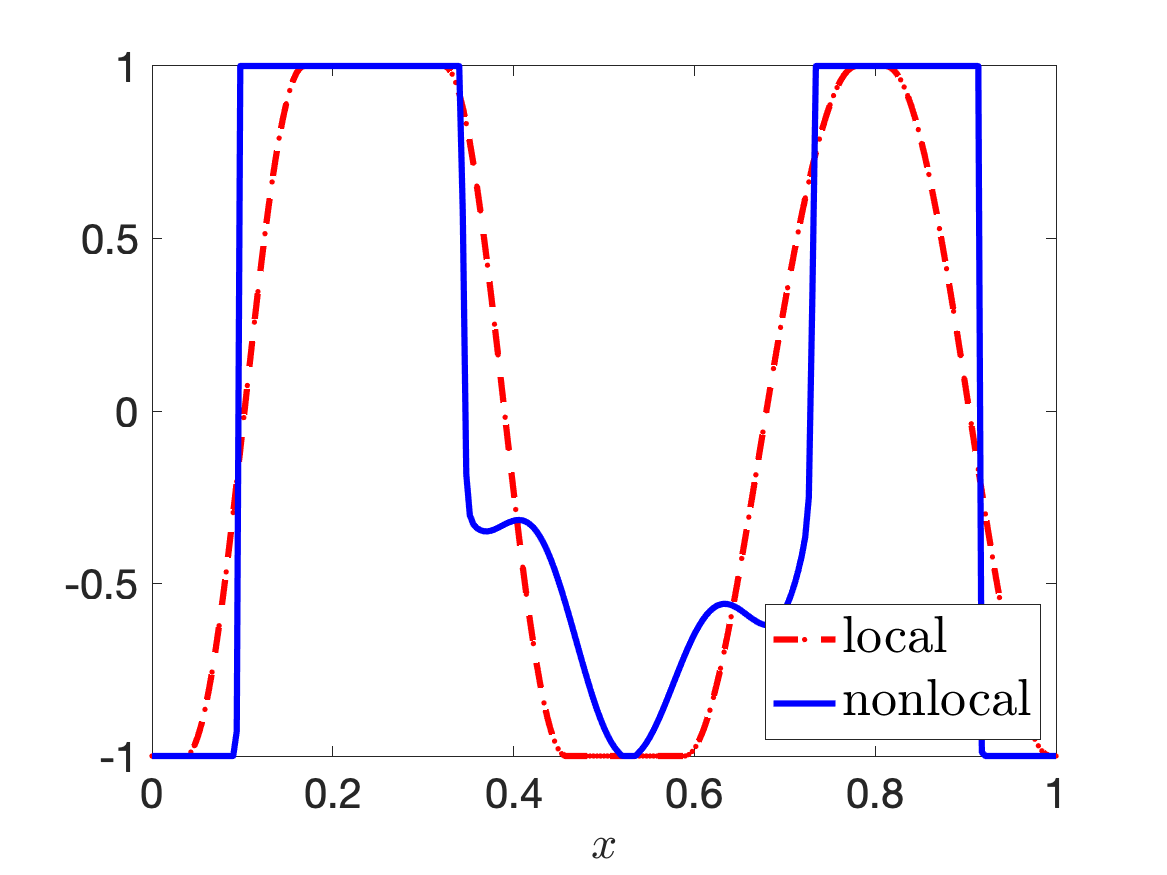}
  \caption*{$t=0.03$}
 \end{subfigure}\\
 \begin{subfigure}{.45\textwidth}
  \centering
  \includegraphics[width=\textwidth]{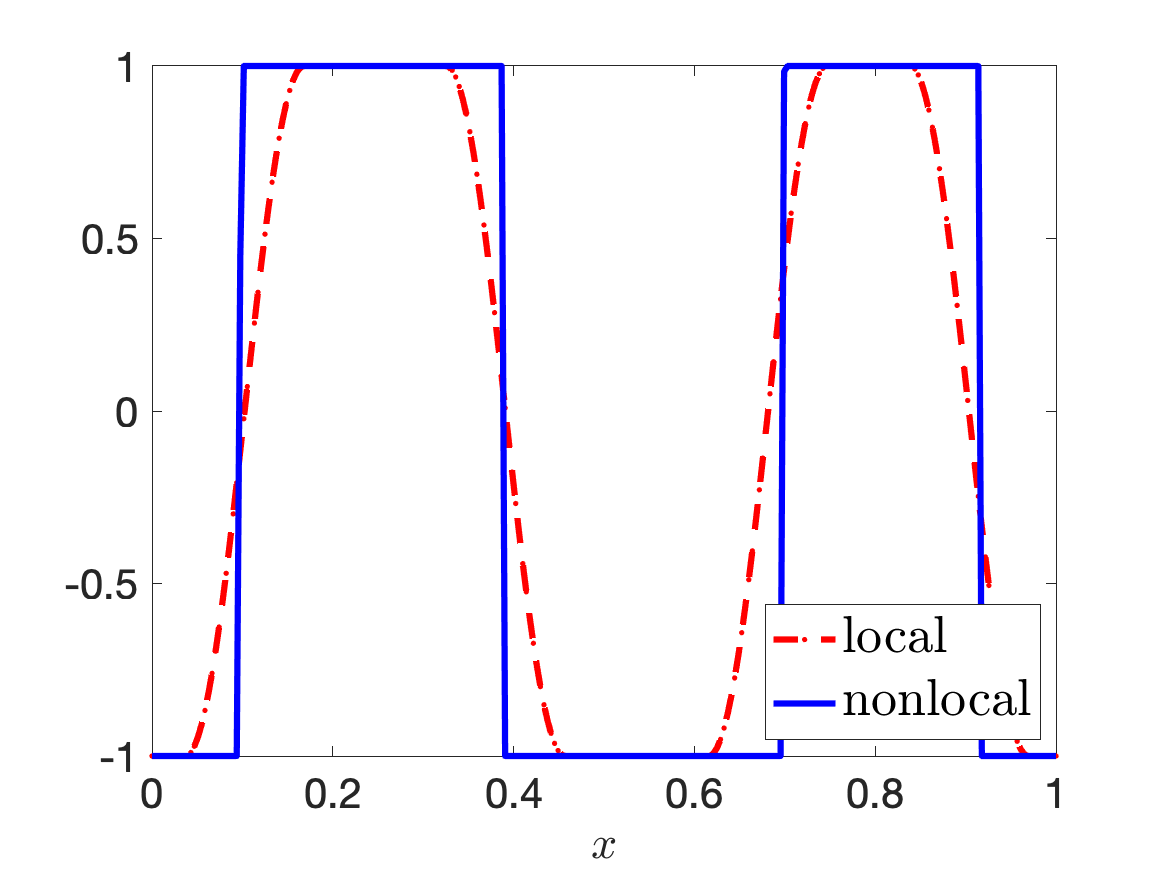}
  \caption*{$t=0.06$}
 \end{subfigure}
 \begin{subfigure}{.45\textwidth}
  \centering
  \includegraphics[width=\textwidth]{ex1/1dplot_ex_neumann_tk10000}
  \caption*{$t=2$}
 \end{subfigure}
 \caption{Evolution of the nonlocal ({\it Case 1}) and local solutions of the Cahn-Hilliard variational inequality at different time instances {for Example~1a}.}
 \label{fig:1a}
\end{figure}
We can clearly observe that the nonlocal solution has jump-discontinuities ({up to the resolution of} the employed computational grid) and admits mostly pure states, while the local solution also takes values in $(-1,1)$ even for large times. These observations are in agreement with our theoretical results, where we know that for $\xi=0$ the solution can admit pure phases throughout the domain. 

\subsection*{Example 1b}
Next, we are going to investigate the behaviour of the solution corresponding  to the ``regional'' type nonlocal operator, defined in \textit{Case~2}. We keep the same settings as in the previous example, apart from $\xi$, which is no longer constant in the present case. We set $\xi_{\min}=\min\{\xi(\x)\}=0.008$, {$\x\in\D$, which} corresponds to $\varepsilon^2=0.00175$ and $\cpot=0.4960$.
In Figure~\ref{fig:1b} we depict the evolution of the solution at different time instances. In contrast to the previous example, we are no longer observing sharp interfaces of the solution. This is explained by the fact that $\xi=\xi(x)$ is spatially dependent and does not vanish for all $\xi\in\D$. Furthermore, we observe {that the phase transitions that occur near the boundary of $\Omega$ have steeper gradients compared to the phase transitions appearing in the interior of the domain. This} is explained by the fact that $\xi(\x)$ is the smallest and equal to $\xi_{\min}$ in the boundary nodes {and according to the theory this leads to sharper interfaces in the solution}. 
\begin{figure}[ht!]
\begin{subfigure}{.45\textwidth}
  \centering
  \includegraphics[width=\textwidth]{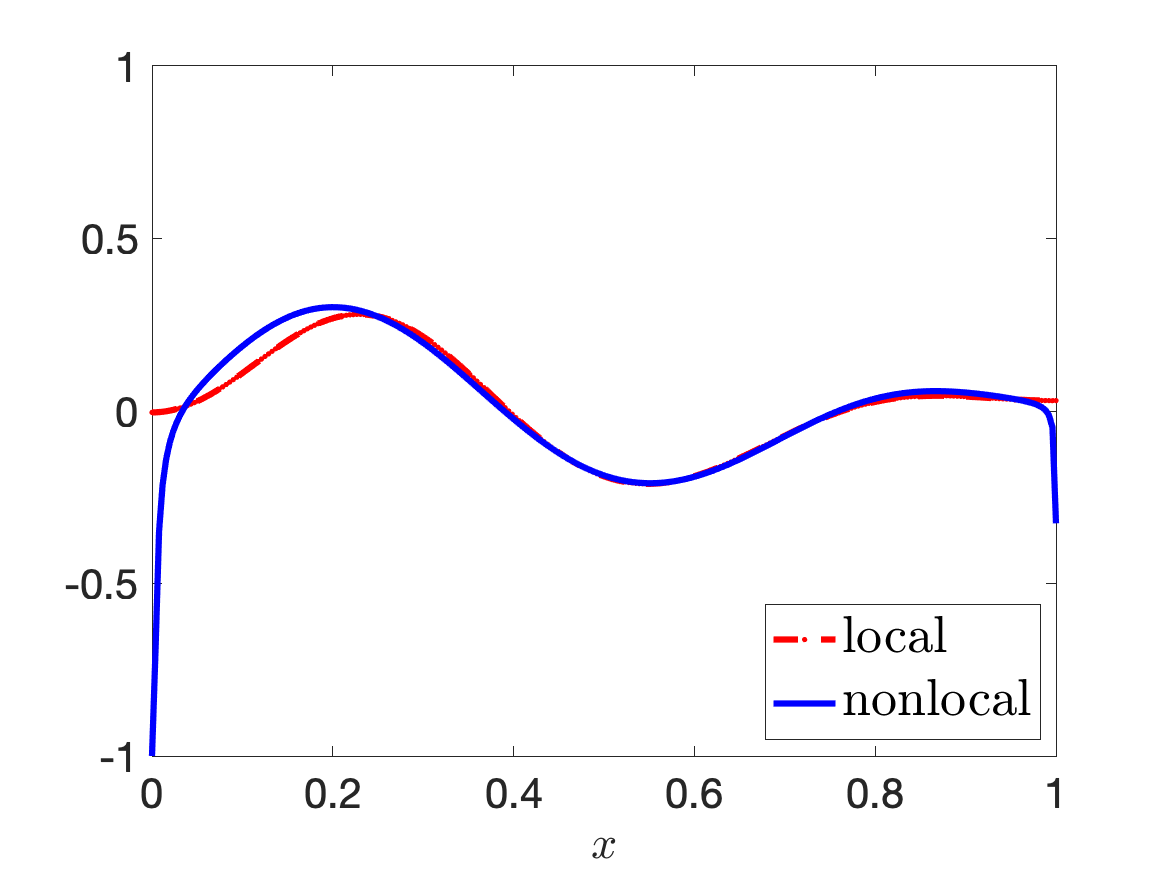}
  \caption*{$t=0.02$}
 \end{subfigure}
 \begin{subfigure}{.45\textwidth}
  \centering
  \includegraphics[width=\textwidth]{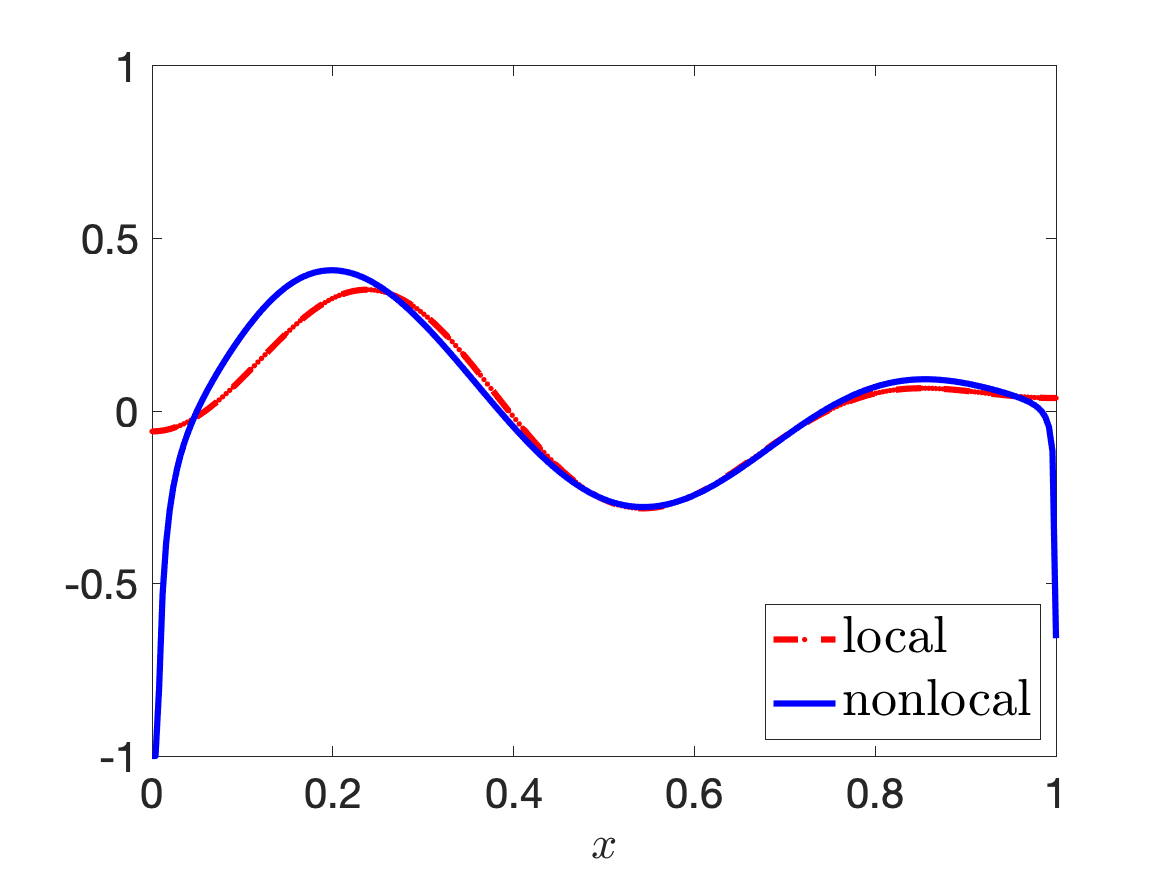}
  \caption*{$t=0.03$}
 \end{subfigure}\\
 \begin{subfigure}{.45\textwidth}
  \centering
  \includegraphics[width=\textwidth]{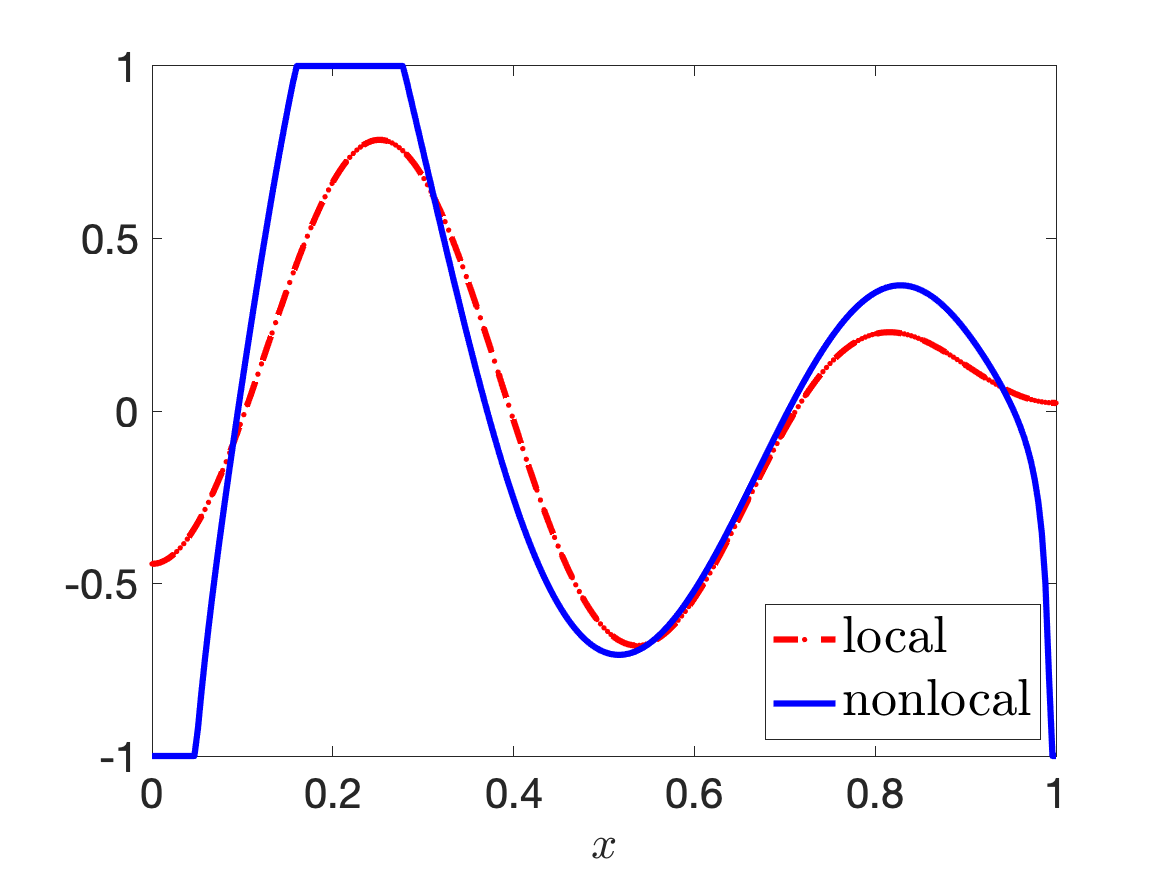}
  \caption*{$t=0.06$}
 \end{subfigure}
 \begin{subfigure}{.45\textwidth}
  \centering
  \includegraphics[width=\textwidth]{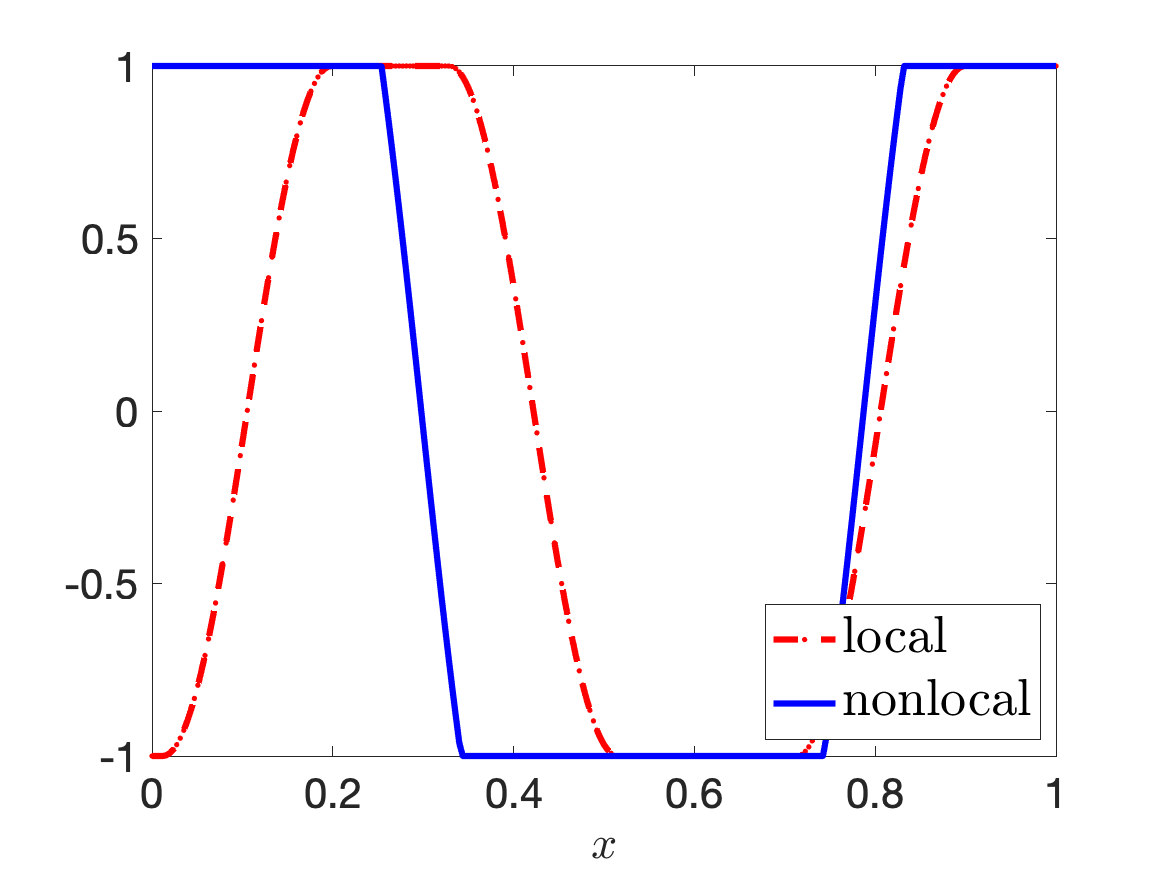}
  \caption*{$t=2$}
 \end{subfigure}
 \caption{Evolution of the nonlocal ({\it Case~2}) and local solutions of the Cahn-Hilliard variational inequality at different time instances {for Example~1b}.}
 \label{fig:1b}
\end{figure}
{
\subsection*{Example 1c}
Now, we again consider the ``regional'' type nonlocal operator, however set $\cpot=\cpot(\x)$ to be a spatially dependent coefficient such that $\xi(\x)=\cker(\x)-\cpot(\x)$ is constant throughout $\D$. In particular, taking $\varepsilon^2=0.00175$ and using the same settings as in Example~1a, we consider $\cpot(\x) = \cker(\x)-0.008$. This leads to $\xi=0.008$ being constant in $\D$. While such modification allows us to get constant values of $\xi$ uniformly in $\D$, this changes a potential, i.e., $F_0(u)=(c_f(x)/2)(1-u^2(\x))$, which, in turn, also leads to a modification of the underlying problem. 
In Figure~\ref{fig:1c} we depict the snapshots of the solution at different time instances. Here, in contrast to the previous example, we are able to get sharp interfaces in the solution similar as in the case of the ``Neumann'' type nonlocal operator. However, in contrast to the ``Neumann'' type problem, this comes with the need to use a non-typical double-well potential.
\begin{figure}[ht!]
\begin{subfigure}{.45\textwidth}
  \centering
  \includegraphics[width=\textwidth]{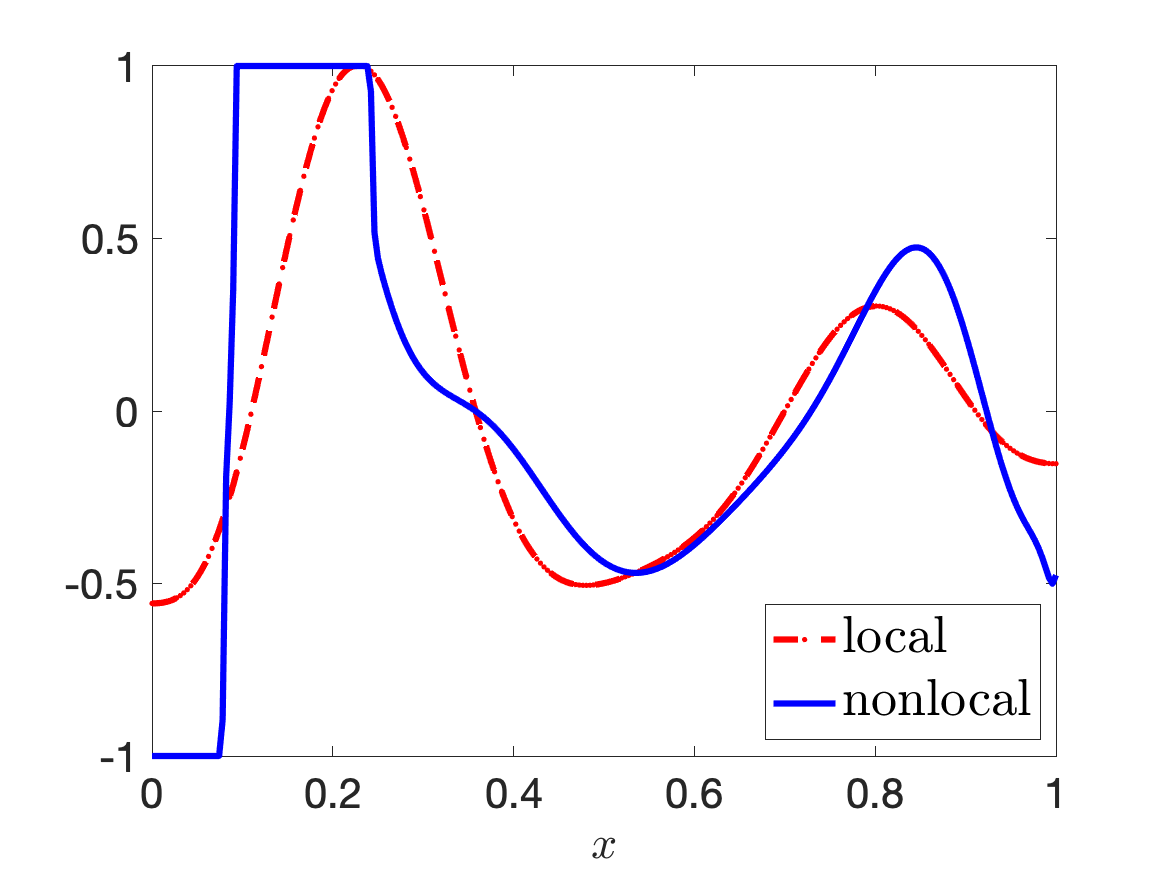}
  \caption*{$t=0.02$}
 \end{subfigure}
 \begin{subfigure}{.45\textwidth}
  \centering
  \includegraphics[width=\textwidth]{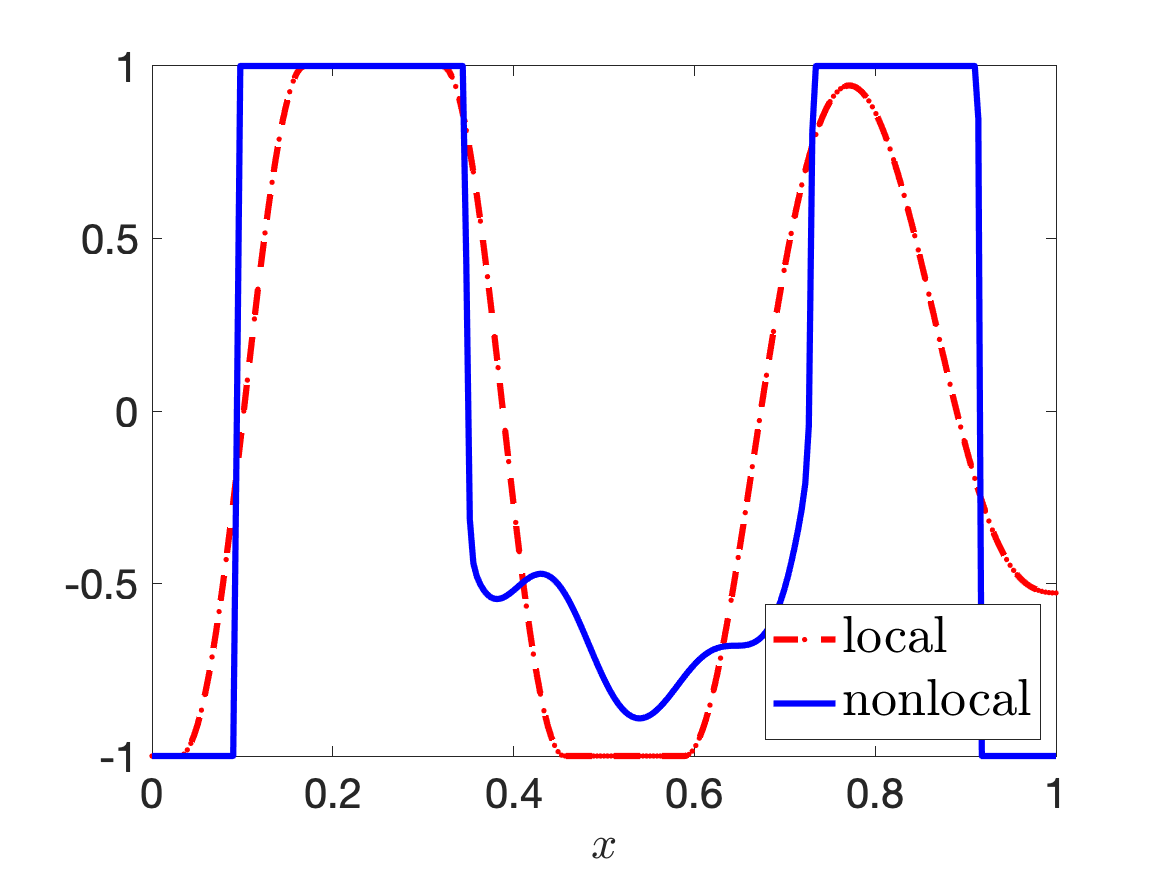}
  \caption*{$t=0.03$}
 \end{subfigure}\\
 \begin{subfigure}{.45\textwidth}
  \centering
  \includegraphics[width=\textwidth]{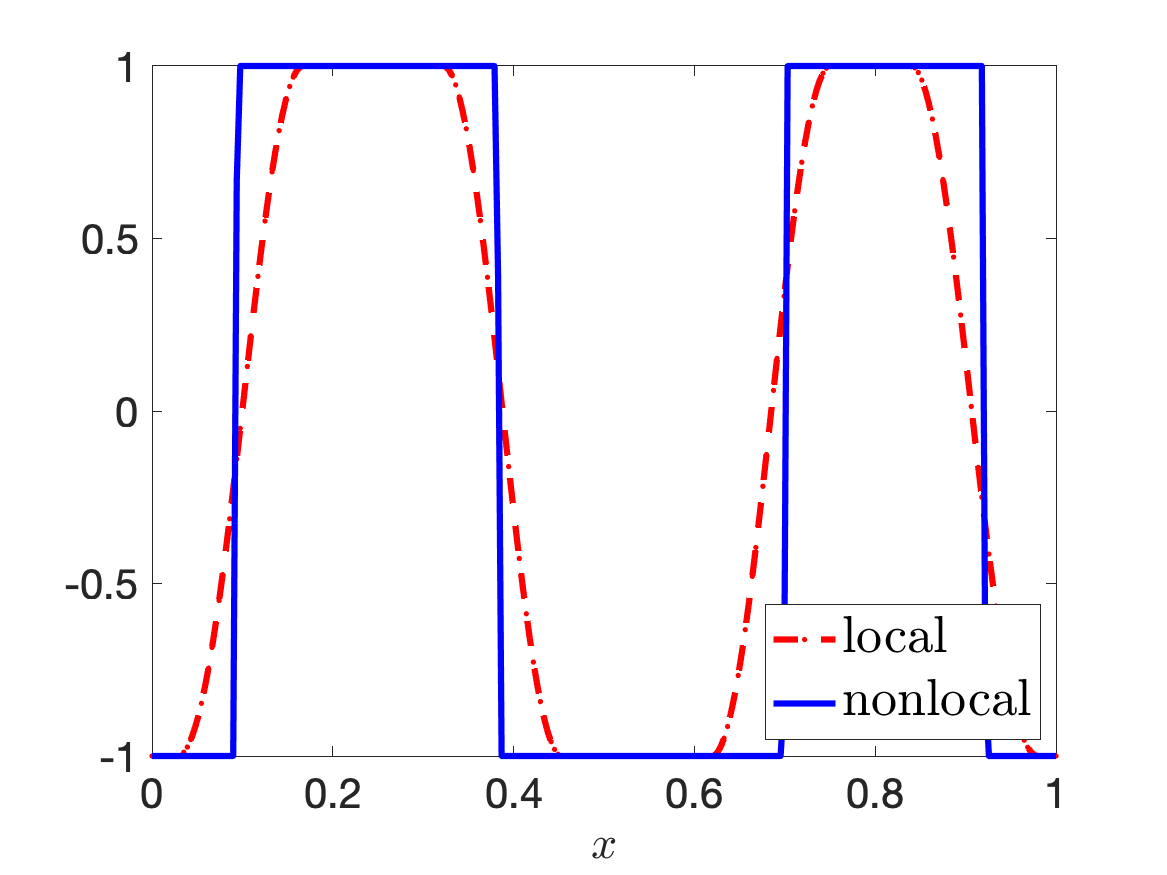}
  \caption*{$t=0.06$}
 \end{subfigure}
 \begin{subfigure}{.45\textwidth}
  \centering
  \includegraphics[width=\textwidth]{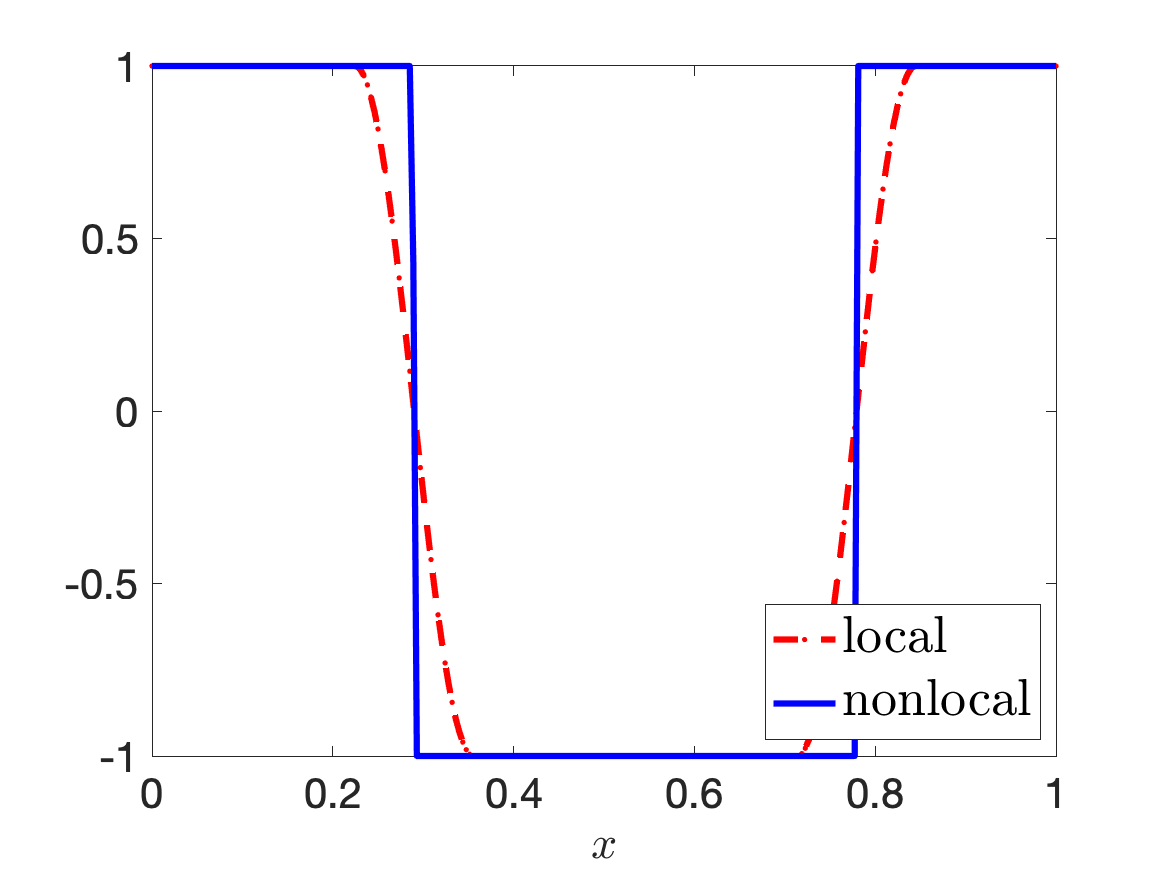}
  \caption*{$t=2$}
 \end{subfigure}
 \caption{Evolution of the nonlocal ({\it Case~2}) and local solutions of the Cahn-Hilliard variational inequality at different time instances and with $\cpot=\cpot(x)$ for Example~1c.}
 \label{fig:1c}
\end{figure}
}

{
\subsection*{Example 1d}
Now, using the``Neumann'' type nonlocal operator we investigate the variation of the solution with respect to the nonlocal interface parameter $\xi$. We keep the same settings as in Example~1a, and to obtain different values of $\xi$ we vary either the coefficient $\cpot$ or a nonlocal interaction radius $\delta$. The corresponding snapshots are depicted in Figure~\ref{fig:1d}. Here, we can observe that the width of the interface is changing with respect to $\xi$. In particular,  for $\xi=0$ that corresponds either to $\cpot=1.56$ or $\delta=0.25$, we can see that the solution admits sharp interfaces, which is also in agreement with the previous examples, and for a larger $\xi$ the interface becomes more diffuse and the nonlocal solution conforms {more closely} to the corresponding local solution. We also observe that for a smaller extent of nonlocal interactions $\delta$ the nonlocal solutions are close to the local one and the interface becomes more diffuse, which is an expected behavior here. Similar results have been also reported in the Cahn-Hilliard case with the regular potential~\cite{du2018CH}.}

{
Overall, we could see that the parameter $\xi$ plays a role of an interface parameter in the nonlocal model, similar to the interface parameter $\varepsilon$ in the local setting. Changing the support of the kernel and the scaling of the double-well potential have a great affect on $\xi$ and, hence, on the width of the interface.
\begin{figure}[ht!]
\center{
\begin{subfigure}{.45\textwidth}
  \centering
  \includegraphics[width=1.2\textwidth]{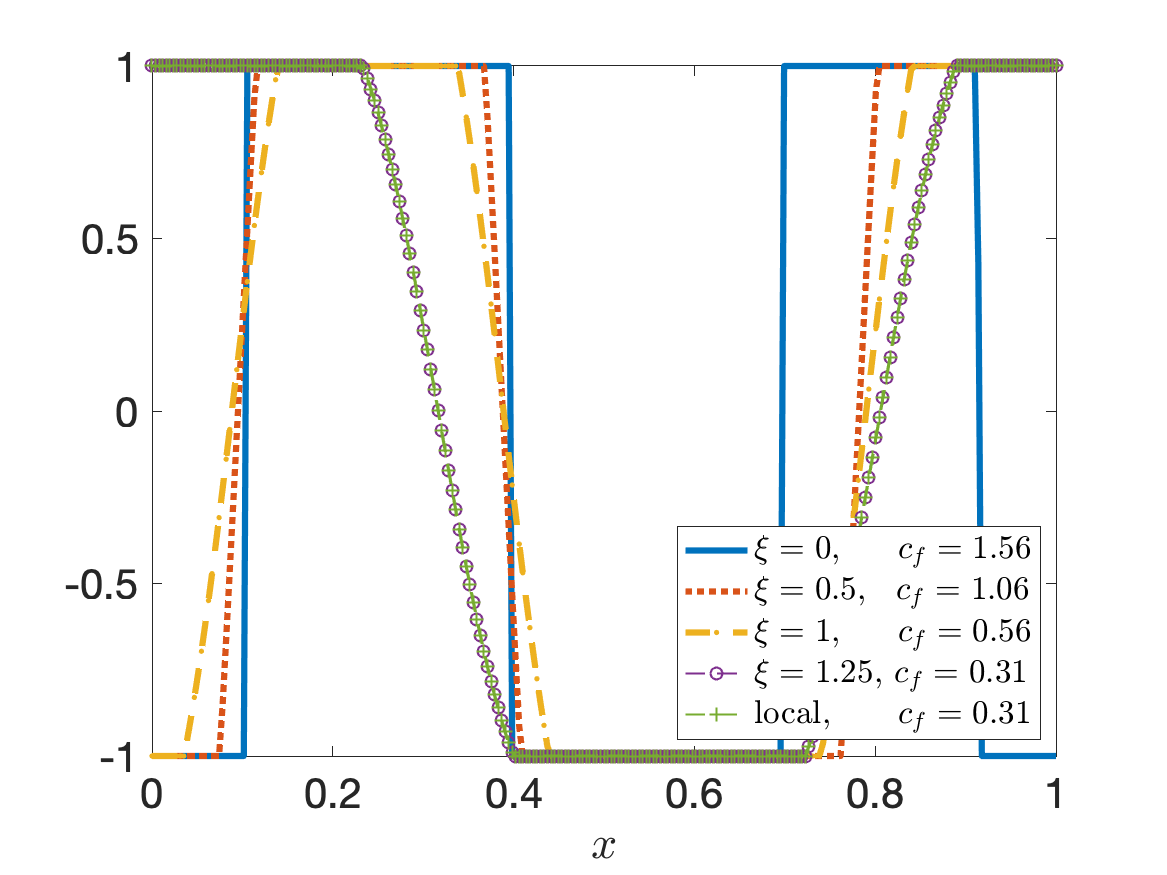}
 \end{subfigure}\hspace{2em}
 \begin{subfigure}{.45\textwidth}
  \centering
  \includegraphics[width=1.2\textwidth]{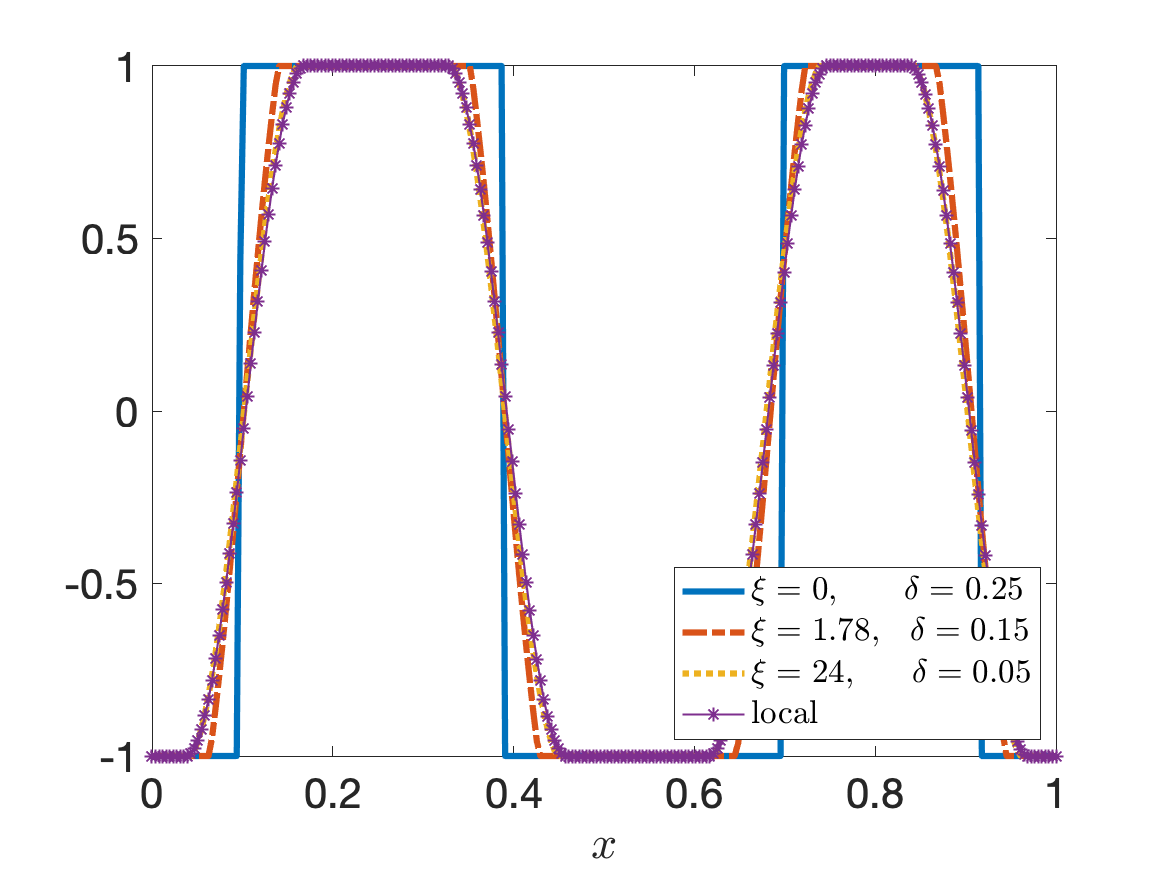}
 \end{subfigure}}
 \caption{Snapshots of the nonlocal ({\it Case~2}) and local solutions of the Cahn-Hilliard variational inequality for different values of $\cpot$ (left) and $\delta$ (right).}
 \label{fig:1d}
\end{figure}
}

Next, we conduct a comparative study for two-dimensional examples. 
\subsection*{Example 2a}
{Now, let $\D=(0,1)^2$ and consider the ``Neumann'' type nonlocal operator $B$, defined as in \textit{Case~1}, where we set $T=1$, $K=2000$, $N=39009$, 
$\delta=0.25$, $\xi=0$, $\cpot=1$, $\varepsilon^2=0.0017$. The initial condition $u_0$ is chosen as 
\[
u_0(\x)=2\left(e^{-(6\x-2.1)^2-(6\xx-3)^2}+e^{-(7\x-4.9)^2-(7\xx-3.5)^2}\right)-1.
\]
In Figure~\ref{fig:2a} we plot the snapshots of the local and nonlocal solutions at different time steps. We can observe that local and nonlocal solutions look alike quantitatively. {However, whereas the nonlocal model can describe perfectly sharp interfaces up to the resolution of the discretization mesh, the interfaces for local solution are diffuse.}
\begin{figure}[ht!]
\begin{subfigure}{.22\textwidth}
  \centering
  \includegraphics[width=1.2\textwidth]{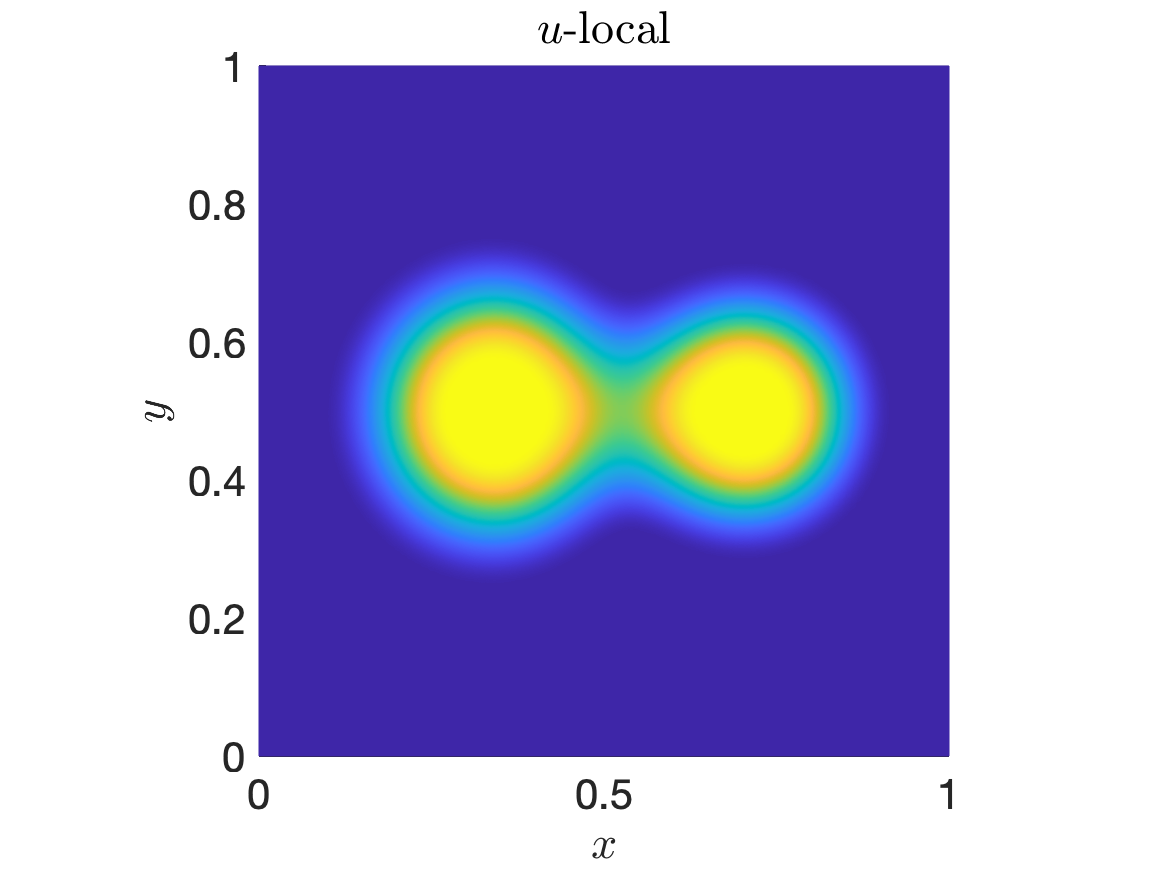}
 \end{subfigure}
 \begin{subfigure}{.22\textwidth}
  \centering
\includegraphics[width=1.2\textwidth]{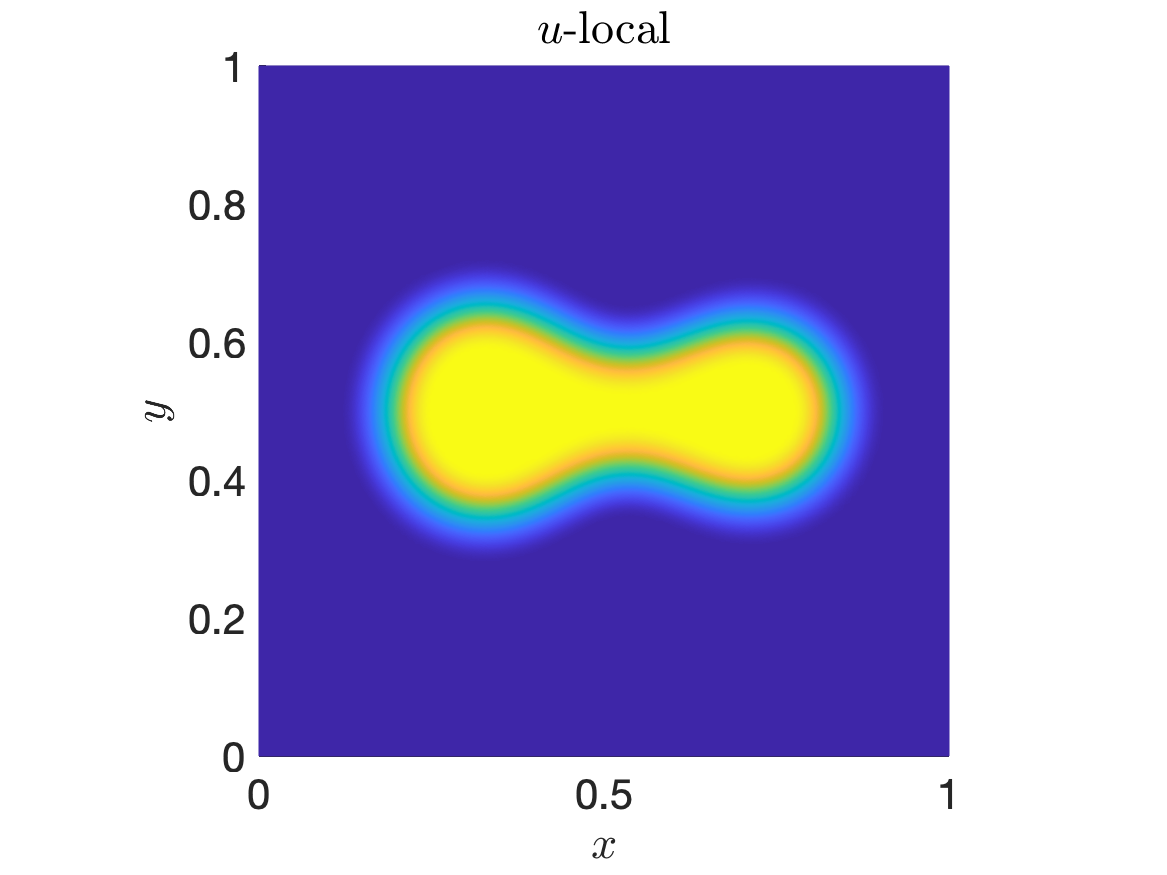}
 \end{subfigure}
 \begin{subfigure}{.22\textwidth}
  \centering
\includegraphics[width=1.2\textwidth]{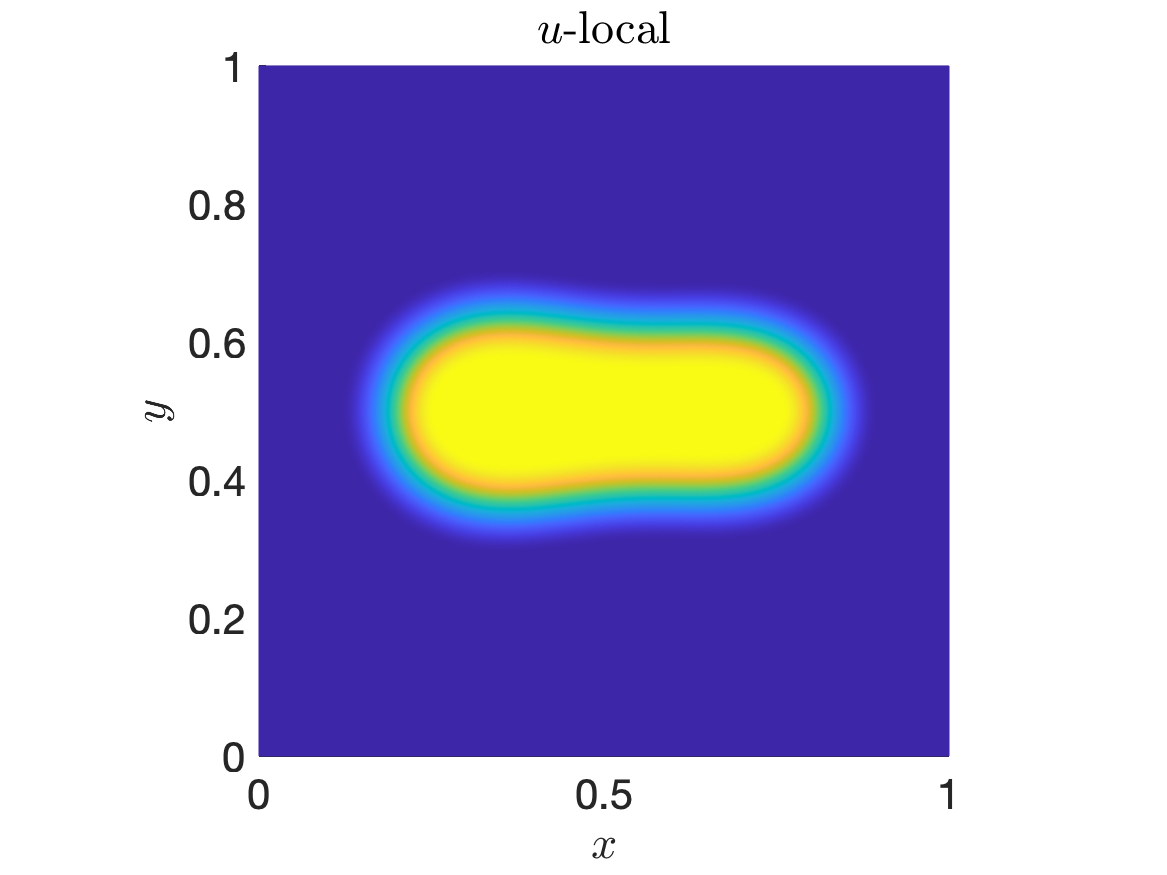}
 \end{subfigure}
 \begin{subfigure}{.22\textwidth}
  \centering
  \includegraphics[width=1.2\textwidth]{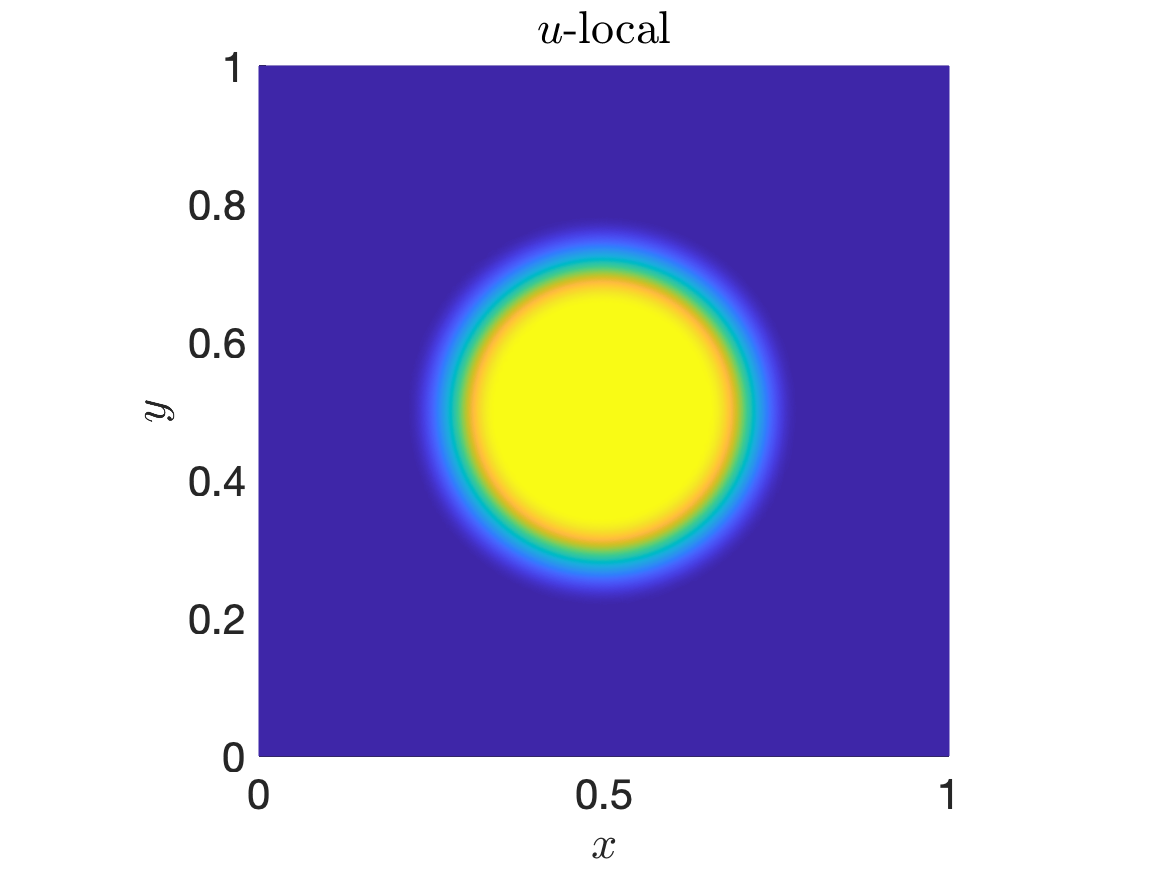}
 \end{subfigure}\\
 \begin{subfigure}{.22\textwidth}
  \centering
 \includegraphics[width=1.2\textwidth]{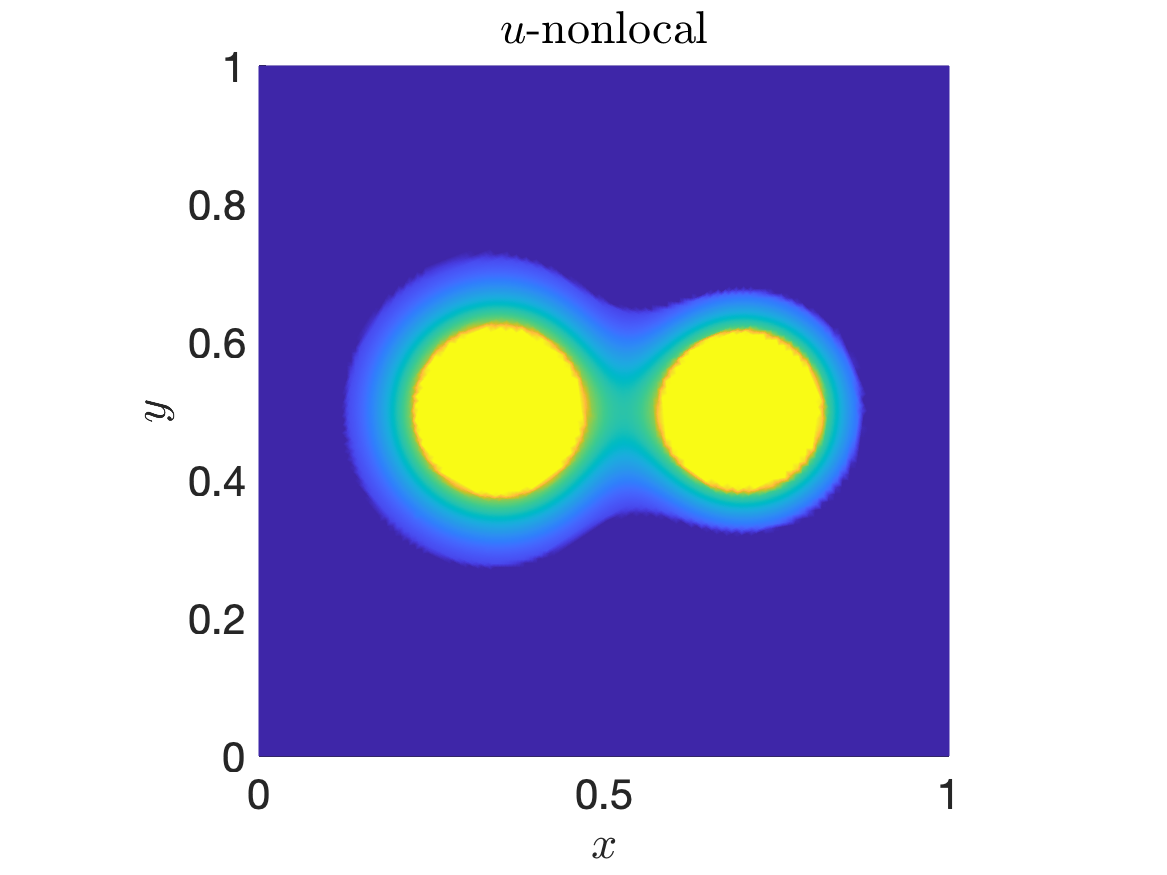}
 \end{subfigure}
 \begin{subfigure}{.22\textwidth}
  \centering
 \includegraphics[width=1.2\textwidth]{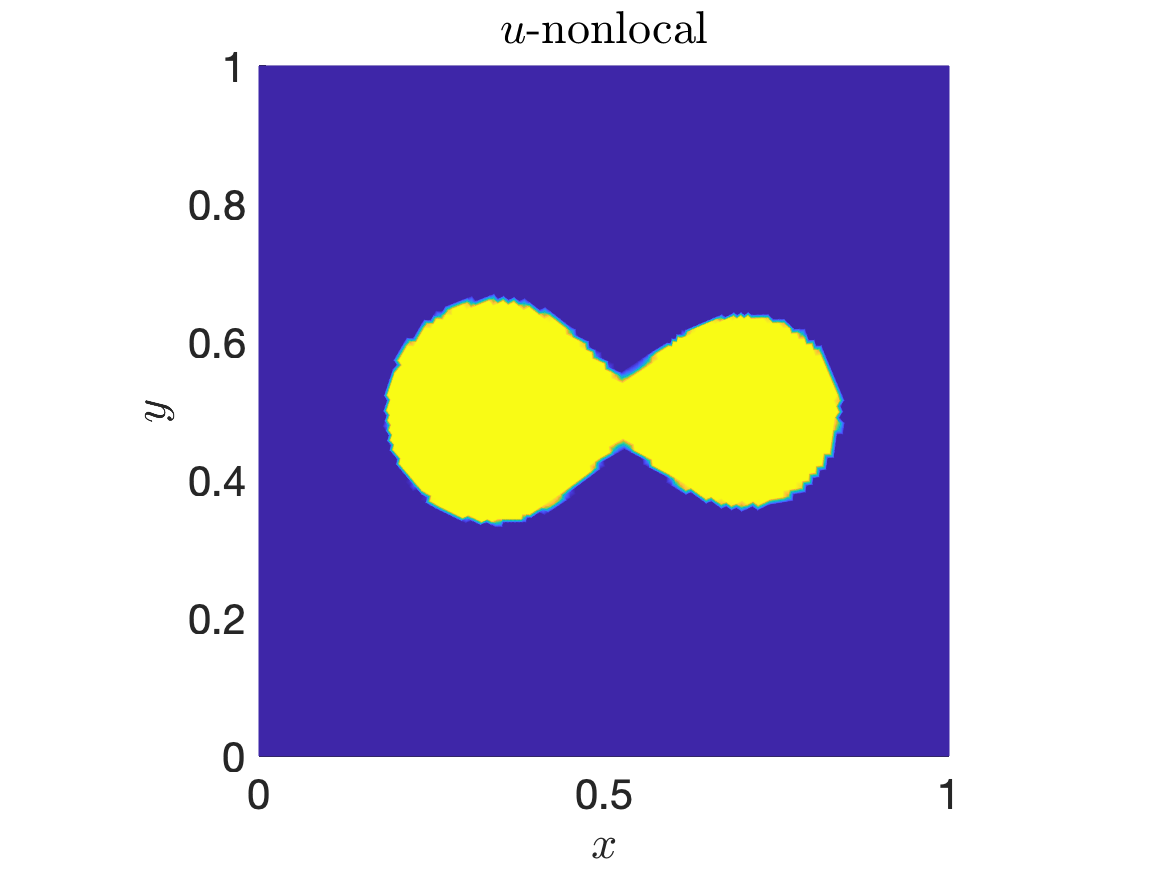}
 \end{subfigure}
 \begin{subfigure}{.22\textwidth}
  \centering
 \includegraphics[width=1.2\textwidth]{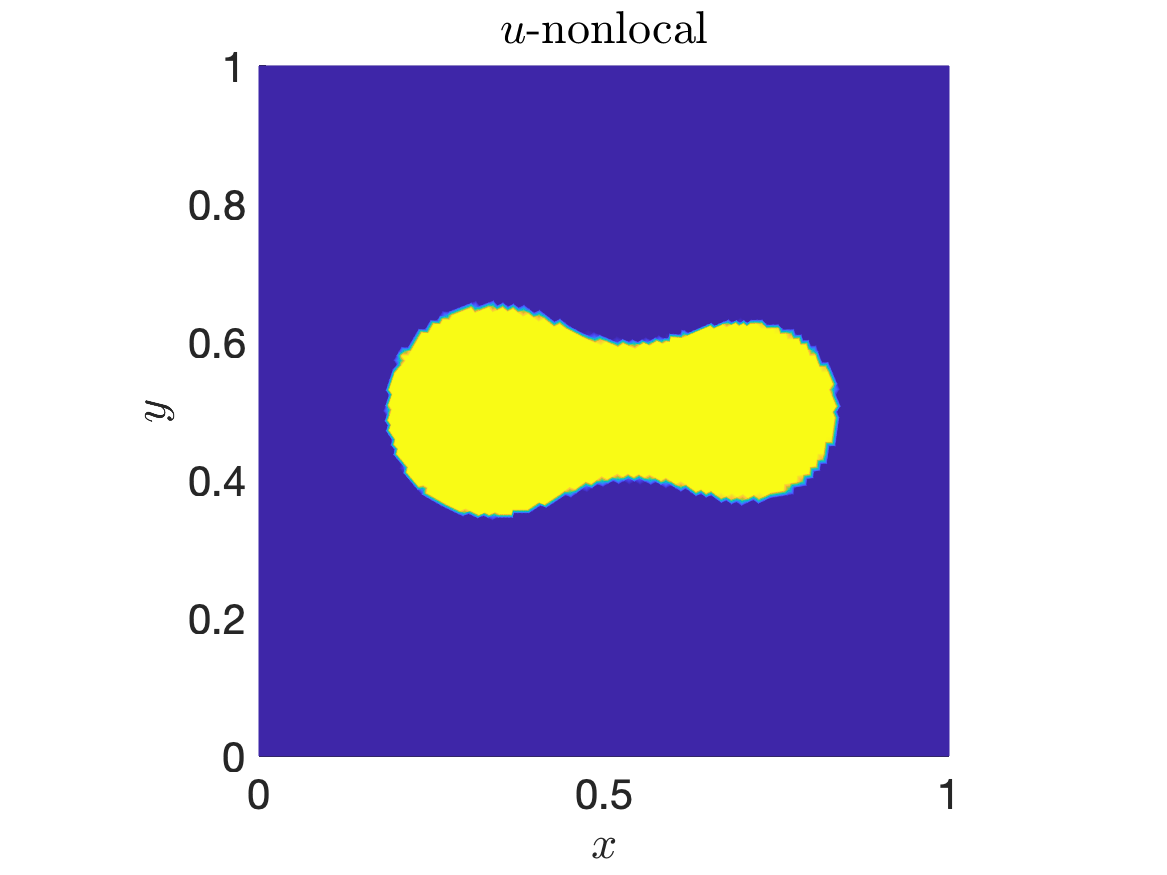}
 \end{subfigure}
  \begin{subfigure}{.22\textwidth}
  \centering
 \includegraphics[width=1.2\textwidth]{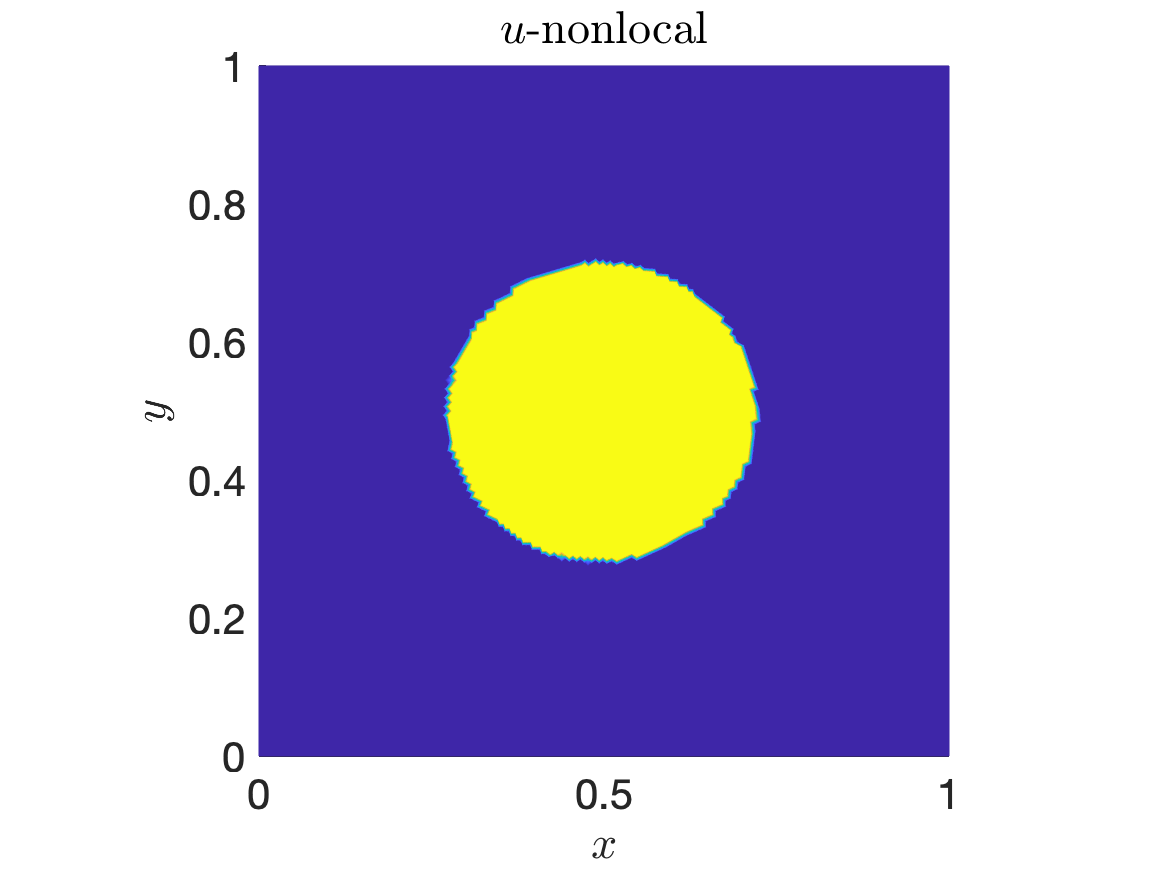}
 \end{subfigure}
 \caption{Evolution of the nonlocal ({\it Case 1}) (bottom) and local (top) solutions of the Cahn-Hilliard variational inequality at different time instances and for Example~2a. From left to right: $t=0.004, 0.009, 0.0225, 1$.}
 \label{fig:2a}
\end{figure}
}

\subsection*{Example 2b}
Now, let $\D=(0,1)^2$ and consider the ``Neumann'' type nonlocal operator $B$, defined as in \textit{Case~1}, where we set $T=2$, $K=1000$, $N=43073$, $\delta=0.1$, $\xi=0.07$, $\cpot=1$, $\varepsilon^2=0.0003$. The initial condition $u_0$ is chosen as $u_0(\x)=\tau(\x)$, where $\tau(\x)$ is drawn from a uniform random distribution on $[-1,1]$ at each grid point. 
In Figure~\ref{fig:2} we plot the snapshots of the local and nonlocal solutions at different time-steps. 
\begin{figure}[ht!]
\begin{subfigure}{.3\textwidth}
  \centering
  \includegraphics[width=\textwidth]{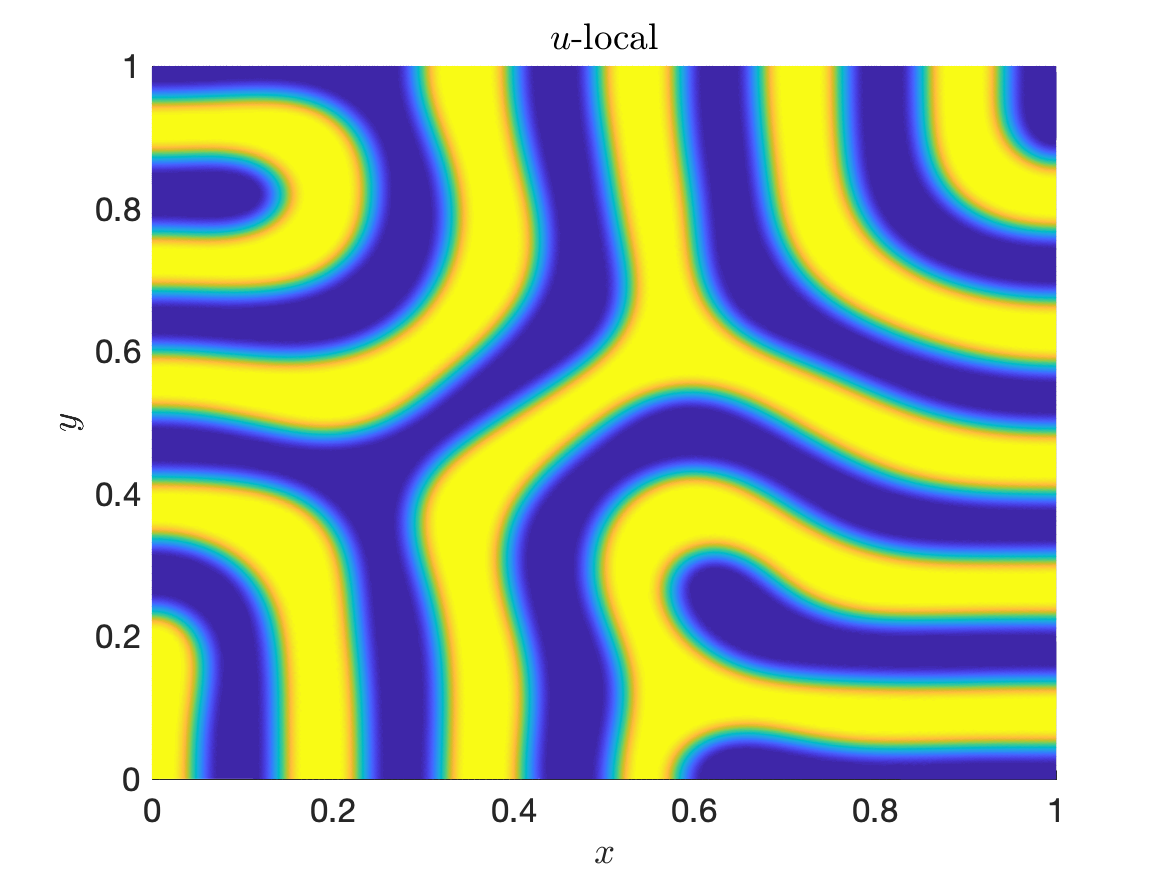}
  \caption*{$t=0.01$}
 \end{subfigure}
 \begin{subfigure}{.3\textwidth}
  \centering
\includegraphics[width=\textwidth]{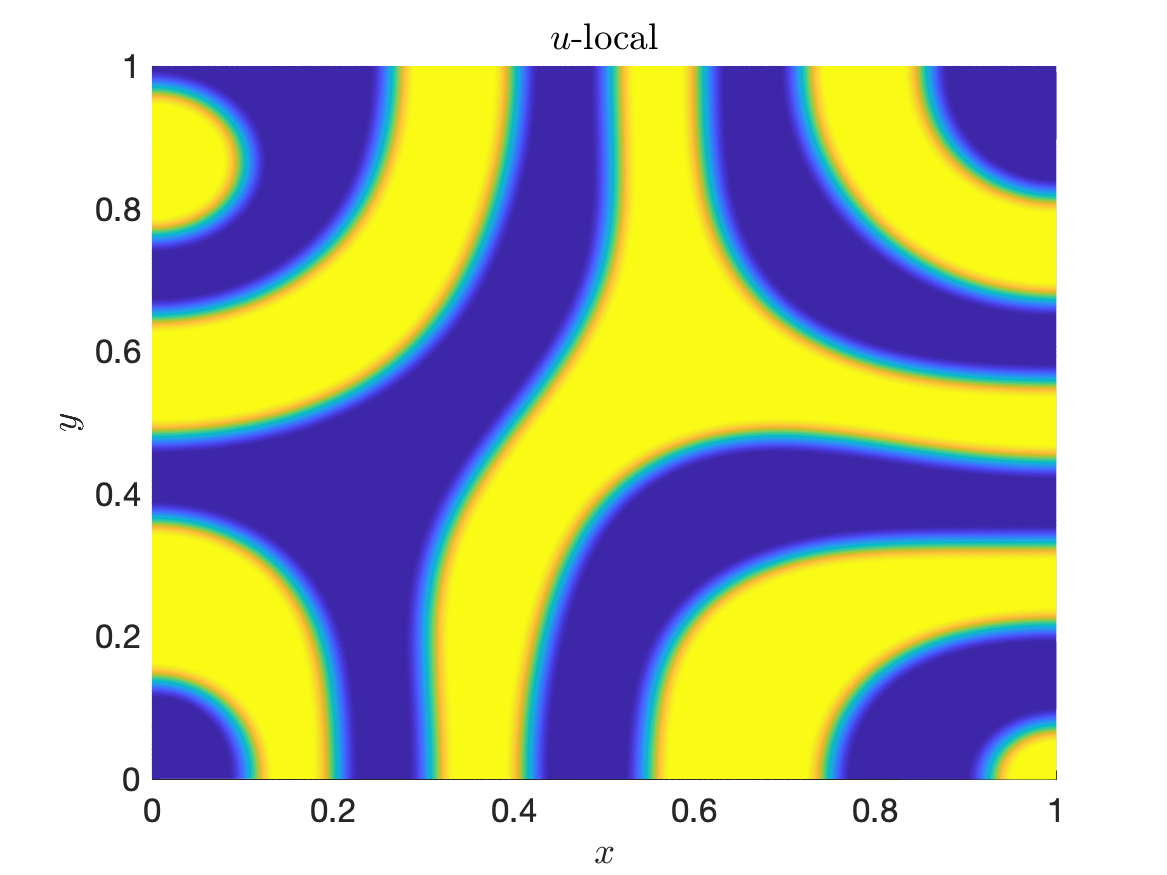}
  \caption*{$t=0.1$}
 \end{subfigure}
 \begin{subfigure}{.3\textwidth}
  \centering
\includegraphics[width=\textwidth]{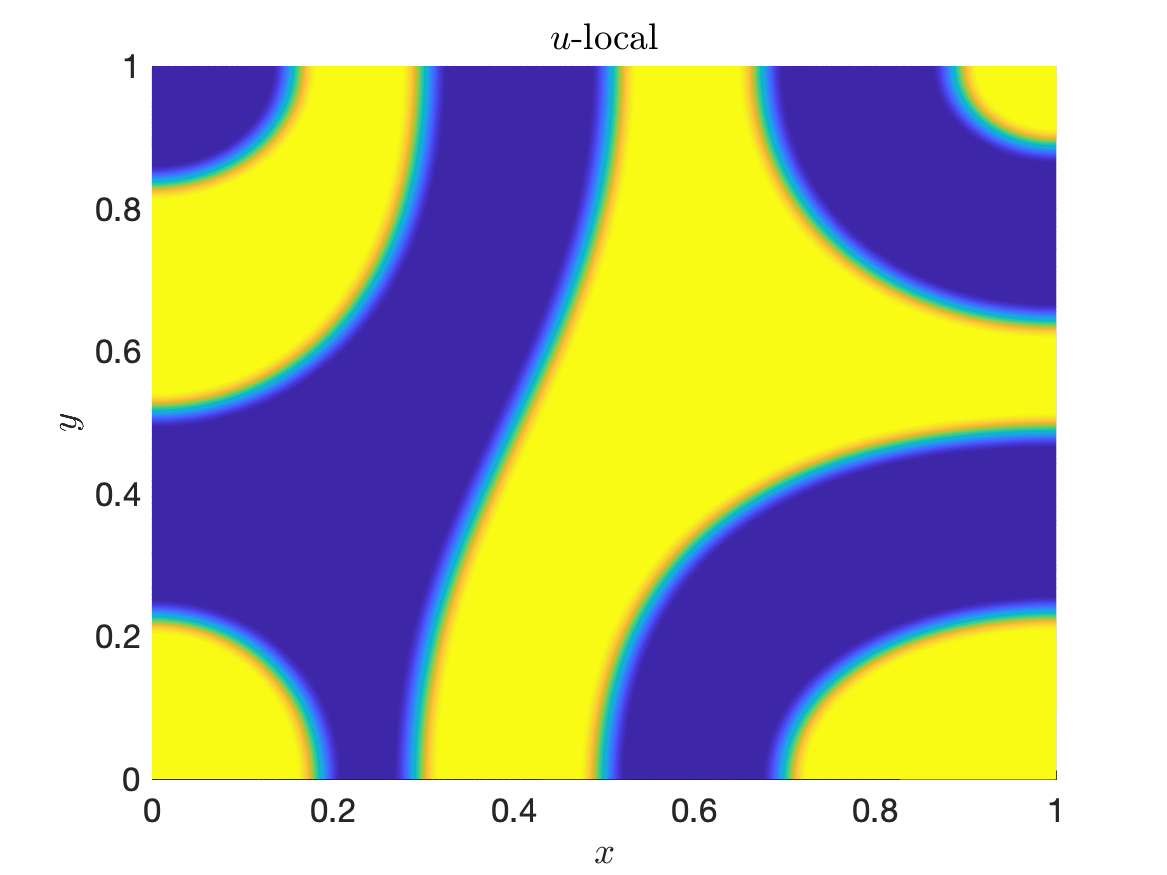}
  \caption*{$t=0.4$}
 \end{subfigure}\\
 \begin{subfigure}{.3\textwidth}
  \centering
 \includegraphics[width=\textwidth]{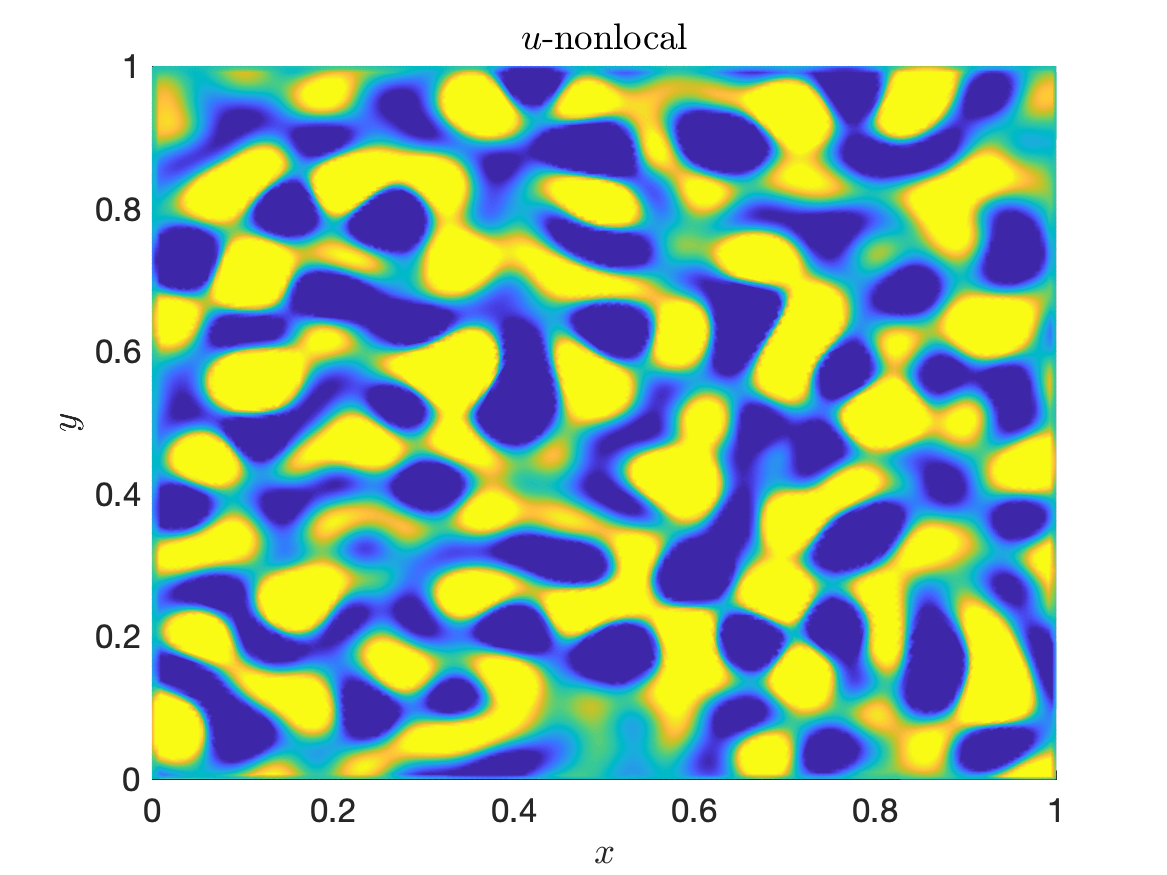}
  \caption*{$t=0.01$}
 \end{subfigure}
 \begin{subfigure}{.3\textwidth}
  \centering
 \includegraphics[width=\textwidth]{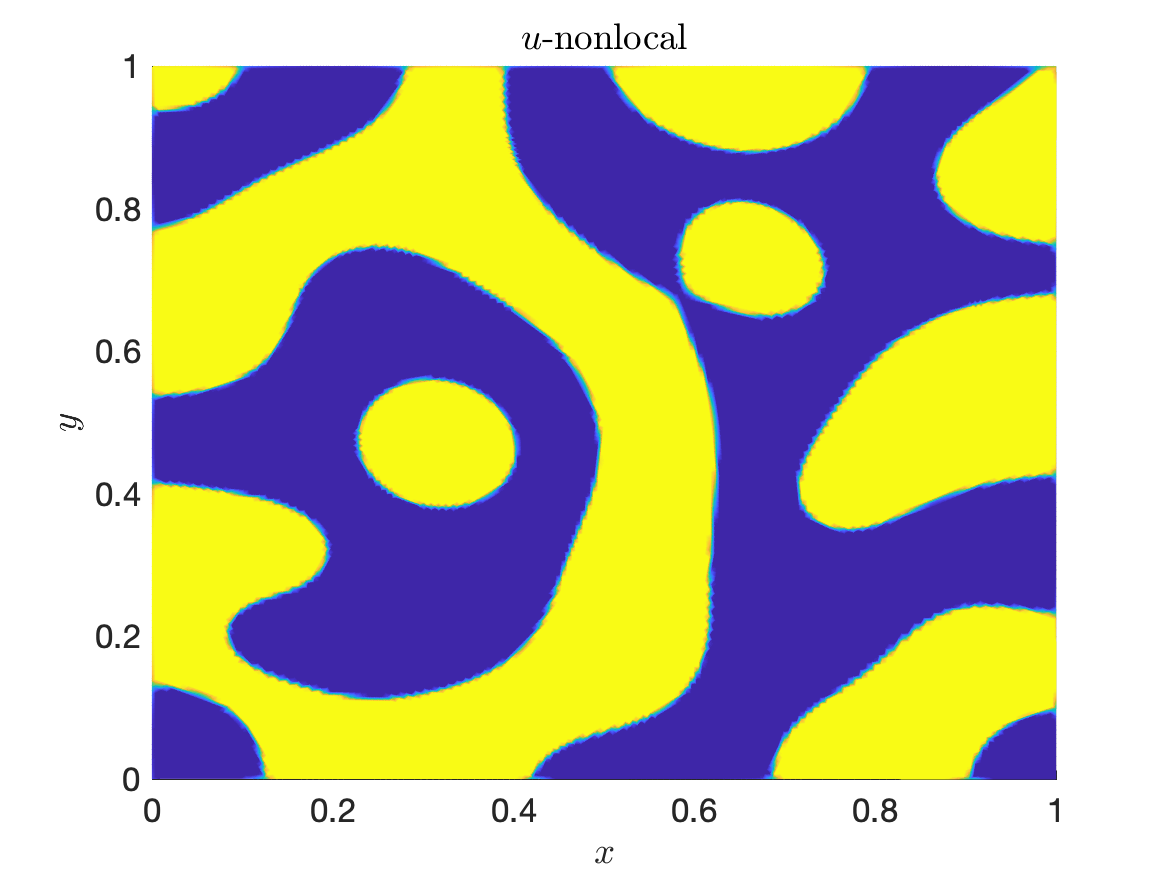}
  \caption*{$t=0.1$}
 \end{subfigure}
 \begin{subfigure}{.3\textwidth}
  \centering
 \includegraphics[width=\textwidth]{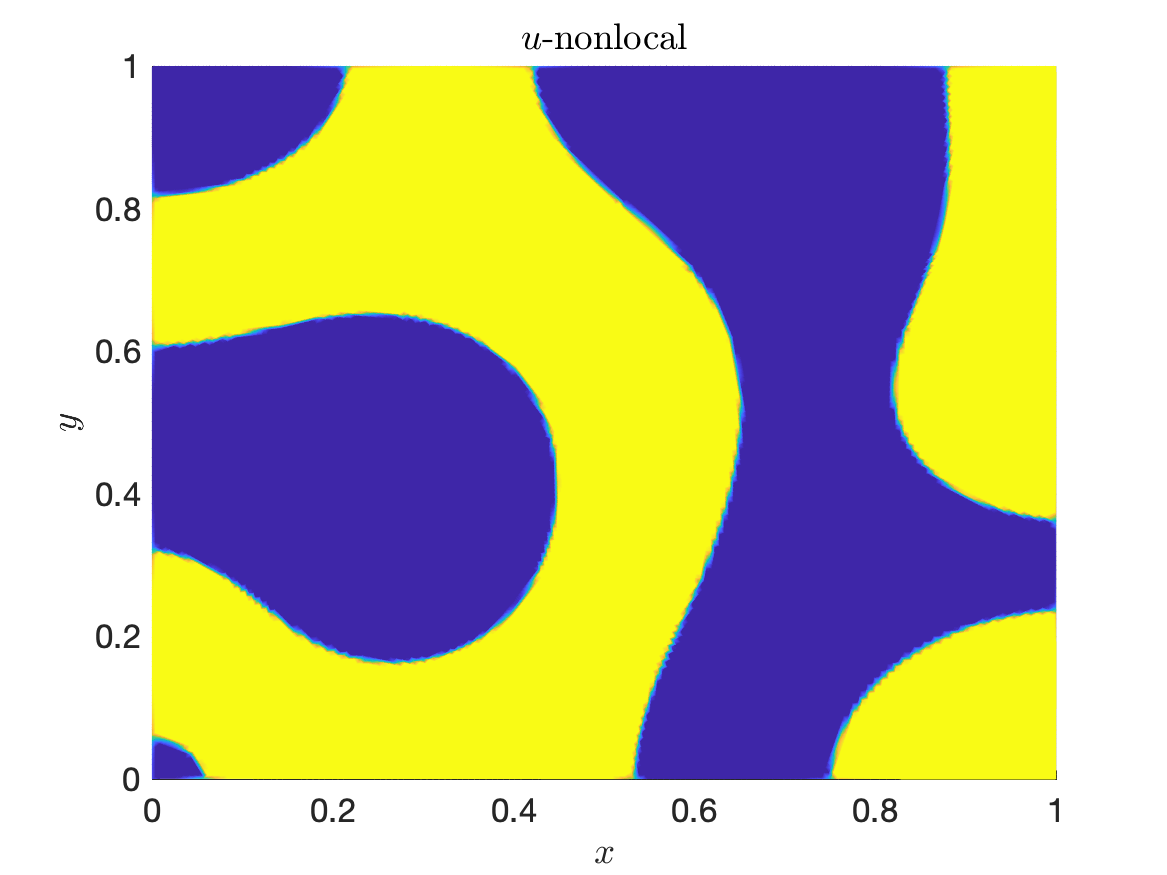}
  \caption*{$t=0.4$}
 \end{subfigure}
 \caption{Evolution of the nonlocal ({\it Case 1}) (bottom) and local (top) solutions of the Cahn-Hilliard variational inequality at different time instances for Example~2. }
 \label{fig:2}
\end{figure}
{From these plots, we observe that similarly as in the previous example, the proposed nonlocal model for \textit{Case~1} delivers sharp interfaces in the solution {whereas the local model results in diffuse interfaces}. We also notice that due the random initial condition, the patterns in the local and nonlocal solutions in the intermediate time steps differ significantly in contrast to Example~2a for which a deterministic initial condition has been used. This is due to the random and non-smooth initial condition, which has the effect that changes in the model can lead to large changes in the final solution. } 

\subsection*{Example 3}
{Next, similarly as in Example~1c, we investigate the case of the ``regional'' nonolocal operator, given as in \textit{Case~2}. }Again, we consider $\cpot=\cpot(\x)$ is a spatially dependent coefficient, such that $\xi(\x)=\cker(\x)-\cpot(x)$ is close to zero throughout a whole domain $\D$. In particular, we chose $\cpot(x):=0.9\cker(\x)$ and obtain $\param(\x)=0.1\cker(\x)$, which is positive but close to zero in $\D$. 

We set $\D=(0,1)^2$, $N=4225$, $T=1$, $K=100$, $\delta=0.3$, and $\varepsilon^2=0.004$. The initial condition $u_0$ is chosen as $u_0(\x)=\tau(\x)$, where $\tau(\x)$ is drawn from a uniform random distribution on $[-1,1]$ at each grid point.
In Figure~\ref{fig:4} we plot the nonlocal and local solutions at different time instances.  
{In this case, similarly as before, by means of modifying the double-well potential we could achieve sharper interfaces in the nonlocal solution for the ``regional'' nonlocal operator compared to the local case. This corresponds to the fact that $\param(\x)$ is very small, much smaller than $\cker(\x)$.  }
\begin{figure}[ht!]
\begin{subfigure}{.3\textwidth}
  \centering
  \includegraphics[width=\textwidth]{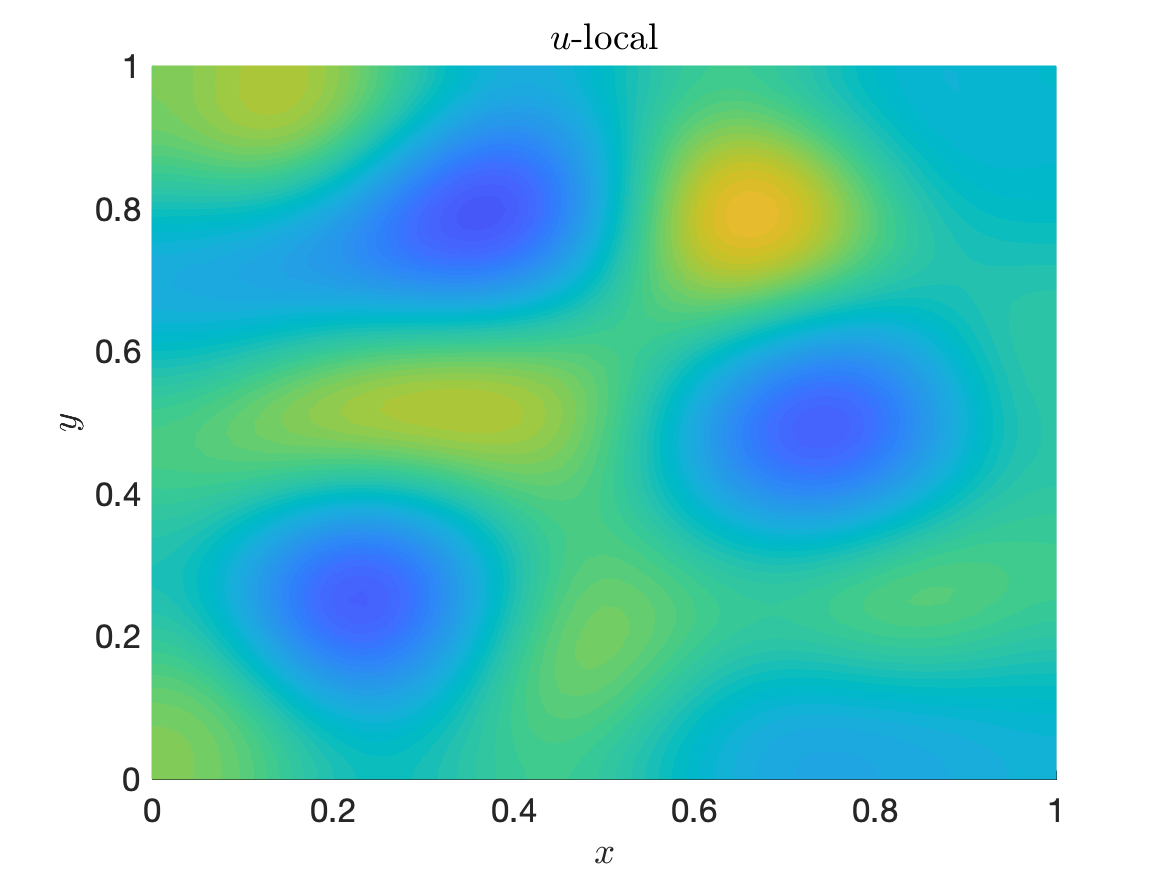}
  \caption*{$t=0.02$}
 \end{subfigure}
 \begin{subfigure}{.3\textwidth}
  \centering
\includegraphics[width=\textwidth]{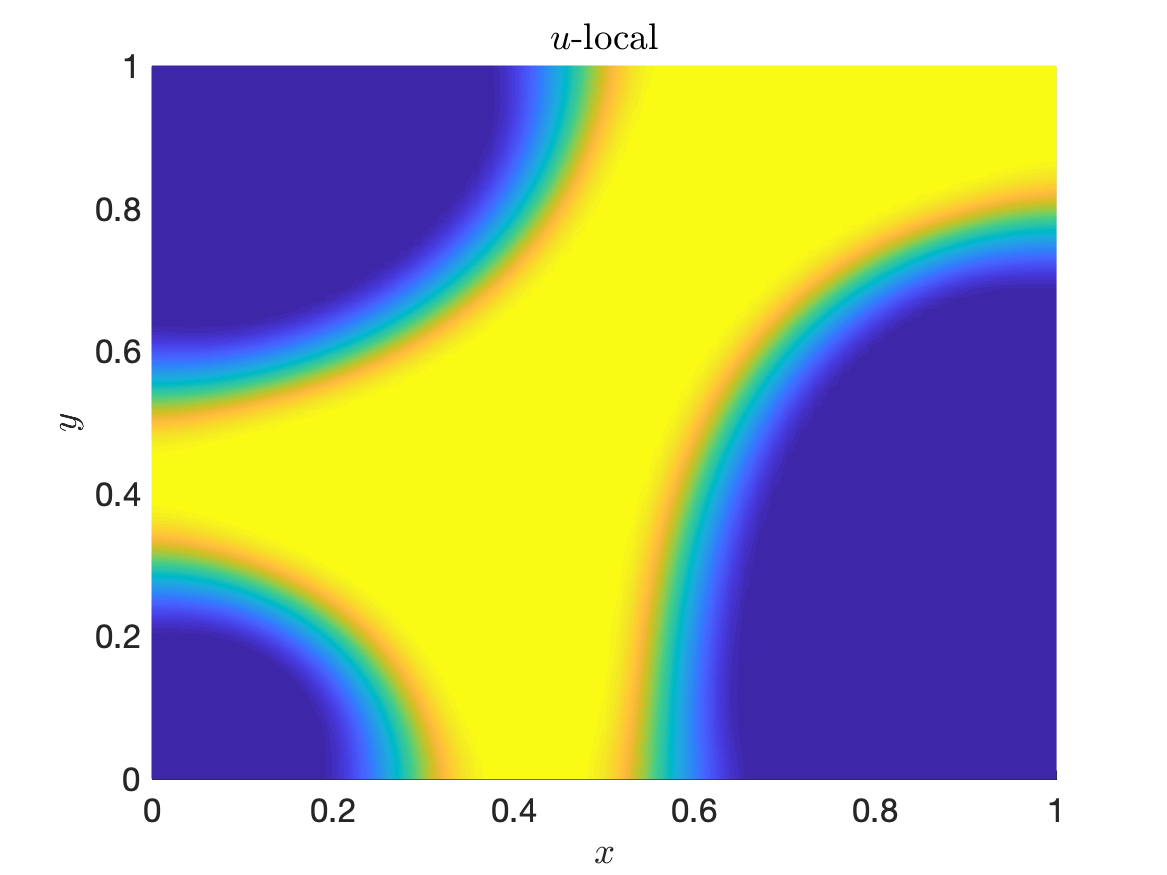}
  \caption*{$t=0.5$}
 \end{subfigure}
 \begin{subfigure}{.3\textwidth}
  \centering
\includegraphics[width=\textwidth]{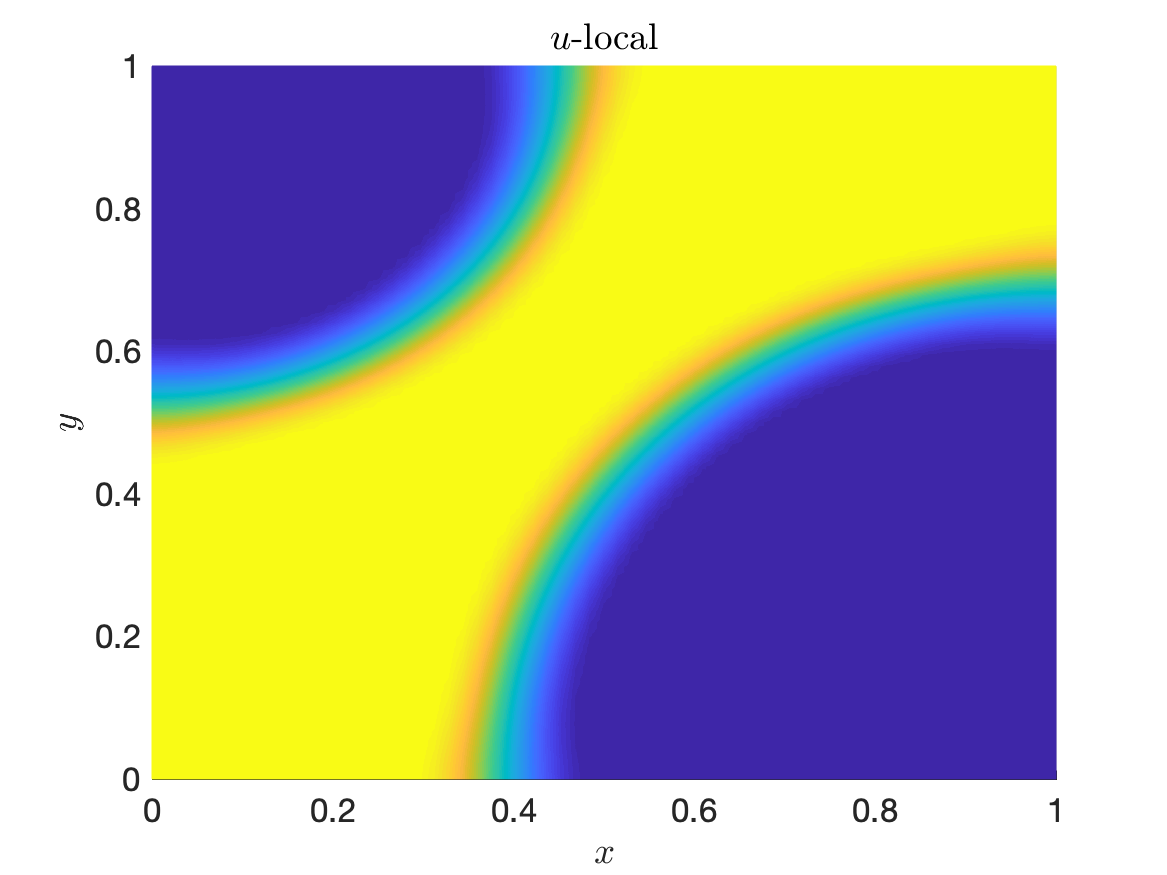}
  \caption*{$t=1$}
 \end{subfigure}\\
 \begin{subfigure}{.3\textwidth}
  \centering
 \includegraphics[width=\textwidth]{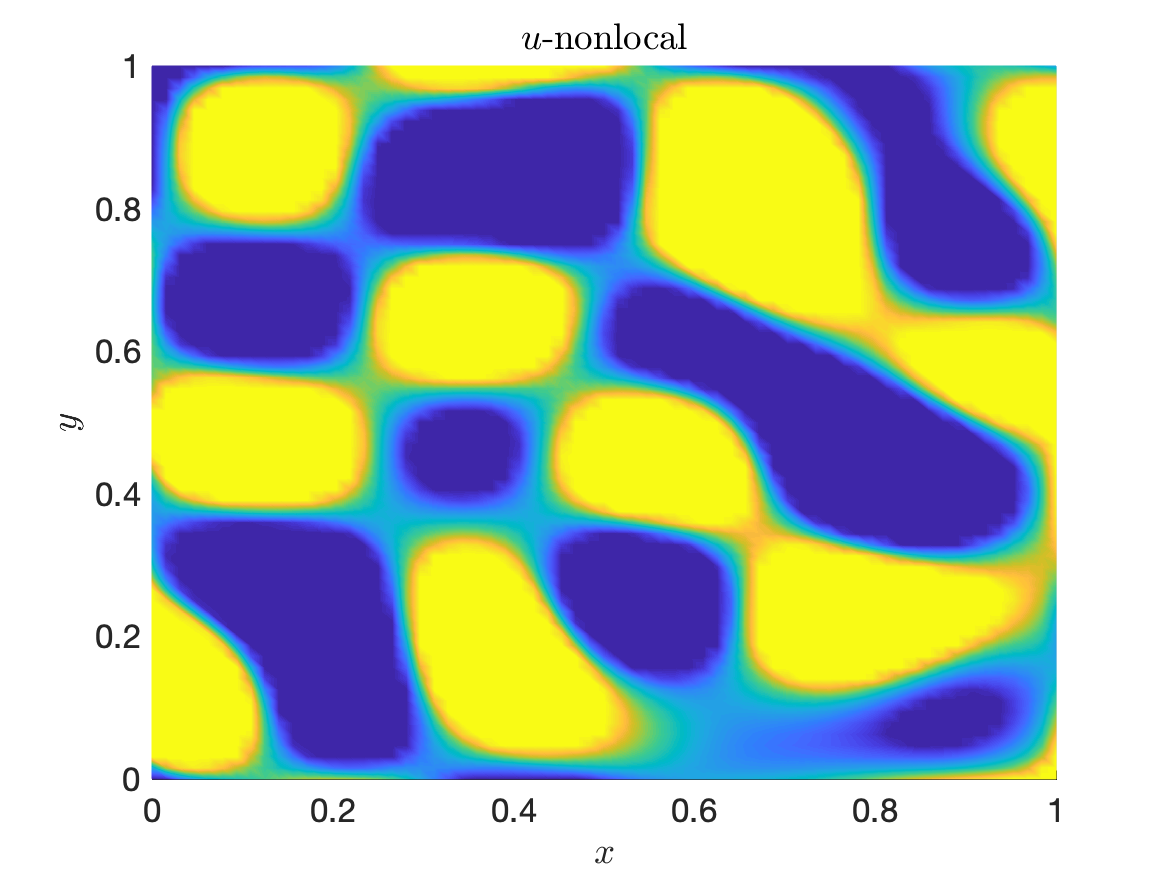}
  \caption*{$t=0.02$}
 \end{subfigure}
 \begin{subfigure}{.3\textwidth}
  \centering
 \includegraphics[width=\textwidth]{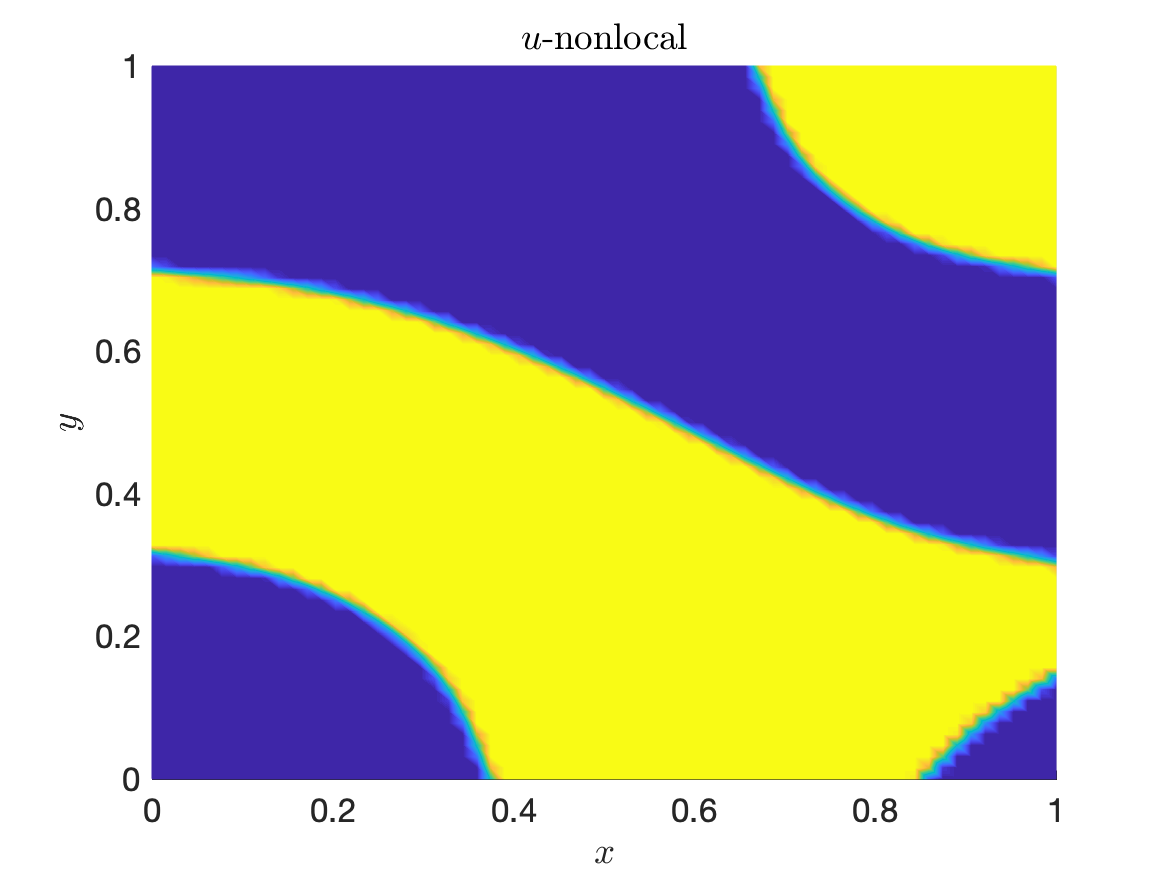}
  \caption*{$t=0.5$}
 \end{subfigure}
 \begin{subfigure}{.3\textwidth}
  \centering
 \includegraphics[width=\textwidth]{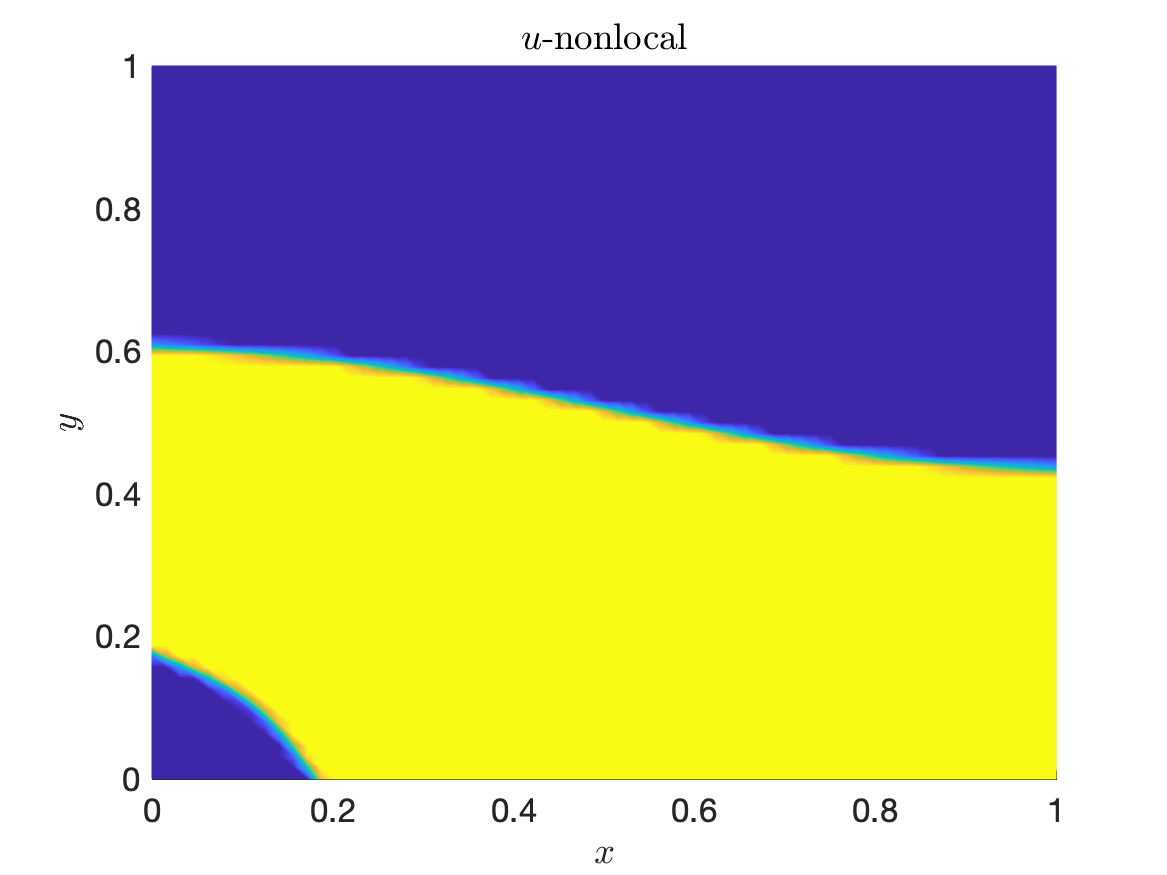}
  \caption*{$t=1$}
 \end{subfigure}
 \caption{Evolution of the nonlocal ({\it Case 2}) (bottom) and local (top) solutions of the Cahn-Hilliard variational inequality at different time instances and with $\cpot(\x)=0.9\varepsilon^2\cker(\x)$ for Example~3.}
 \label{fig:4}
\end{figure}

{In summary, we can conclude that the numerical results reported are in agreement with the theoretical developments. Furthermore, the numerical investigations also reveals how the choice of the boundary conditions, model, and kernel parameters affect the sharpness of the interface in the solution. Moreover, we have also demonstrated that the nonlocal solution with sharp interfaces can be stably simulated on a given mesh for arbitrary interface parameter $\xi\geq 0$ without tying  the discretization parameters to the interface width (including the case $\xi=0$, which corresponds to a sharp interface on an analytic level).   }

\section{Concluding remarks}\label{sec:conclusion}

In this work we have presented a nonlocal Cahn-Hilliard model that permits solutions to achieve only pure phases. We have performed a detailed analyses of the well-posedness of the problem as well as for the regularity of solutions. We have also provided an efficient discretization scheme, based on finite elements and implicit/explicit time-stepping schemes, that are used in several numerical experiments that illustrate the theoretical results. 

Further study of efficient discretization and approximation techniques for the model based on, e.g., explicit-implicit, cf.~\eqref{eq:explicit_bh}, or higher-order time-marching schemes, or even model order reduction approaches, are interesting questions {for future investigation}.

{We also note that local phase-field models are often used as a substitute of the sharp-interface models in applications such as, e.g., solidification. The obvious advantage of using a phase-field model is that it avoids explicit tracking of the interface. However, the interface width in the corresponding local discrete models is limited by the mesh discretization, which can lead to requiring excessively fine meshes in implementations and thus to a high computational cost. The nonlocal interface model, when compared to the local one in this context, could offer much greater flexibility and a promising alternative. Here, we do not need to compromise on the sharpness of the interface, which can be chosen independently of the mesh width, and allow for higher fidelity simulations on coarser grids. 
}

\end{document}